%% file: main.tex
\documentclass{amsbook}

\input{preamble.tex}

\begin{document}

\authornew{Benedikt Peterseim}
\geburtsdatum{08.04.1999}
\geburtsort{Haan}
\date{20.11.2023 \\(Revised version, 25.02.2024)}

\betreuer{Advisor: Prof. Dr. Eveliina Peltola}
\zweitgutachter{Second Advisor: Dr. Peter Kristel}

\institut{Institut f\"ur Angewandte Mathematik}
\title{On Monadic Vector-Valued Integration}
\ausarbeitungstyp{Master's Thesis  Mathematics}

\maketitle

\notocchapter{Acknowledgements}
\thispagestyle{empty}

First of all, I would like to thank my supervisors, Eveliina Peltola and Peter Kristel, for their invaluable support. 
For the countless hours that Peter devoted to guiding me during this project, I am deeply grateful. 
I would also like to thank David Aretz for interesting discussions and many helpful comments. 
Finally, I thank Tobias Fritz for numerous remarks and suggestions which helped to imporove this document.

\clearpage

\tableofcontents

\include{Chapters/1_Introduction}

\include{Chapters/2_k_Spaces_Linear_k_Spaces_and_Smith_Duality}

\include{Chapters/3_Monads_Tensor_Products_and_Duality_in_Closed_Symmetric_Monoidal_Categories}

\include{Chapters/4_Monadic_Measure_and_Integration_Theory}

\include{Chapters/5_More_On_Paired_Linear_k_Spaces}

\begin{appendix}
\include{Chapters/6_Appendix}
\end{appendix}
\bibliographystyle{alpha}
\bibliography{bibliography}

\end{document}

%% file: preamble.tex
\usepackage[english]{babel}

\usepackage[a4paper]{geometry}

\usepackage{amsmath}
\usepackage{quiver}
\usepackage{amssymb}
\usepackage{amsthm}
\usepackage[shortlabels]{enumitem}
\usepackage{graphicx}
\usepackage{mathtools}
\usepackage[colorlinks=true, allcolors=blue]{hyperref}
\usepackage[makeroom]{cancel}
\usepackage{stmaryrd}
\usepackage[capitalise]{cleveref}
\usepackage{booktabs}
\usepackage{tabularray}
\usepackage{microtype}
\usepackage{MA_Titlepage}
\usepackage[toc,page]{appendix}

\usepackage{verbatim}
\newcounter{comments}
\newenvironment{displaycomment}{\begin{list}{}{\rightmargin=1cm\leftmargin=1cm}\item\sf\begin{small}}{\end{small}\end{list}}
\def\comments{
\setcounter{comments}{1}

}

\comments

\makeatletter

\newcommand*\notocchapter[1]{%
  \if@openright\cleardoublepage\else\clearpage\fi
  \thispagestyle{empty}\global\@topnum\z@
  \@afterindenttrue
  \let\@secnumber\@empty
  \@makeschapterhead{#1}\@afterheading
}

\makeatother

\setlist{label*=\arabic*.}

\title{On Monadic Vector-Valued Integration (Draft)}

\newtheorem{theorem}{Theorem}[subsection]
\newtheorem{thm}[theorem]{Theorem}
\newtheorem{lem}[theorem]{Lemma}
\newtheorem{prop}[theorem]{Proposition}
\newtheorem{cor}[theorem]{Corollary}

\theoremstyle{definition}
\newtheorem{defn}[theorem]{Definition}
\newtheorem{construction}[theorem]{Construction}
\newtheorem{notation}[theorem]{Notation}
\newtheorem{example}[theorem]{Example}
\newtheorem{warning}[theorem]{Warning}
\newtheorem{problem}[theorem]{Problem}

\newtheorem{question}[theorem]{Question}

\newtheorem{remark}[theorem]{Remark}

\newcommand{\vspan}{\text{span}}
\newcommand{\nats}{\mathbb{N}}

\newcommand{\reals}{\mathbb{R}}
\newcommand{\complex}{\mathbb{C}}
\newcommand{\sets}{\mathsf{Set}}
\newcommand{\topsp}{\mathsf{Top}}
\newcommand{\ktop}{k\mathsf{Top}}
\newcommand{\spaces}{hk\mathsf{Top}}
\newcommand{\khaus}{\spaces}
\newcommand{\haus}{\mathsf{Haus}} 
\newcommand{\compacta}{\mathsf{CompHaus}}
\newcommand{\monoids}{\mathsf{Mon}}

\newcommand{\groups}{\mathsf{Grp}}
\newcommand{\rings}{\mathsf{Ring}}
\newcommand{\modules}{\mathsf{Mod}}

\newcommand{\vect}{{\mathsf{Vect}(\spaces)}}
\newcommand{\hvect}{{\mathsf{Vect}(\khaus)}}

\newcommand{\sclin}{\scompl \vect}
\newcommand{\lctvs}{\mathsf{LCTVS}}

\newcommand{\klctvs}{k\mathsf{LCTVS}}
\newcommand{\lklctvs}{lk\mathsf{LCTVS}}
\newcommand{\plin}{\mathsf{L}}
\DeclareMathOperator{\Hom}{Hom}

\newcommand{\cco}{C_{\mathsf{c.o.}}}
\newcommand{\op}{\text{op}}

\newcommand{\mnd}{\mathcal{M}_c}

\newcommand{\scompl}{r} 

\newcommand{\hquot}{h}

\newcommand{\ptimes}{\,\widehat{\otimes}\,}
 
\newcommand{\sepfun}{\mathsf{sep}} 

\newcommand{\id}{\mathsf{id}}

\newcommand*\diff{\mathop{}\!\mathrm{d}}
\DeclareMathOperator*{\colim}{colim}
\DeclareMathOperator*{\supp}{supp}
\DeclareMathOperator*{\im}{im}
\DeclareMathOperator*{\coim}{coim}

%% file: Chapters/1_Introduction.tex
\chapter{Introduction} 

In recent times, there has been a growing interest in a structuralist understanding of probability, measure and integration theory \cite{fritz2020synthetic,lucyshyn2013riesz}. The present thesis contributes to this programme in three ways. First, we construct a commutative probability monad on the cartesian closed category of $hk$-spaces (also known as CGWH spaces in the literature). Secondly, in order to achieve this in a seamless way, we develop the theory of \emph{paired linear $hk$-spaces}, a functional-analytic category tailored to the duality between measures and functionals. Finally, vector-valued integration emerges naturally from the free-forgetful adjunction between paired linear $hk$-spaces and $hk$-spaces, inducing a commutative monad of compactly supported measures and leading to a theory of \emph{monadic vector-valued integration}. A concise summary of the main results will be given in \cref{sec_main_results}, with further background and motivation provided in \cref{sec_background_overview}.

\section{Preliminaries and Conventions} 

Subsequently, we will assume some familiarity with the following subjects.

\begin{enumerate}[1.]
    \item \emph{General topology}. We will refer to the standard accounts \cite{engelking1989general, kelley2017general, willard2012general}. Less frequently occurring notions from general topology will be recalled when needed. We will not assume prior knowledge of cartesian closed categories of spaces.
    \item \emph{Functional analysis}. Basic knowledge of functional analysis, including the fundamentals of locally convex topological vector spaces and Fréchet spaces in particular, will be assumed. For a standard textbook account, see \cite{rudin1991functional}. 
    \item \emph{Category theory}. Although the very basics of category theory will be required, including a good understanding of universal properties, adjunctions, (co-)limits, and how these interrelate, we will not assume familiarity with monoidal categories or monads. We will employ the classical reference \cite{mac2013categories}, alongside the more recent \cite{riehl2017category}.
\end{enumerate}

Throughout this text, $\mathbb{K}$ will denote either the field of real or complex numbers and a ``vector space'' will always be (unless explicitly stated otherwise) a vector space over $\mathbb{K}$. By default, a ``measure'' will always be a countably additive $\mathbb{K}$-valued measure of bounded variation. In general, we will follow the measure-theoretic terminology of \cite{bogachev2007measure}.

\section{Main Results}\label{sec_main_results}

The main application of our findings is the construction of a commutative probability monad $\mathcal{P}$ on the cartesian closed category of $hk$-spaces (called ``CGWH spaces'' in \cite{strickland2009category} and ``$k$-Hausdorff $k$-spaces'' or ``compactly generated spaces'' in \cite{rezk2017compactly}; see \cref{defn_k_sp_hk_sp}). Background on cartesian closed categories of spaces and probability monads will be provided in \Cref{sec_prob_monads_on_cccs_of_spaces,sec_functs_and_cccs_of_spaces}, respectively. The monad $\mathcal{P}$ restricts to the classical \emph{Giry monad} on Polish spaces (i.e.~completely metrisable separable spaces), meaning that on a Polish space $X$, $\mathcal{P}(X)$ is the space of Borel probability measures with the topology of weak convergence of measures (``convergence in distribution''). The study of the probability monad $\mathcal{P}$ also leads to some interesting insights on the space $\mathcal{P}(\mathcal{S}'(\mathbb{R}^n))$ of probability measures on the space of tempered distributions, an object of great interest in the mathematics of quantum field theory (see \cref{sec_from_pol_to_cart_cl}). For example, the Fourier transform becomes a \emph{homeomorphism} between $\mathcal{P}(\mathcal{S}'(\mathbb{R}^n))$ and the space of continuous normalised positive-definite functions on $\mathcal{S}(\mathbb{R}^n)$ (\cref{thm_bochner_minlos_levy_fernique}). The construction of $\mathcal{P}$ along with its main properties will be obtained in Chapter 4, \cref{sec_prob_mon_P}. \par 
These results fit into a more holistic perspective, as follows. The space $\mathcal{P}(X)$ of probability measures will be constructed as a subspace of a space of $\mathbb{K}$-valued measures $\mathcal{M}(X)$. Similar to $\mathcal{P}$, this space of measures gives rise to a commutative monad on the category $\spaces$ of $hk$-spaces. Moreover, $\mathcal{M}(X)$ admits the structure of a \emph{paired linear $hk$-space}, which is a vector space with a compatible $hk$-space topology and additional ``paired linear structure'' (see \cref{defn_paired_linear_hk_space_}). Using the language of paired linear $hk$-spaces, a unifying picture emerges which is summarised by the following three theorems. 

\begin{thm}
    The category $\plin$ of paired linear $hk$-spaces admits a \emph{closed symmetric monoidal structure} consisting of a tensor product $\ptimes$ and an internal hom $[-,-]$. For paired linear $hk$-spaces $V,W$, the internal hom $[V,W]$ is the closed subspace of $C(V,W)$ consisting of morphisms $V\to W$ (where $C(V,W)$ is the space of continuous maps, as formed in the category of $hk$-spaces). The functors $\ptimes$ and $[-,-]$ determine each other by the tensor-hom adjunction, 
        $$ [V \;\ptimes W, Z] \cong [V, [W, Z]], \qquad (V,W,Z \in \plin)$$
    a natural isomorphism of paired linear $hk$-spaces. \par 
    Moreover, for every paired linear $hk$-space $V$, the natural map 
        $$ V \to V^{**}, \;\; x \mapsto (\phi \mapsto \phi(x)), $$
    is an isomorphism of paired linear $hk$-spaces, where $V^*:=[V,\mathbb{K}]$ denotes the \emph{$\plin$-dual}, i.e.~the space of morphisms $V\to \mathbb{K}$.  Furthermore, $\plin$ has all limits and colimits. In summary, $\plin$ is a \emph{bicomplete $*$-autonomous category}.
\end{thm}

This theorem will follow from a general construction, given in \cref{construction_thm_for_star_aut_cats_over_cccs} (Chapter 3, \cref{sec_closed_symm_mon_cats_over_cart_cl_cats}). 
The rich categorical structure of $\plin$ derives its worth only from the fact that $\plin$ is also rich in examples, as the next result demonstrates (see Chapter 3, \cref{Frechet_space_paired_linear_k_space_in_unique_way}).

\begin{thm}
    The category $\mathsf{Fre}$ of Fréchet spaces (with continuous linear maps as morphisms) as well as its dual category $\mathsf{Bra}$ of \emph{Brauner spaces} (which are certain locally convex topological vector spaces, see \cref{defn_brauner_space}) embed as full subcategories into $\plin$ and the self-duality of $\plin$ restricts to the duality between Fréchet and Brauner spaces (which we refer to as \emph{Smith duality}, see \cref{sec_intro_smith_duality}). Moreover, the tensor product of Fréchet spaces as taken in $\plin$ coincides with the completed projective tensor product of Fréchet spaces (see Chapter 5, \cref{chap_5_sec_tensor_prods}).
\end{thm}

Finally, and most importantly, the functional-analytic and measure-theoretic notions that we are interested in are naturally expressed in the language of paired linear $hk$-spaces:

\begin{thm}
    For every $hk$-space $X$, there is a \emph{free paired linear $hk$-space} $\mathcal{M}_c(X)$ on $X$. When $X$ is Hausdorff, the elements of $\mathcal{M}_c(X)$ can be identified with the compactly supported Radon measures on $X$ (Chapter 4, \cref{thm_riesz_type_repr_thm_Mc}), and its structure as a paired linear $hk$-space is uniquely determined by the ($\spaces$-enriched) adjunction,
        $$ C(X, V) \cong [\mathcal{M}_c(X), V], \;\; f \mapsto \left(\mu \mapsto \int_X f(x) \diff \mu(x) \right). \qquad (X \in \spaces, V\in \plin) $$
    In other words, every continuous map $X\to V$ can be uniquely extended to a morphism $\mathcal{M}_c(X)\to V$, and this extension is given by vector-valued integration (see Chapter 4, \cref{remark_vv_integral_from_free_forget_adj}). The resulting vector-valued integral coincides with the classical \emph{Pettis integral} whenever both are comparable (see Chapter 5, \cref{sec_comp_to_pettis}). As a particular case of the above isomorphism, the space of continuous functions $C(X)=C(X,\mathbb{K})$ and $\mathcal{M}_c(X)$ are mutually dual, for \emph{any} $hk$-space $X$, yielding the following commutative diagram of functors:
    \[\begin{tikzcd}
	{\mathsf{Fre}} && \plin \\
	&&& \spaces \\
	{\mathsf{Bra}^\op} && {\plin^\op}
	\arrow[hook, from=1-1, to=1-3]
	\arrow[hook, from=3-1, to=3-3]
	\arrow["{(-)^*}"', from=1-1, to=3-1]
	\arrow["{(-)^*}"', from=1-3, to=3-3]
	\arrow["{\mathcal{M}_c}"', curve={height=6pt}, from=2-4, to=1-3]
	\arrow["{C}", curve={height=-6pt}, from=2-4, to=3-3]
    \end{tikzcd}\]
    Finally, for \emph{any} $hk$-space $X$, the dual of the space of measures $\mathcal{M}(X)$ can be identified with the space $C_b(X)$ of bounded continuous functions on $X$ (see Chapter 4, \cref{sec_M_C_b_as_paired_lin_hk_sp,sec_k_reg_measures}),
        $$ \mathcal{M}(X)^* \cong C_b(X), \qquad C_b(X)^* \cong \mathcal{M}(X),$$
    which may be viewed as a general version of the Riesz representation theorem, bespoke for the setting of $hk$-spaces. 
\end{thm}

In conclusion, paired linear $hk$-spaces provide a functional-analytic category which is tailored to rendering spaces of measures and continuous functions as well-behaved as possible.

\section{Background and Motivation}\label{sec_background_overview}

\subsection{The analysis of ``functionals'' and cartesian closed categories of topological spaces}\label{sec_functs_and_cccs_of_spaces}

The subject of functional analysis bears the word ``functional'' in its very name. This term, or rather the french \emph{fonctionelle}, first appeared as a noun -- quite early on in the modern history of the subject -- in Fréchet's doctoral thesis \cite{frechet1905fonctions}, following a suggestion of Hadamard \cite{taylor1982study}. It therefore lies at the very roots of the field. Designating a function whose argument is a function, a ``functional'' is a particular kind of \emph{higher-order function}. The latter could be defined recursively as a function whose arguments and values may both be higher-order functions. Higher-order functions are modelled syntactically by the (simply typed) \emph{$\lambda$-calculus}, which (for our purposes) is nothing but a collection of rules on how higher-order functions ought to behave. These syntactic rules in turn receive semantic content through the notion of \emph{cartesian closed category}. \par 
Since higher-order functions lie at the very heart of functional analysis in the form of functionals, it should be considered natural for this theory to employ the notion of cartesian closed category from the very start. Similar arguments advocating the use of cartesian categories in the context of analysis have been given previously, as in the introduction of \cite{kriegl1997convenient}. Let us recall the definition of a cartesian closed category.

\begin{defn}
    A cartesian closed category is a category $\mathsf{C}$ with finite products such that for every object $X \in \mathsf{C}$, the functor $(-)\times X$ has a right adjoint $(-)^X$ (the ``exponential''), i.e.~we have a natural isomorphism,
        $$ \hom_{\mathsf{C}}(X\times Y, Z) \cong \hom_{\mathsf{C}}(Y, Z^X). \qquad (Y,Z \in \mathsf{C})$$
\end{defn}

In contrast to other approaches to functional analysis that embrace the concept of a cartesian closed category (e.g.~the bornological approach employed in \cite{meyer1999analytic}, for instance, in which \emph{boundedness} is the primary notion instead of continuity), we shall not deviate from the conceptual formula,
\begin{equation}\label{eqn_linalg_plus_top_eq_funcana}
    \text{linear algebra} \:+\: \text{topology} \;=\; \text{(linear) functional analysis},
\end{equation}
that has put topological vector spaces at the basis of many treatments of functional analysis. Rather, we would like to subtly change, or restrict, what is meant by ``topology'' in \cref{eqn_linalg_plus_top_eq_funcana} by passing to a cartesian closed closed category of topological spaces. (A more radical such change is proposed by the more recent program of \emph{condensed mathematics} \cite{scholze2019condensed}, the precise relation to which we give in  \cref{appendix_relation_to_condesed}.)  
The most common cartesian closed category of topological spaces is the one spanned by \emph{$hk$-spaces} (also known \emph{compactly generated weak Hausdorff spaces}, or \emph{CGWH spaces}).

\begin{defn}[$k$-space, $hk$-space]\label{defn_k_sp_hk_sp}
    A map $f: X \to Y$ between topological spaces $X,Y$ is \emph{$k$-continuous} if for every compact Hausdorff space $K$ and every continuous map $p: K\to X$, the composite $f \circ p: K \to Y$ is continuous. A topological space $X$ is a \emph{$k$-space} if every $k$-continuous map from $X$ is continuous. \par
    A $k$-space $X$ is \emph{$k$-Hausdorff}, or an \emph{$hk$-space}, if the diagonal,
        $$ (=_X) = \{(x, y) \in X\times X\mid x = y\} \subseteq X\times X, $$
    is closed in $X\times X$, where $\times$ denotes the product as formed in the category of $k$-spaces (with continuous maps as morphisms). We write $\spaces$ for the category of $hk$-spaces.
\end{defn}

\begin{warning}
    When $X$ and $Y$ are $k$-spaces, the notation $X\times Y$ will always refer to the product as formed in the category of $k$-spaces (or, equivalently, in $\spaces$), and not to the product in the category of topological spaces, which we will instead denote by $\times_{\topsp}$.
\end{warning}

We will review the basic theory of $k$-spaces in Chapter 2. \par 
A less well-known cartesian closed category of spaces is the category of QCB spaces, which first arose in the context of \emph{domain theory} and \emph{computable analysis} \cite{simpson2003towards,schroder2021admissibly}.

\begin{defn}
    A QCB space (short for ``quotient of countably based space'') is a topological space that is the quotient of some second countable topological space, i.e.~a topological space $X$ is a QCB space if there exists a second countable space $Y$ and topological quotient map $p:Y\to X$.
\end{defn}

Every QCB space is a $k$-space and the inclusion functor from the category of QCB spaces into that of $k$-spaces preserves exponentials (i.e.~function spaces) as well as countable limits and colimits (see Chapter 2 for further details). Because of their strong countability properties (e.g.~every QCB space is separable), they seem particularly attractive in an analytic or measure theoretic context. 

\subsection{Probability and measure monads}\label{sec_prob_monads_on_cccs_of_spaces} Since their introduction by Giry \cite{giry1981categorical}, probability monads have found numerous applications in areas such as \emph{non-parametric Bayesian statistics} \cite{jost2021probabilistic}, \emph{probabilistic programming} \cite{heunen2017convenient}, as well as the abstract study of probability and non-determinism via so-called \emph{Markov categories} \cite{fritz2020synthetic}. In addition, the observation that spaces of probability measures give rise to monads (which we will define shortly) also provides important insights into the nature of probability itself. Elementary introductions to the concept of a probability monad can be found in \cite{perrone2018categorical} and \cite{asilis2020probability}. We will instead give an explanation using one of Giry's original examples, which shows the usefulness of this idea in a more analytical context (where one is interested in convergence of probability measures, for example). \par 
Let us begin by giving a (shortened) definition of the general notion of monad in category theory.

\begin{defn}[Monad]
    Let $\mathsf{C}$ be a category. A \emph{monad} $(T, i, m)$ on $\mathsf{C}$ is an endofunctor, 
            $$T:\mathsf{C}\to \mathsf{C}, $$ 
    together with two natural transformations, the \emph{unit},
            $$ i: X \to T(X), \qquad (X \in \mathsf{C})$$
    and the \emph{multiplication},
            $$ m: T(T(X)) \to T(X), \qquad (X \in \mathsf{C}) $$
    such that two diagrams, expressing a certain coherence between $i$ and $m$, commute (for the full definition, see Chapter 3, \cref{monad_defn}). 
\end{defn} 

With this definition, a \emph{measure monad} is a monad $M$ on some category of ``spaces'' (such as the category $\sets$ of sets, the category of measurable spaces or the category of compact Hausdorff spaces) such that for each such space $X$, the space $M(X)$ is a space of measures on $X$ (such as finitely supported measures defined on arbitrary subsets, arbitrary probability measures defined on measurable subsets, or Radon measures defined on Borel sets). A \emph{probability} monad is a measure monad for which these measures are probability measures. Of course, this is not a formal definition -- what is meant by ``space'' or ``measure'' is free for interpretation. \par 
One possible conceptualisation of probability monads is that they formalise what a ``coherent notion of (spaces of) probability measure(s)'' is. The following example illustrates this, illuminating also an important aspect of why Polish spaces (i.e.~separable, completely metrisable topological spaces) make for such a good probability theoretic setting.

\begin{example}[The Giry monad on Polish spaces, \cite{giry1981categorical}]\label{ex_giry_monad}
    Let $\mathsf{Pol}$ be the category of Polish spaces (with continuous maps as morphisms). For a Polish space $X$, let $\mathcal{P}(X)$ be the space of Borel probability measures on $X$ equipped with the topology of weak convergence of measures (see Chapter 4, \cref{defn_weak_convergence_of_measures}). Then $\mathcal{P}(X)$ is again a Polish space and the maps 
        $$f_*:\mathcal{P}(X) \to \mathcal{P}(Y),\;\;\; \mu \mapsto f_*\mu, \qquad (f: X\to Y, \:Y\in \mathsf{Pol})$$
        $$\delta_\bullet: X \to \mathcal{P}_r(X) \;\;\; x \mapsto \delta_x, $$
        $$ \mathbb{E}: \mathcal{P}(\mathcal{P}(X))\to \mathcal{P}(X), \;\; \mathbb{E}(\pi)(A) := \int_{\mathcal{P}(X)} \mu(A) \diff \pi(\mu), \qquad(A \in \mathcal{B}(X)) $$
    are well-defined, continuous and natural in $X$. Here, $\delta_x$ denotes the Dirac measure concentrated at $x$, $f_*\mu$ is the pushforward of $\mu$ under the continuous map $f$ and $\mathcal{B}(X)$ is the Borel $\sigma$-algebra on $X$. 
    The endofunctor,
        $$ \mathcal{P}: \mathsf{Pol} \to \mathsf{Pol}, \;\; X\mapsto \mathcal{P}(X), \; f\mapsto f_*, $$
    forms a monad with $\delta_\bullet$ and $\mathbb{E}$ as the unit and multiplication. 
\end{example}

In probabilistic terms, the map $\delta_\bullet$ corresponds to embedding determinism into probability, while $\mathbb{E}$ is to be interpreted as the expectation of (the law of) a random probability measure, flattening ``higher order probability'' down to simple probability. That these fundamental operations (including the pushforward of probability measures) are continuous and cohere as we would expect is exactly the statement that $(\mathcal{P}(X), \delta_\bullet, \mathbb{E})$ is a monad on $\mathsf{Pol}$.

\subsection{From Polish spaces to cartesian closed categories of spaces}\label{sec_from_pol_to_cart_cl} Many spaces of interest in probability-theoretic applications are not Polish spaces (and do not admit any useful Polish topology). One example of particular importance is given by spaces of distributions such as the space $\mathcal{S}'(\mathbb{R}^n)$ of tempered distributions, or the space $\mathcal{D}'(\mathbb{R}^n)$ (which is given as the dual of the space of compactly supported smooth functions). For instance, \emph{white noise} (the distributional derivative of Brownian motion) can be understood as a probability measure on $\mathcal{S}'(\mathbb{R})$. Moreover, the Osterwalder-Schrader axioms of Euclidean quantum field theory \cite[p.~91]{glimm1981quantum} describe a certain class of probability measures on $\mathcal{D}'(\mathbb{R}^n)$, by which one might say that Euclidean quantum field theories ``are'' probability measures on a space of distributions. Therefore, the significance of a suitable space of probability measures on non-metrisable spaces like $\mathcal{S}'(\mathbb{R}^n)$ and $\mathcal{D}'(\mathbb{R}^n)$ cannot be dismissed. \par
A crucial observation is now that nearly all spaces appearing in applications can be obtained from a Polish space (or at least a metric space) by formation of either
\begin{enumerate}
    \item countable limits or colimits or,
    \item spaces of continuous maps.
\end{enumerate}
Let us demonstrate this principle using the space $\mathcal{D}'(\mathbb{R}^n)$ as an example. For each compact ball $B_k$ of radius $k\in \mathbb{N}$, the space of smooth functions with support in $B_k$ is a separable Fréchet space, hence a Polish space. Now, the space of test functions $\mathcal{D}(\mathbb{R}^n)$ is given as the filtered colimit,
    $$ \mathcal{D}(\mathbb{R}^n) := \colim_{k\in \mathbb{N}} \mathcal{D}(B_k). $$
(This filtered colimit is usually taken in the category of locally convex spaces, but there is room for variation in this.) We then pass to the space of continuous maps $C(\mathcal{D}(\mathbb{R}^n), \mathbb{K}$) and finally to the subspace of linear maps $\mathcal{D}'(\mathbb{R}^n)$ in this space. The condition of linearity is given by an equation, and hence this subspace is an equaliser, a particular kind of (finite) limit (see Chapter 2, \cref{L_VW_QCB}). Thus, the space of distributions can be obtained by forming a countable colimit, followed by a mapping space and a finite (in particular, countable) limit. These operations are sufficient to construct a very rich variety of examples. \par 
We are thus led to considering a category closed under these operations, i.e.~a countably complete and cocomplete cartesian closed category, such as the category $\spaces$ of $hk$-spaces or the category $h\mathsf{QCB}$ of $k$-Hausdorff QCB spaces. In addition to the coherent notion of ``space'' that these base categories provide, we would like an equally coherent notion of probability measure for this setting. In other words, we face the following problem, whose solution is the main application of our methods (see Chapter 4).

\begin{problem}\label{problem_measure_coherence}
    Construct a probability monad on $\spaces$ that restricts to a probability monad on $h\mathsf{QCB}$ and further to the Giry monad (\cref{ex_giry_monad}) on Polish spaces.
\end{problem}

In addition, we would like to construct monads of $\mathbb{K}$-valued measures that admit a functional analytic interpretation via a suitable Riesz representation theorem. The classical setting of functional analysis being based on topological vector spaces places serious obstrutions on a straightforward solution of these problems (since the definition of a topological vector space employs the ``incorrect'' $\topsp$-product, see \cref{warning_lin_hk_sp_vs_top_vect_sp_intro}). These obstructions invite us to think about a notion of topologised linear space adapted to the setting of $hk$-spaces and QCB spaces, which will allow us to solve \cref{problem_measure_coherence} and related problems in a unified manner.

\subsection{Some foundational problems of the notion of topological vector space}\label{sec_problems_with_top_vec_sp} A further motivation for looking beyond topological vector spaces is that they exhibit some fundamental shortcomings as a category. These have been lamented many times (e.g.~in the already mentioned \cite{kriegl1997convenient} and \cite{scholze2019condensed}) and can be considered part of ``mathematical folklore'', but a concise summary seems to be missing from the literature. Let us sketch some of the main points here.

\subsubsection{Unavoidable Discontinuity}\label{sec_continuity_problems} The following problem is pointed out in \cite[p.~2]{kriegl1997convenient} to illustrate difficulties with using locally convex topologies in the context of infinite-dimensional calculus. It is a rather severe problem: the evaluation pairing between a locally convex space $V$ and its dual $V'$, equipped with \emph{any} locally convex topology, \emph{cannot} be continuous unless $V$ is normable (see Chapter 2, \cref{prop_ex_cont_pairing}). 

\emph{Why is the unavoidable discontinuity of evaluation maps such a grave problem?} One area of application suffering severely from this pathology is calculus in locally convex topological vector spaces (LCTVS). This is one of the motivations for the ``convenient vector spaces'' of \cite{kriegl1997convenient}, which also contains a detailed historical account of infinite-dimensional calculus with references to further problems in the same spirit \cite[pp.~73-77]{kriegl1997convenient}. For example, they mention 25 inequivalent definitions of the derivative at a single point of a topological vector space and explain the importance of having a cartesian closed setting for smooth maps; see \cite[pp.~1-2]{kriegl1997convenient} for the relation to the problem of unavoidably discontinuous evaluation maps. \par 
On the side of measure and integration theory, which is our concern, the unavoidable discontinuity of evaluation maps leads to the unpleasant situation that there does not exist \emph{any} LCTVS topology on the space of compactly supported measures $\mathcal{M}_c(X)$ on a non-compact space $X$ with $C(X)$ as its dual that makes the integration pairing, 
$$ \int: \mathcal{M}_c(X) \times_{\topsp} C(X) \to \mathbb{K}, \;\; (f, \mu) \mapsto \mu(f) = \left(\int f(x) \diff \mu(x) \right), $$
continuous with respect to the $\topsp$-product $\times_{\topsp}$.  

\subsubsection{Tensor troubles} The category of locally convex Hausdorff topological vector spaces does not admit a good notion of tensor product, in the following sense. Recall that the \emph{projective tensor product} $\otimes_\pi$ of locally convex topological vector spaces can be characterised by the following universal property, analogous to the algebraic case (see \cite[p.~33]{kriegl2016frechet}). If $b:V\times_{\topsp} W \to Z$ is a continuous bilinear map, then there is a unique extension $\tilde{b}: V\otimes_\pi W \to Z$ of $b$ along the inclusion $V\times W \to V\otimes_\pi W$ (given by mapping a pair $(x,y)$ to the simple tensor $x\otimes y$). While this universal property is the one we expect, the resulting tensor product $\otimes_\pi$ does \emph{not} satisfy the equally expected \emph{tensor-hom adjunction}, which would turn the category $\lctvs$ of locally convex Hausdorff topological vector spaces into a \emph{closed monoidal category} (\cref{defn_closed_symmetric_mon_cat}). This means that there does not exist a ``space of continuous linear maps'' functor
    $$ L(-,-): \lctvs \times \lctvs \to \lctvs,  $$
such that for all locally convex topological vector spaces $V,W,Z$,
    $$ \Phi^{V,W,Z}: L(V, L(W,Z)) \to L(V \otimes_\pi W, Z), \;\; f \mapsto (x \otimes y \mapsto f(x)(y)) $$
is a bijection. This is because if we did have such functor $L(-,-)$ and adjunction, then the evaluation map $\mathrm{ev}: V \times_{\topsp} L(V, \mathbb{K})\to \mathbb{K}$ would be continuous, since
    $$ \Phi^{L(V,\mathbb{K}), V, \mathbb{K}}(\id_{L(V, \mathbb{K})})(\phi \otimes x) = \phi(x) = \mathrm{ev}(x, \phi). $$
But we have seen that this cannot possibly hold for all locally convex spaces. \par 
Since discontinuous bilinear maps become ubiquitous in the setting of non-normable locally convex topological vector spaces, one is lead to considering variations of the projective tensor product, such as the \emph{inductive tensor product}. (This tensor product satisfies the necessary universal property with respect to \emph{separately continuous} bilinear maps, instead of only ($\times_{\topsp}$-)continuous bilinear maps.) This bifurcation of concepts -- concepts that should really be determined uniquely by their expected properties -- is analogous to the situation of the many inequivalent definitions of the derivative in topological vector spaces mentioned before. \par 
\emph{Why is it important to have a good notion of tensor product?} The primary answer to this question is the need for such tensor products in various applications. A particularly prominent such application is distribution theory (see, for example, \cite{treves2016topological}). Topological tensor products are also built into Segal's definition of a conformal field theory \cite{segal1988definition} (and other flavours of \emph{functorial quantum field theory}), and they arise naturally in the study of topological algebras and non-commutative geometry (see, for example, \cite{meyer1999analytic}, where a ``convenient'' bornological approach is employed). \par 
On a more basic level (and most importantly), having a closed monoidal structure on a category of topologised vector spaces means that the topology on spaces of continuous linear maps $L(-,-)$ is uniquely defined by a universal property, and behaves as we would expect it to.

\section{Overview}

\subsection{Linear $hk$-spaces and linear QCB spaces}

Mac Lane's classical ``Categories for the Working Mathematician'' contains the phrase: ``All told, this suggests that in $\topsp$ we have been studying the wrong mathematical objects. The right ones are [$k$-spaces].'' \cite[p.~188]{mac2013categories} Already Brown's original article anticipates ``that the category of $k$-spaces [may be] adequate and convenient for all purposes of topology.'' \cite[p.~304]{brown1963ten} Functional analysis is certainly one such purpose, and higher order functions (``functions of functions'' -- the \emph{raison d'être} of cartesian closed categories such as $\ktop$) enter its very name in the form of \emph{functionals}. This suggests replacing the definition of a topological vector space, i.e.~a $\mathbb{K}$-module object in $\topsp$ (in category-theoretic terminology), by the notion of \emph{linear $k$-space}, which are instead $\mathbb{K}$-module objects in $\ktop$.

\begin{defn}
    A \emph{linear $k$-space} (over the real or complex numbers $\mathbb{K}\in \{\reals, \complex\}$) is a $k$-space $V$ together with two continuous maps (addition and scalar multiplication),
        $$ +: V \times V \to V,$$
        $$ \cdot: \mathbb{K}\times V \to V,$$
    such that $V$ forms a vector space with respect to the operations $\cdot$ and $+$. (Here, the product $V \times V = k(V\times_\topsp V)$ denotes the product in $\ktop$, as always!) A \emph{linear $hk$-space} is a linear $k$-space which is also $k$-Hausdorff, and a \emph{linear QCB space} is a linear $k$-space which is also a QCB space. 
\end{defn}

\begin{remark}
    This definition is not new -- after all, it is in some sense already contained in the definition of a $k$-space. Instead, given Mac Lane's assessment, it should be surprising that linear $k$-spaces have not received more attention. Hausdorff linear $k$-spaces are called ``Lineare $k$-R\"aume'' by Frölicher and Jarchow in \cite{frolicher1972dualitatstheorie} and ``compactly generated vector spaces'' by Seip in \cite{seip1979convenient}. In the context of functional analysis, the term ``compactly generated'' also has the meaning of being generated algebraically by a compact subset, which is why we prefer the term ``linear $k$-space''. 
\end{remark}

Despite the status of $k$-spaces as a ``more correct'' notion of topological space, not much work has been done to understand \emph{linear} $k$-spaces (in contrast to the vast body of literature on topological vector spaces). The present work provides a step towards closing this gap.

\begin{warning}\label{warning_lin_hk_sp_vs_top_vect_sp_intro}
    Every $k$-space is a topological space, by definition, and the definition of linear $k$-spaces is superficially almost identical to that of a topological vector space. \emph{However,} a linear $k$-space is \emph{not} necessarily a topological vector space with respect to the same topology and vector space structure, see \cref{ex_lin_k_space_not_top_vec_sp}. 
    In the other direction, a topological vector space is also not necessarily a linear $k$-space (with respect to the same topology and vector space structure), see \cref{top_vec_spaces_not_lin_k_spaces}. What does hold is that the so-called $k$-ification of a topological vector space is always a linear $k$-space (see \cref{k_ification_lin_k_sp}). 
\end{warning}

\begin{remark}
    Linear $hk$-spaces can also be interpreted in terms of Clausen-Scholze's \emph{condensed mathematics} (or the essentially equivalent \emph{pyknotic sets} of \cite{barwick2019pyknotic}). More specifically, they can be viewed as particular kinds of \emph{condensed vector spaces} \cite[Proposition 1.7]{scholze2019condensed}. A special case of condensed vector spaces are the \emph{quasi-separated} ones which are equivalent to Waelbroeck's \emph{compactological spaces}, see \cite{lucien1971compactological} for the definition of compactological spaces and \cite[Proposition 1.2]{clausen2022condensed} for the equivalence. See \cref{appendix_relation_to_condesed} for further details.
\end{remark}

\begin{remark}
    Linear QCB spaces (specifically, \emph{$k$-locally convex} linear QCB spaces, in the terminology of Chapter 5) were studied by Schröder in \cite{schroder2016towards} in the context of \emph{computable functional analysis}, illustrating the constructive nature of these objects.
\end{remark}

Numerous examples of linear $hk$-spaces and linear QCB spaces will be given in Chapter 2, \cref{ex_lin_ksp}. Most importantly for now, every Fréchet space is a linear $hk$-space and every separable Fréchet space is a linear QCB space. \par 
With linear $hk$-spaces, the problems described in \cref{sec_problems_with_top_vec_sp} vanish at once. For any two linear $hk$-spaces $V,W$ there is a natural linear $hk$-space $L(V,W)$ of continuous linear maps $V\to W$, as well as a tensor product $V \otimes W$, satisfying the expected universal property and the tensor hom adjunction, 
    $$ L(V \otimes W, Z) \cong L(V, L(W,Z)). \qquad (V,W,Z\in \vect) $$
In other words, the category of linear $hk$-spaces forms a closed symmetric monoidal category with respect to $L(-,-)$ and $\otimes$ (see Chapter 3, \cref{ex_vect_closed_symm_mon}). The topology on $L(V,W)$ is given by the subspace topology that it inherits from the space of all continuous maps $W^V = C(V,W)$, as formed according to the cartesian closed structure of $\spaces$. This topology is not only natural, but it also tends to be well-behaved analytically. \par 
For example, when $H$ is a separable Hilbert space, it follows that the space $L(H, H)$ of bounded operators on $H$ carries the \emph{sequential strong topology} (see \cref{convergence_in_LVW,lem_LVW_seq}),~i.e. a sequence is convergent in $L(H,H)$ if, and only if, it is strongly convergent (i.e.~converges pointwise), and a subset $C\subseteq L(H,H)$ is closed if, and only if, for every strongly convergent sequence of operators in $C$, its limit is again in $C$. For many applications, such as the representations of Lie groups on $H$ appearing in quantum physics (which are usually required to be strongly continuous), this topology is therefore an appropriate one.

\subsection{Replete linear $hk$-spaces}

In addition to a version of the notion of topological vector space adapted to the setting of $k$-spaces, we would also like properties analogous to ``locally convex'' and ``complete''. One could define these in a reasonably direct manner, which we will do in Chapter 5 under the names of ``$k$-locally convex'' and ``$k$-complete''. For the purposes of an introduction, we will only give the following simpler definition here, which capture $k$-local convexity and $k$-completeness at the same time:

\begin{defn}\label{defn_repl_lin_hk_spaces}
    A replete linear $hk$-space is a linear $hk$-space $V$ such that the canonical map 
        $$ \eta: V \to (V^\wedge)^\wedge, \;\; x \mapsto (\phi \mapsto \phi(x)), $$
    is a closed embedding, where for a linear $hk$-space $W$,
        $$ W^\wedge := L(W, \mathbb{K}), $$
    denotes the \emph{natural dual} of $W$.  
\end{defn}

The idea behind this definition is the following. If $V$ is embedded as a closed subset in its double dual, then the dual $V^\wedge$ separates points (a local convexity-type property) and $V$ will inherit completeness properties (such as sequential completeness, see Chapter 5, \cref{prop_k_compl_implies_seq_compl,replete_implies_k_complete}) from $V^{\wedge\wedge}$. Further justification for the claim that repleteness is a simultaneous local convexity and completeness condition will be provided in Chapter 5, where we will see that the category of replete linear $hk$-spaces is, in fact, equivalent to a full subcategory of locally convex Hausdorff topological vector spaces consisting of \emph{$k$-complete} and \emph{compactly determined} spaces. Compactly determined spaces were introduced by Porta in \cite{porta1972compactly}. Part of the equivalence between replete linear $hk$-spaces and $k$-complete, compactly determined locally convex Hausdorff topological vector spaces was established already by Frölicher and Jarchow in \cite{frolicher1972dualitatstheorie}, see Chapter 5 for further details. \par 
Replete linear $hk$-spaces form a very wide class of spaces. For example, every Fréchet space is replete. This will follow directly from the phenomenon of \emph{Smith duality} which we will come to shortly in \cref{sec_intro_smith_duality}. In addition, the category of replete linear $hk$-spaces is closed under various constructions, and this follows from general category-theoretic developments. \par  
Namely, \cref{defn_repl_lin_hk_spaces} is a special case of a general procedure to obtain certain reflective subcategories (exponential ideals, to be precise) of closed symmetric monoidal categories. These subcategories consist of so-called \emph{separated objects} which we will introduce in Chapter 3, \cref{sec_triple_dualisation_lem_and_sep_obj}. This general method guarantees \emph{a priori} that given a linear $hk$-space $V$ and a replete linear $hk$-space $W$, the space of continuous linear maps $L(V,W)$ is again replete. Moreover, one can turn any linear $hk$-space $V$ into a replete one by means of forming the \emph{repletion} $\scompl V$. There is a ``repleted'' tensor product $\otimes_r = \scompl (- \otimes -)$ that turns the category $\sclin$ into a closed symmetric monoidal category with respect to $L(-,-)$ and $\otimes_r$. As with general linear $hk$-spaces, the problems with locally convex topological vector spaces described in \cref{sec_problems_with_top_vec_sp} do not occur for replete linear $hk$-spaces. \par 
Given any $hk$-space $X$, one can form the \emph{free replete linear $hk$-space} $\mathcal{M}_c(X)$ on $X$. As an $hk$-space, this is the same space as the free \emph{paired} linear $hk$-space $\mathcal{M}_c(X)$ mentioned in \cref{sec_main_results} (we will introduce paired linear $hk$-spaces shortly in \cref{sec_intro_plin_hk_sp}). When $X$ is a Hausdorff space $\mathcal{M}_c(X)$ can be identified with the space of compactly supported Radon measures on $X$ (see Chapter 4, \cref{sec_free_repl_lin_is_spc_of_measures}). It satisfies the following universal property (see Chapter 3, \cref{cor_univ_prop_free_repl_sp}). For every replete linear $hk$-space $V$ and any continuous map $f: X\to V$, there exists a unique continuous linear map $\tilde{f}: \mathcal{M}_c(X) \to V$ making the diagram 
\[\begin{tikzcd}
	{\mathcal{M}_c(X)} & V \\
	X
	\arrow["{\tilde{f}}", dashed, from=1-1, to=1-2]
	\arrow["{\delta_\bullet}", from=2-1, to=1-1]
	\arrow[from=2-1, to=1-2]
\end{tikzcd}\]
commute, where $\delta_\bullet$ is the map that assigns to a point $x\in X$ the Dirac measure $\delta_x$ at $x$. In fact, under the identification of $\mathcal{M}_c(X)$ with a space of measures (see Chapter 4), the unique continuous linear map $\tilde{f}$ is given by integrating the (vector-valued!) function $f$ against measures $\mu \in \mathcal{M}_c(X)$:
    $$ \tilde{f}(\mu) = \int_X f(x) \diff \mu(x). $$
Thus, the vector-valued integral arises naturally from the free-forgetful adjunction between the category $r\vect$ of replete linear $hk$-spaces and $\spaces$. 

\begin{remark}
    From the above universal property, it follows directly that under the inclusion of the category of linear $hk$-spaces into that of condensed vector spaces, every replete linear $hk$-space is \emph{$\mathcal{M}$-complete} in the sense of \cite[Definition 4.1]{sclausen2020lectures}. 
\end{remark} 

\subsection{Smith duality}\label{sec_intro_smith_duality} The following phenomenon is very natural from the perspective of $k$-spaces, but may come as quite a surprise to those used to the standard notion of reflexivity of locally convex topological vector spaces: \emph{taking duals in $\mathsf{Vect}(\spaces)$, every Fréchet space is reflexive.} This was first observed in the case of Banach spaces by Smith \cite{smith1952pontrjagin} without using the language of $k$-spaces, but instead interpreting the result as an extension of Pontryagin duality (of locally compact abelian groups) to Banach spaces. A precise formulation is as follows.

\begin{theorem}[Smith Duality]
    Let $V$ be a Fréchet space. Then the canonical evaluation map 
        $$V \to (V^\wedge)^\wedge, \;\;\; x\mapsto (\phi \mapsto \phi(x)), $$
    is an isomorphism (in $\mathsf{Vect}(\spaces)$, i.e.~a linear homeomorphism).
\end{theorem}

 A more refined version uses the notion of \emph{Brauner space}, the notion dual to that of a Fréchet space under Smith duality. 

\begin{defn}\label{defn_brauner_space}
     A linear $hk$-space is a \emph{Brauner space} if it is hemicompact and also a complete locally convex topological vector space (with respect to the same topology). Recall that an $hk$-space is \emph{hemicompact} if it is the sequential colimit of compact spaces with injective transition maps. We write $\mathsf{Bra}$ for the category of Brauner spaces with continuous linear maps as morphisms.
\end{defn}

The term ``Brauner space'' was coined by Akbarov \cite[p.~220]{akbarov2003pontryagin} in honour of Brauner's investigation of this class of spaces in \cite{brauner1973duals}.

\begin{remark}
    Hemicompactness of a general topological space $X$ is usually defined as follows: $X$ is hemicompact if it admits an increasing sequence $(K_n)$ of compact subsets such that every further compact set is contained in one of the $K_n$. When $X$ is an $hk$-space, this reduces to the simpler definition given above by \cite[Lemma 3.6, Lemma 3.7]{strickland2009category}.
\end{remark}

We can now state:

\begin{thm}
    If $V$ is a Fréchet space, then $V^\wedge$ is a Brauner space and, conversely, if $V$ is a Brauner space, then $V^\wedge$ is a Fréchet space. Therefore, Smith duality yields an equivalence of categories, 
        $$ (-)^\wedge: \mathsf{Fre}^\op \to \mathsf{Bra}. $$
\end{thm}

We will review a proof of this statement in \cref{sec_smithduality} of Chapter 2. 

\begin{remark}
    Smith duality also admits a natural interpretation in terms of \emph{condensed mathematics}, see \cite[Theorem 4.7]{sclausen2020lectures}.
\end{remark}

The only linear $k$-spaces that are \emph{both} Fréchet \emph{and} Brauner spaces are the finite-dimensional ones (by the Baire category theorem). This expresses a certain dichotomy between spaces of functions (``extensive quantities'') and spaces of measures or distributions (``intensive quantities''), in the following sense. Many (if not most) such spaces (of functions, measures or distributions) appearing in applications of functional analysis are either Fréchet spaces or duals of Fréchet spaces: function spaces (e.g.~$C([0,1])$, $\mathcal{S}(\mathbb{R}^n)$, ...) tend to be Fréchet spaces, while spaces of distributions or measures (e.g. $\mathcal{M}([0,1])$, $\mathcal{S}'(\mathbb{R}^n)$, ...) are their duals. This raises a question: can we find a version of Smith duality applying to both Fréchet and Brauner spaces at the same time, putting function spaces and spaces of measures or distributions on equal grounds? Can we find a closed symmetric monoidal category containing all Brauner spaces and all Fréchet spaces in which \emph{every} space is reflexive, in a way that generalises Smith duality? \par
As a direct consequence of Smith duality, every Fréchet space as well as every Brauner space is replete. Hence, the first natural question in this direction might be:

\begin{question}\label{question_reflexivity}
    Is \emph{every} replete linear $hk$-space reflexive in the sense that the canonical map $V \to V^{\wedge\wedge}$ is an isomorphism? (Remember that for a linear $hk$-space $W$, $W^\wedge$ denotes the \emph{natural dual}, see \cref{defn_natural_dual}.) What about the case when $V$ is a QCB space?
\end{question}

The answer to the first part of the question is negative. In \cite{haydon1972probleme}, an example was found of an $hk$-space $X$ for which the locally convex space $lC(X)$ associated to the linear $hk$-space $C(X)$ (see Chapter 5, \cref{defn_ksat_l}) is not complete, implying that $C(X)$ is not reflexive in the above sense by \cite[Theorem 4.4 (5)]{frolicher1972dualitatstheorie}. The question whether every replete linear QCB space is reflexive in the above sense remains open.

\subsection{Paired linear $hk$-spaces}\label{sec_intro_plin_hk_sp} Fortunately, there is way to circumvent \cref{question_reflexivity} which at the same time admits a natural interpretation in terms of the classical functional-analytic notion of \emph{dual system}, 
which we will refer to under the name of \emph{paired vector space}. 

\begin{defn}\label{defn_paired_vector_space}
    A \emph{paired vector space} (also called \emph{dual pair}, or \emph{dual system}) is a (plain, algebraic!) vector space $V$ together with a point-separating subspace $V^*\subseteq V'$ of its dual space $V'$. (Recall that a subset $S\subseteq V'$ is \emph{point-separating} if $\phi(x) = 0$ for all $\phi\in S$ implies $x=0$.) A \emph{morphism of paired vector spaces} $V,W$ is a linear map $f:V\to W$ such that for all $\phi \in W^*$, $f'(\phi)\in V^*$. Here, 
        $$ f': W' \to V', \;\; \phi \mapsto \phi \circ f, $$
    is the adjoint.
\end{defn}

Equivalently, one could define a paired vector space in a more symmetric way as a pair $(V,W)$ of vector spaces together with a non-degenerate bilinear map $b: V\times W \to \mathbb{K}$ (for example, this is the definition of a dual system in \cite[p.~123]{schaefer1971topological}). However, we would like to emphasise the role of the space $V$ as the ``carrier'', as the underlying set of a paired vector space. The point-separating subspace $V^*$ is to be viewed as ``additional structure''. This is also reflected in using the term ``paired vector space'' instead of the more common name ``dual system''. \par 
Paired vector spaces are a fundamental tool in the study of a locally convex topological vector space in terms of its dual, which Schaefer calls ``the central part of the modern theory of topological vector spaces, for it provides the setting for the deepest and most beautiful results of the subject'' \cite[p.~122]{schaefer1971topological}. \par 
Remarkably, paired vector spaces also have a very rich structure as a category, and admit a vast generalisation to the setting of symmetric monoidal categories called the \emph{separated-extensional Chu construction} (see Chapter 3, \cref{sec_chu_construction}). This construction enables us to ``internalise'' the notion of paired vector space to the category $\spaces$ of $hk$-spaces, resulting in the notion of \emph{paired linear $hk$-space}.

\begin{defn}\label{defn_paired_linear_hk_space_}
    A \emph{paired linear $hk$-space} is a linear $hk$-space together with a closed subspace $V^* \subseteq V^\wedge$ of its natural dual such that 
        $$ V \to (V^*)^\wedge, \;\; x \mapsto (\phi \mapsto \phi(x)), $$
    is a closed embedding. A \emph{morphism of paired linear $hk$-spaces} $V,W$ is a continuous linear map $f:V\to W$ such that for all $\phi \in W^*$, $f^\wedge(\phi)\in V^*$. 
\end{defn}

Thus, a paired linear $hk$-space is a linear $hk$-space with additional \emph{paired linear structure}. \par 
Any replete linear $hk$-space $V$ becomes a paired linear $hk$-space by taking $V^*=V^\wedge$. This applies in particular to the free replete linear $hk$-space $\mathcal{M}_c(X)$ on an $hk$-space $X$, which is also the free paired linear $hk$-space on $X$. Moreover, every Fréchet space admits only one unique paired linear structure compatible with its structure as a replete linear $hk$-space, and the same applies to Brauner spaces (see Chapter 3, \cref{Frechet_space_paired_linear_k_space_in_unique_way}). As described in \cref{sec_main_results}, paired linear $hk$-spaces possess a very rich structure as a category (they form a $*$-autonomous category) and they provide a natural language for the functional-analytic formulation of measure and integration theory. \par 
Using this language, we will construct a paired linear $hk$-space $\mathcal{M}(X)$ of ($\mathbb{K}$-valued) measures, dual to the paired linear $hk$-space $C_b(X)$ of continuous bounded functions, which gives rise to a commutative measure monad $\mathcal{M}$. Passing to the subspace of probability measures, we obtain the probability monad $\mathcal{P}$, thus resolving \cref{problem_measure_coherence}.

\subsection{Monadic vector-valued integration} The properties of $\mathcal{M}_c$ as a commutative monad and the free paired linear $hk$-space enable us to quickly develop a very satisfying theory of vector-valued integration (see Chapter 4, \cref{sec_monadic_vec_int}). This instantiates Kock's idea of ``commutative monads as a theory of distributions'' \cite{kock2011commutative} for the vector-valued integral. A general category-theoretic treatment of this idea was given by Lucyshyn-Wright in his doctoral thesis \cite{lucyshyn2013riesz}, in which the main example given uses \emph{convergence spaces} as a ``base category'', which form a well-behaved category, but can be rather unwieldy objects. Our contribution in this direction is to give an independent general development, streamlined to our intended application to $hk$-spaces, which -- in contrast to convergence spaces -- are simply particular topological spaces.   \par 
One view on monads coming from universal algebra (the study of algebraic structures in general) is that they encode ``generalised algebraic theories'' (see Chapter 3, \cref{sec_eilernberg_moore_cat}). In this view, our monad $\mathcal{M}_c$ encodes the ``generalised algebraic theory of vector-valued integration''.

%% file: Chapters/2_k_Spaces_Linear_k_Spaces_and_Smith_Duality.tex
\chapter{Linear $hk$-Spaces, Linear QCB Spaces and Smith Duality}

In this Chapter, we will first review the basic theory of $k$-spaces, $hk$-spaces and QCB spaces. We will then turn to \emph{linear $hk$-spaces} in \cref{sec_linear_hk_spaces_section}, providing many examples as well as basic constructions such as spaces of continuous functions, free linear $hk$-spaces and quotients. In addition, we investigate the topology on the space $L(V,W)$ of continuous linear maps between linear $hk$-spaces $V,W$ in \cref{sec_topology_of_sp_of_lin_cont_maps}. Finally, linear $hk$-spaces display the phenomenon of \emph{Smith duality}, the subject of \cref{sec_smithduality}. 

\section{$k$-Spaces and $hk$-Spaces}

\subsection{Equivalent characterisations of $k$-spaces}

Recall from the introduction that a $k$-space is a topological space $X$ for which continuity of a map $f:X\to Y$ (to another topological space $Y$) is equivalent to continuity of $f\circ p$ for any continuous map $p: K \to X$ from a compact Hausdorff space $K$ (\cref{defn_k_sp_hk_sp}). Several alternative equivalent ways to define the notion of $k$-space are treated below in \cref{char_k_sp} which will employ the following terminology. 

\begin{defn}
    Let $X$ be a topological space. A subset $U\subseteq X$ is \emph{$k$-open} (resp.~$k$-closed) if, for every compact Hausdorff space $K$ and every continuous map $f:K \to X$, $f^{-1}(U)$ is open (resp.~closed) in $K$.
\end{defn}

\begin{prop}\label{char_k_sp}
    Let $X$ be a topological space. Then the following are equivalent:
    \begin{enumerate}[1.]
        \item $X$ is a $k$-space (see \cref{defn_k_sp_hk_sp}). 
        \item Every $k$-open subset of $X$ is open.
        \item Every $k$-closed subset of $X$ is closed.
        \item $X$ is a quotient of some locally compact Hausdorff space.
        \item $X$ is the quotient of a disjoint union of compact Hausdorff spaces by some equivalence relation.
        \item $X$ is a colimit, in the category of topological spaces, of a diagram of compact Hausdorff spaces. 
    \end{enumerate}
\end{prop}
\begin{proof}
    For the equivalence of the first three points, see \cite[Section 3.1]{rezk2017compactly}, for the other ones, see \cite[Lemma 3.2, Corollary 3.4]{escardo2004comparing}.  
\end{proof}

\subsection{Examples of $k$-spaces and $k$-ification}

\subsubsection{Locally compact spaces} As \cref{char_k_sp} entails, every locally compact Hausdorff space is a $k$-space.

\subsubsection{Sequential spaces and metrisable spaces} Recall that \emph{sequential spaces} are those topological spaces on which it suffices to check continuity on sequences. More precisely:

\begin{defn}
    A map between topological spaces is \emph{sequentially continuous} if it maps convergent sequences to convergent sequences. A topological space $X$ is \emph{sequential} if any sequentially continuous map from $X$ (to any further topological space) is continuous. 
\end{defn}

Since a convergent sequence is nothing but a continuous map from the compact Hausdorff space $\mathbb{N}\cup \{\infty\}$, the one-point compactification of the natural numbers, we have that (see also \cite[Proposition 1.6]{strickland2009category}):

\begin{example}\label{ex_seq_spaces_k_space}
    Every sequential space is a $k$-space, including in particular all metrisable spaces.
\end{example}

\subsubsection{Closed subspaces}\label{open_and_closed_subsp_ksp} Closed subspaces of $k$-spaces are again $k$-spaces \cite[Proposition 3.4 (1)]{rezk2017compactly}. The analogous statement concerning open subspaces is true for $k$-Hausdorff $k$-spaces, which we will introduce later, see \cref{perm_prop_haus}.

\subsubsection{\texorpdfstring{$k$-ification}{k-ification}, limits and co-limits} 

An arbitrary topological space can be turned into a $k$-space, by ``$k$-ifying'' it:

\begin{defn}[$k$-ification]\label{defn_k_ification}
    Let $X$ be a topological space. The \emph{$k$-ification} of $X$, written $kX$, is given by the $k$-space whose underlying set is the same as that of $X$ and whose topology is given by the $k$-open subsets of $X$ (the \emph{$k$-ification of the topology} of $X$).
\end{defn}

Given a continuous map $f:X\to Y$ between topological spaces $X,Y$, we have that $f$ is also continuous when viewed as a map
    $$ kf: kX \to kY, \;\;\; x \mapsto f(x). $$
This allows us to make the following definition.

\begin{defn}[$\ktop$, $k$-ification as a functor]
    We write $\ktop$ for the category of $k$-spaces (with continuous maps as morphisms) and 
        $$ k: \topsp \to \ktop, \;\;\; X \mapsto kX, \; f \mapsto kf$$
    for the \emph{$k$-ification functor}.
\end{defn}

The $k$-ification functor can be characterised by a universal property \cite[Corollary 1.10]{strickland2009category}: 

\begin{prop}[Universal property of $k$-ification]\label{univ_prop_kification}
    The $k$-ification functor is right adjoint to the inclusion functor 
        $$\topsp(-): \ktop \hookrightarrow \topsp. $$
    In other words, a map $f: X \to Y$, where $X$ is a $k$-space and $Y$ is an arbitrary topological space, is continuous if, and only if, it is continuous when viewed as a map $X\to kY$. More concisely, a continuous map $X\to Y$ is equivalently a continuous map $X \to kY$: 
    \begin{equation}\label{kification_adjunction}
        \Hom_{\topsp}(\topsp(X), Y) \cong \Hom_{\ktop}(X, kY). \qquad (X \in \ktop, Y \in \topsp)
    \end{equation}
\end{prop}

Yet another way to phrase this is that $k$-spaces form a coreflective subcategory of $\topsp$. As an immediate corollary, we obtain:

\begin{cor}\label{ksp_complete_cocomplete}
    The category $\ktop$ of $k$-spaces is bicomplete, i.e.~it has all limits and colimits. The colimits are formed as in $\topsp$ and the limits are the $k$-ification of those in $\topsp$. In particular, the product in $\ktop$ is given by
        $$ X \times Y := k(X \times_\topsp Y),$$
    for any two $k$-spaces $X,Y$. Moreover, arbitrary quotients and disjoint sums (coproducts as formed in $\topsp$) of $k$-spaces are again $k$-spaces.
\end{cor}
\begin{proof}
    Coreflective subcategories always inherit (co-)limits in the fashion described (see \cite[p. 142, Proposition 4.5.15]{riehl2017category})
\end{proof}

\begin{warning}
    The notation $X\times Y$ (for $k$-spaces $X,Y$) will always refer to the product as formed in $\ktop$, which does \emph{not} coincide with the product in $\topsp$. An example of a $\topsp$-product of two $k$-spaces that fails to be a $k$-space will be given in \cref{example_VxVdash_not_k_space}, for a more elementary (but also less natural) example, see \cite[Example 3.3.29]{engelking1989general}.
\end{warning}

How much does $k$-ification modify the topology of a (Hausdorff) topological space? Concerning convergence of sequences and compactness, the answer is: \emph{not at all}. This is essentially a further consequence of the adjunction \eqref{kification_adjunction} and the subject of the following two corollaries.

\begin{cor}\label{conv_seq_in_kification}
    Let $X$ be a topological space. Then a sequence converges in $kX$ if, and only if, it converges in $X$.
\end{cor}
\begin{proof}
    Using the adjunction \eqref{kification_adjunction}, 
        $$\Hom_{\haus}(\nats \cup \{\infty\}, X) \cong \Hom_{\ktop}(\nats \cup \{\infty\}, kX),  $$
    where $\nats \cup \{\infty\}$ is the one-point compactification of the natural numbers. Since convergent sequences can be identified with continuous maps out of the one-point compactification of $\nats$, this shows the claim.
\end{proof}

\begin{cor}\label{cpct_sets_in_kififcation}
    Let $X$ be a topological space in which every compact subset is Hausdorff. Then the compact subsets of $X$ and its $k$-ification coincide.
\end{cor}
\begin{proof}
    Let $K\subseteq X$ be compact and let $i:K\to X$ be the inclusion map. Since $K$ is a $k$-space (it is compact \emph{Hausdorff}, by assumption), by the universal property of $k$-ification (\cref{univ_prop_kification}), $i$ is also continuous as a map $K \to kX$. Since the image of $K$ under this map is $K$ itself, and the image of a compact set under a continuous function is compact, this implies that $K\subseteq kX$ is compact. \par 
    Conversely, let $K\subseteq kX$ be compact. Then $K$ is also compact in $X$, since the topology of $kX$ is finer. In more detail, let $(U_i)_{i\in I}$ be a cover of $K$ by open subsets of $X$. Since every open subset of $X$ is $k$-open (i.e.~an open subset of $kX$), this implies that $(U_i)_{i\in I}$ is also an open cover of $K$ by open subsets of $kX$. Since $K$ is compact in $kX$, this implies that $(U_i)_{i\in I}$ has a finite subcover. Hence, $K$ is compact in $X$.
\end{proof}

\subsubsection{$k$-subspaces and $k$-embeddings} 

We have seen that arbitrary closed subsets of $k$-spaces become $k$-spaces under the subspace topology (see \cref{open_and_closed_subsp_ksp}). An \emph{arbitrary} subset of a $k$-space can be made into a $k$-space be choosing a ``corrected'' subspace topology: the $k$-ification of the subspace topology. 

\begin{defn}[$k$-subspaces, $k$-embeddings]\label{ksubspaces_def}
    Let $X$ be a $k$-space.
    \begin{enumerate}[1.]
        \item The \emph{$k$-subspace topology} on $Y\subseteq X$ is the $k$-ification of the subspace topology on $Y$.
        \item If $Y\subseteq X$ is endowed with the $k$-subspace topology, we call the resulting $k$-space a \emph{$k$-subspace} of $X$. 
        \item A continuous map $\iota: Y \to X$ of $k$-spaces is said to be a \emph{$k$-embedding} if the induced map,
            $$ \iota|^{\text{im}(\iota)}: Y \to \text{im}(\iota),$$
        where $\text{im}(\iota)$ carries the $k$-subspace topology, is a homeomorphism.
    \end{enumerate}
\end{defn}

\begin{remark}
    Notice that, on closed subsets, the $k$-subspace topology agrees with the usual subspace topology: $k$-ification is superfluous, since the usual subspace topology already yields a $k$-space. For the same reason, closed embeddings of $k$-spaces are equivalently $k$-embeddings with closed image. 
\end{remark}

\subsection{Cartesian closure of $k$-spaces}\label{ktop_ccc}

We now come to the main reason for why we work with $k$-spaces instead of arbitrary topological spaces: they form a cartesian closed category, thus providing a coherent notion of spaces of continuous maps. Although being essentially determined by the structure of $\ktop$ as a category, let us first give a concrete definition of what these mapping spaces look like.

\begin{defn}[Spaces of Continuous Maps]\label{mapping_spaces_defn}
    Let $X,Y$ be $k$-spaces. 
    \begin{enumerate}
        \item Define the topological space $\cco(X,Y)$ as the set of continuous function from $X$ to $Y$, endowed with the \emph{compact-open topology}. By definition, this topology is generated by the sub-basic open sets 
            $$W(i, K, U):= \{f:X\to Y\mid f(i(K))\subseteq U\},$$ 
        with $i, K, U$ ranging over compact Hausdorff spaces $K$, continuous maps $i:K\to X$ and open subsets $U\subseteq Y$.
        \item Now, the \emph{space of continuous maps} from $X$ to $Y$,
            $$Y^X := C(X,Y):= k\cco(X,Y),$$
        is the $k$-space given by the $k$-ification of $\cco(X,Y)$.
    \end{enumerate}
\end{defn}

\begin{remark}\label{cco_for_X_k_Hausdorff}
    If in the above definition, we assume that $X$ is Hausdorff (or more generally \emph{$k$-Hausdorff}, see \cref{khaus_defn}), the compact-open topology is equivalently generated by the subbase consisting of the sets
        $$ W(K, U) := \{f:X\to Y\mid f(K)\subseteq U\}, $$
    where $K$ ranges over the compact Hausdorff subsets of $X$ and $U$ ranges over the open subsets of $Y$. In the general case, however, the slightly more complicated definition is necessary, since the image of a compact Hausdorff space may fail to be Hausdorff.
\end{remark}

This construction does what it promises (see \cite[Proposition 2.12]{strickland2009category}):

\begin{prop}[Cartesian closure of $\ktop$]
    Let $X,Y,Z$ be $k$-spaces. Then the maps
        $$ Z^{X\times Y} \to (Z^Y)^X, \;\;\; f \mapsto (x \mapsto (y \mapsto f(x,y))),$$
    (``currying'') and 
        $$ (Z^Y)^X \to Z^{X\times Y}, \;\;\; f \mapsto ((x,y) \mapsto f(x)(y)), $$
    (``uncurrying'') are mutually inverse homeomorphisms.
\end{prop}

From this it follows that arbitrary ``higher order function terms'' such as 
    $$ F \mapsto (g \mapsto F(f \mapsto ((x, y) \mapsto g(x, f(y))))) $$
define continuous maps, provided the continuity of their constituents. We will make heavy use of this. 

\begin{example}
    When $X$ is a compact Hausdorff space and $Y$ is a metric space, then $\cco(X,Y)$ is the space of continuous maps $X\to Y$ with the topology of uniform convergence, which is induced by the metric,
        $$ d(f,g) = \sup_{x\in X} d(f(x), g(x)). $$
    Hence, $\cco(X,Y)$ is a $k$-space, so $C(X,Y) = k\cco(X,Y) = \cco(X,Y)$ and the topology on $C(X,Y)$ is exactly the topology of uniform convergence.
\end{example}

\subsection{The \texorpdfstring{$k$-Hausdorff}{k-Hausdorff} (equivalently, weak Hausdorff) property: \texorpdfstring{$hk$-spaces}{hk-spaces}}

The appropriate analogue of the Hausdorff property in the context of $k$-spaces is the \emph{$k$-Hausdorff property}. 

\begin{defn}\label{defn_k_hausdorff_property_hk_space_}
    A $k$-space $X$ is said to be \emph{$k$-Hausdorff} if the identity relation (``diagonal'')
    $$(\,=_X) \: :=\: \{(x,y)\in X\times X\mid x=y\}\;\;\,[\;\subseteq X\times X = k(X \times_\topsp X)\;]$$
    is closed in $X\times X$ (where the product is taken in $\ktop$, as always). For brevity, instead of ``$k$-Hausdorff $k$-space'', we will  mostly use the term ``$hk$-space''. Analoguous to how one obtains $k$-spaces by ``$k$-ifying'' topological spaces, we will see that $hk$-spaces are obtained by ``$h$-ification'' (or ``$k$-Hausdorffification'') of $k$-spaces. We will hence denote the full subcategory of $\ktop$ spanned by $hk$-spaces as $\spaces$.
\end{defn}

\begin{remark}
    The (usual) Hausdorff condition is equivalent to the identity relation being closed \emph{in the product $X\times_\topsp X$, taken in the category $\topsp$ of topological spaces}. Hence, the $k$-Hausdorff property essentially results from replacing the $\topsp$-product by the $\ktop$-product in the usual Hausdorff property.
\end{remark}
\begin{remark}
    Although the Hausdorff property is most often defined in terms of separation by open subsets, closure of the identity relation is usually how it is most directly used. For example, this is what allows us to conclude that the solution set of an equation, 
        $$ \{x \in X \mid f(x)=g(x)\}, $$
    i.e.~the equaliser of two continuous maps $f,g$, is closed. 
\end{remark}

$k$-Hausdorff $k$-spaces are also known as \emph{weak Hausdorff} $k$-spaces \cite[Proposition 2.14]{strickland2009category}:

\begin{prop}\label{khaus_equiv_weakhaus}
    A $k$-space $X$ is $k$-Hausdorff if, and only if, it is \emph{weakly Hausdorff}, meaning that for each compactum $K$ and every continuous map $f:K\to X$, $f(K)\subseteq X$ is closed.
\end{prop}

In general, $\topsp$-products of quotient maps may fail to be quotient maps \cite[p. 143, Example 7]{james2000topology}. Fortunately, this pathology can be cured by passing to $hk$-spaces \cite[Proposition 2.20]{strickland2009category}:

\begin{prop}\label{khausdorff_quots}
    Let $W, X, Y, Z$ be $k$-spaces. If $f:W\to X$ and $g:Y\to Z$ are quotient maps, then so is $f\times g: W\times Y \to X\times Z$.
\end{prop}

The Hausdorff condition is not well-behaved with respect to topological quotients. For example, there is no simple condition on an equivalence relation on a Hausdorff space that would guarantee the Hausdorff property of the quotient space (see \cite[p. 141]{james2000topology}). In contrast, closure of the equivalence relation is necessary and sufficient to preserve the $k$-Hausdorff property of $hk$-spaces \cite[Proposition 2.21]{strickland2009category}: 

\begin{prop}\label{khausdorff_quots_closed}
    Let $X$ be an $hk$-space and let $\sim$ be any equivalence relation on $X$. Then $X/\sim$ is $k$-Hausdorff if, and only if, $(\sim) \subseteq X\times X$ is closed.
\end{prop}

\subsubsection{The $k$-Hausdorff reflection and (co-)completeness of $k$-Hausdorff spaces} The category of $k$-Hausdorff spaces inherits completeness and cocompleteness from the category of $k$-spaces by means of the \emph{$k$-Hausdorff reflection}.

\begin{defn}[$k$-Hausdorff reflection, $\khaus$]\label{khaus_defn}
    \text{}
    \begin{enumerate}[1.]
        \item Let $X$ be a $k$-space and let $\sim_{\min}\subseteq X\times X$ be the smallest closed equivalence relation (i.e. the intersection of all closed equivalence relations) on $X$. The \emph{$k$-Hausdorff reflection}, or $h$-ification, $hX$ is defined as the quotient of $X$ by $\sim_{\min}$:
            $$ \hquot X := X/\sim_{\min}.$$
        \item We denote the category of $hk$-spaces (with continuous maps as morphisms) by $\khaus$.
    \end{enumerate}
\end{defn}

The $k$-Hausdorff reflection deserves this name:

\begin{prop}[Universal property of the $k$-Hausdorff reflection]
    Let $X$ be a $k$-space. Then:
    \begin{enumerate}[1.]
        \item $\hquot X$ is a $k$-Hausdorff space.
        \item (Universal property) For any continuous map $f:X\to Y$ from $X$ to some $k$-Hausdorff space there is a unique map $\overline{f}: \hquot X \to Y$ that factors through the projection $p:X \to \hquot X$. That is, the following diagram is commutative:
        \[\begin{tikzcd}
	       X & Y \\
	       \hquot X
	       \arrow["f", from=1-1, to=1-2]
	       \arrow["p"', from=1-1, to=2-1]
	       \arrow["{\overline{f}}"', dashed, from=2-1, to=1-2]
        \end{tikzcd}\]
        In other words, $\khaus$ is a reflective subcategory of $\ktop$, the reflector (i.e.~the left adjoint to the inclusion $\khaus \hookrightarrow \ktop$) being given by $h$-ification,
            $$ \hquot : \ktop \to \khaus. $$
    \end{enumerate}
\end{prop}
\begin{proof}
    \text{}
    \begin{enumerate}[1.]
        \item This follows directly from \cref{khausdorff_quots_closed}.
        \item The kernel equivalence relation on $X$ 
            $$\sim_{\ker(f) }\: := \:(f\times f)^{-1}[=_Y] \:\subseteq \:X\times X, $$
        given by, 
            $$ x \;\sim_{\ker(f)}\; y \;\;\;:\Leftrightarrow\;\;\; f(x)=f(y), $$
        is a closed equivalence relation $X$, being the preimage under the continuous map $f\times f$ of the identity relation $(=_Y)$ on $Y$, which is closed in $Y\times Y$, since $Y$ is assumed to be $k$-Hausdorff. Therefore, $\sim_{\min}\:\subseteq\: \sim_{\ker(f)}$ and by the universal property of the quotient, $f$ uniquely factors through the quotient projection $p$ via the induced map $\overline{f}$, which is what we wanted to show.
    \end{enumerate}
\end{proof}

As a direct consequence of the fact that $\spaces$ is a reflective subcategory of $\ktop$, we obtain:

\begin{cor}\label{cor_lims_colims_in_spaces}
    The category $\khaus$ of $k$-Hausdorff spaces is bicomplete, i.e.~complete and cocomplete. Limits are formed as in $\ktop$ (i.e. as the $k$-ification of limits in $\topsp$) and colimits are obtained as $k$-Hausdorff reflection of colimits in $\ktop$ (equivalently, in $\topsp$).
\end{cor}

\subsubsection{Permanence properties and cartesian closure of $hk$-spaces} For some types of colimits, taking the $k$-Hausdorff reflection as described in \cref{cor_lims_colims_in_spaces} above is superfluous: the colimit (as taken is $\ktop$ or, equivalently, $\topsp$) is often already $k$-Hausdorff, since the class of $k$-Hausdorff spaces exhibits some excellent permanence properties, as we now discuss.

\begin{prop}[Permanence properties of $hk$-spaces]\label{perm_prop_haus}
    \text{}
    \begin{enumerate}[1.]
        \item Any $k$-subspace of a $hk$-space is again $k$-Hausdorff (see \cref{ksubspaces_def} for the definition of a a $k$-subspace). For open subspaces as well as closed subspaces, the $k$-subspace topology coincides with the usual subspace topology.
        \item More generally, if $f: X\to Y$ is an injective continuous map, and $Y$ is $k$-Hausdorff, then $X$ is also $k$-Hausdorff.
        \item Any disjoint sum (i.e.~a coproduct in $\ktop$ or, equivalently, $\topsp$) of $hk$-spaces is again an $hk$-space. 
        \item Any quotient (as defined for topological spaces) of an $hk$-space by a closed equivalence relation is again $k$-Hausdorff. (This was already mentioned as \cref{khausdorff_quots_closed}.)
        \item If $Y$ is an $hk$-space and $X$ is \emph{any} $k$-space, then the space of continuous maps $Y^X$ (as formed in $\ktop$, see \cref{mapping_spaces_defn}) from $X$ to $Y$ is a $hk$-space.
    \end{enumerate}
\end{prop}
\begin{proof}
    \text{}
    \begin{enumerate}[1.]
        \item The statement about $k$-subspaces is a particular case of the next point, while the assertion about open subspaces and closed subspaces is \cite[Lemma 2.26]{strickland2009category}.
        \item By injectivity,
            $$ x_1 = x_2 \;\;\; \Leftrightarrow \;\;\; f(x_1)=f(x_2). \qquad (x_1,x_2\in X)$$
        Therefore, 
            $$ (=_X) = (f\times f)^{-1}[=_Y],$$
        which is closed, since $f$ is continuous and $X$ is $k$-Hausdorff, showing that $Y$ is $k$-Hausdorff, as well.
        \item See \cite[Proposition 4.7 (4)]{rezk2017compactly}.
        \item See \cref{khausdorff_quots_closed}.
        \item For any $f,g\in Y^X$,
            $$ f=g \;\;\;\Leftrightarrow \;\;\; \forall x\in X:\, f(x)=g(x). $$
        This displays the identity relation on $Y^X$ as an intersection of pre-images of a closed set:
            $$ (=_{Y^X}) = \bigcap_{x\in X} (\mathsf{ev}_x\times \mathsf{ev}_x)^{-1}[=_Y],$$
        where $\mathsf{ev}_x:Y^X\to Y$ is the evaluation map at $x\in X$, continuous by cartesian closure of $\spaces$, and $(=_Y)$ is closed in $Y\times Y$, by assumption.
    \end{enumerate}
\end{proof}

Using the cartesian closure of $\ktop$ and the final point of \cref{perm_prop_haus}, we obtain:

\begin{cor}
    $\spaces$ is a cartesian closed category.
\end{cor}

As a final observation concerning the $k$-Hausdorff property, we note the following relationship to the classical Hausdorff property.

\begin{prop}\label{kification_khausdorff}
    Let $X$ be a Hausdorff topological space. Then $kX$ is an $hk$-space.
\end{prop}
\begin{proof}
    Since $X$ is Hausdorff, the diagonal is closed in $X\times_{\topsp} X$, so it is also closed in the finer topology of $kX \times kX = k(kX \times_{\topsp} kX)$.
\end{proof}

\subsubsection{Sequential colimits of $hk$-spaces}

Sequential colimits in $\spaces$ along closed inclusions of subspaces have a simple characterisation \cite[Lemma 3.7]{strickland2009category}: 

\begin{lem}\label{seq_colims_in_spaces}
    Let $X$ be an $hk$-space and $(A_n)$ be an increasing sequence of closed subspaces. Then the following are equivalent:
    \begin{enumerate}
        \item The union 
            $$ A := \bigcup_{n\in \mathbb{N}} A_n \subseteq X $$
        is closed in $X$ and we have an isomorphism 
            $$ A = \colim_{n\in \mathbb{N}} A_n. $$
        \item For each compact subset $K$ of $X$, there is some $n\in \mathbb{N}$ such that $K\cap A = K \cap A_n$.
    \end{enumerate}
\end{lem}

\subsection{The spaces $C(X,Y)$ for codomains that are metric spaces} When $Y$ is a metric space, we can give a very explicit description of both the convergent sequences and the compact subsets in the space $Y^X$ of continuous maps that arises from the cartesian closed structure of $\spaces$. This applies in particular to the space of continuous functions $C(X)$ on an $hk$-space $X$. 

\begin{lem}\label{lem_convergence_in_Y_X}
    Let $Y$ be a metric space and let $X$ be any $hk$-space. Then a sequence $(f_n)$ converges in $C(X,Y)=Y^X$ if, and only if, it converges uniformly on compact subsets of $X$.
\end{lem}
\begin{proof}
    The space $Y^X$ is given as the $k$-ification of the space $\cco(X,Y)$ (see \cref{mapping_spaces_defn}) and the convergent sequences in a topological space and its $k$-ification agree (see \cref{conv_seq_in_kification}). The claim now follows from the fact that, since $X$ is an $hk$-space, the compact-open topology on $\cco(X,Y)$ is exactly the topology of uniform convergence on compact subsets (see \cref{cco_for_X_k_Hausdorff}).
\end{proof}

The Arzelà-Ascoli theorem has a natural formulation in the context of $k$-spaces.

\begin{thm}[Arzelà-Ascoli]\label{thm_arzela_ascoli}
    Let $X$ be an $hk$-space and let $Y$ be a metric space. Then a subset of $C(X,Y)$ is compact if, and only if, it is pointwise relatively compact, equicontinuous and closed.
\end{thm}
\begin{proof}
    The classical Arzelà-Ascoli theorem makes the same assertion about $\cco(X,Y)$, \emph{with the compact open topology} \cite[p.~134, 18]{kelley2017general}. Since the compact subsets of $C(X,Y)$ and $\cco(X,Y)$ coincide (see \cref{cpct_sets_in_kififcation}), it only remains to show that a closed, pointwise relatively compact and equicontinuous subset of $C(X,Y)$ is also closed in $\cco(X,Y)$. So let $A\subseteq C(X, Y)$ be closed, pointwise relatively compact and equicontinuous. Then $A$ is $k$-closed in $\cco(X,Y)$, meaning that for every compact subset $K\subseteq \cco(X,Y)$, $A\cap K$ is closed in $K$. Now, let $K$ be the closure of $A$ in $\cco(X,Y)$. By the classical Arzelà-Ascoli theorem, $K$ is compact. Hence, $A\cap K = A$ is closed in $K$ and therefore $A$ is closed in $\cco(X,Y)$, which is what we wanted to show.
\end{proof}

\subsection{Quotients of countably based (QCB) spaces}

Recall that a \emph{second-countable} space is a topological space which has a countable base. Second-countable space are also known as \emph{countably based spaces}.

\begin{defn}
    A \emph{QCB space} is a topological space which is the quotient of a countably based space. In more detail, a topological space $X$ is a QCB space if there exists a second-countable space $Y$ and a quotient map $Y\twoheadrightarrow X$. We denote the category of QCB spaces by $\mathsf{QCB}$.
\end{defn}

The relevance of QCB spaces and their remarkable properties were first recognised in the context of domain theory and computable analysis \cite{simpson2003towards,schroder2021admissibly}. In our setting, they provide a subcategory of particularly well behaved $k$-spaces which is nevertheless closed under the formation of spaces of continuous maps, countable limits and countable colimits.

\subsubsection{QCB spaces are sequential, hence $k$-spaces} Second countable spaces are sequential and the quotient of a sequential space is again sequential, so every QCB is a sequential space. In particular, every QCB space is a $k$-space (see \cref{ex_seq_spaces_k_space}).

\subsubsection{QCB spaces are closed under exponentials and countable (co-)limits} Since a countable coproduct of countably based spaces is again countably based, and quotients of QCB spaces are evidently QCB spaces, it follows that $\mathsf{QCB}$ is closed under the formation of countable colimits in $\ktop$. What is remarkable is that QCB spaces are also closed under the formation of countable limits and exponentials \cite[Corollary 7.3, Remark 7.4]{escardo2004comparing}:

\begin{prop}\label{QCB_cart_closed}
    Let $X,Y$ be QCB spaces. Then $X\times Y$ and $Y^X=C(X,Y)$ are QCB spaces, as well. Therefore, $\mathsf{QCB}$ is a cartesian closed category. Moreover, $\mathsf{QCB}$ is closed under countable limits and colimits (as formed in $\ktop$).
\end{prop}

\begin{remark}\label{rem_qcb_canonical}
    In \cite{escardo2004comparing}, it is shown that  \cref{QCB_cart_closed} still holds if one forms exponentials and limits in  cartesian closed categories of topological spaces other than $\ktop$ (satisfying certain mild properties). Hence, the topology of $C(X,Y)$ truly does not depend on arbitrary choices when $X$ and $Y$ are QCB spaces -- it will be the same in \emph{any} reasonable category of topological spaces.
\end{remark}

As a consequence of closure under countable limits, we obtain:

\begin{cor}\label{closed_subspaces_QCB}
    A \emph{closed} subspace of a QCB space is a QCB space (in the subspace topology). 
\end{cor}
\begin{proof}
    Let $X$ be a QCB space and let $Y\subseteq X$ be a closed subspace. Let $Z:=X/Y$ be the quotient obtained by collapsing $Y$ to a point $*\in Z$. As a quotient of a QCB space, $Z$ is a QCB space, and as a closed subspace of a $k$-space, $Y$ is a $k$-space. Now, $Y$ is the equaliser (in $\ktop$) of the canonical projection $X\to Z$ and the constant map $X\to Z$, $x \mapsto *$. Since $\mathsf{QCB}$ inherits countable limits (in particular, equalisers) from $\ktop$ (see see \cref{QCB_cart_closed}), $Y$ is a QCB space.
\end{proof}

\subsubsection{QCB spaces satisfy strong countability properties} Recall that a topological space is \emph{Lindelöf} if every open cover has a countable subcover. It is furthermore \emph{hereditarily Lindelöf} if every subspace has the Lindelöf property, and \emph{hereditarily separable} every subspace is separable.

\begin{prop}\label{QCB_separable_lindelof}
    Every QCB space is hereditarily separable and hereditarily Lindelöf. 
\end{prop}
\begin{proof}
    Second-countable spaces are both hereditarily Lindelöf and hereditarily separable, and the properties of being hereditarily Lindelöf and of being hereditarily separable are stable under the formation of quotients. 
\end{proof}


\section{Linear $hk$-Spaces}\label{sec_linear_hk_spaces_section}

\subsection{Linear $k$-spaces, linear $hk$-spaces, linear QCB spaces} Since the product in $\ktop$ is different from the one in the category of topological spaces, a similarly adjusted notion of vector space ``internal'' to the category of $k$-spaces provides a natural replacement for the notion of topological vector space.

\begin{defn}\label{def_lin_hk_space}
    A \emph{linear $k$-space} is a $k$-space $V$ together with two continuous maps (addition and scalar multiplication),
        $$ +: V \times V \to V,$$
        $$ \cdot: \mathbb{K}\times V \to V,$$
    such that $V$ forms a vector space with respect to the operations $\cdot$ and $+$. (Here, the product $V \times V = k(V\times_\topsp V)$ denotes the product in $\ktop$, as always!) A \emph{linear $hk$-space} is a linear $k$-space which is also $k$-Hausdorff, and a \emph{linear QCB space} is a linear $k$-space which is also a QCB space. We denote the category of linear $hk$-spaces by $\vect$.
\end{defn}

\begin{remark}
    In concise category-theoretic terminology, linear $k$-spaces are exactly $\mathbb{K}$-module objects over the ring object $\mathbb{K}$ in $\ktop$,  linear $hk$-spaces are precisely $\mathbb{K}$-module objects in $\spaces$, and linear QCB spaces are nothing but $\mathbb{K}$-module objects in $\mathsf{QCB}$. (Note that, in these assertions, we use that the product in all of these categories agrees.) 
\end{remark}

\begin{warning}\label{warning_topvec_vs_linksp}
    Every $k$-space is a topological space, by definition, and the definition of a linear $k$-space is very similar to that of a topological vector space. \emph{However,} a linear $k$-space is \emph{not} necessarily a topological vector space with respect to the same topology and vector space structure, see \cref{ex_lin_k_space_not_top_vec_sp}. Moreover, in the other direction, a topological vector space is also not necessarily a linear $k$-space (with respect to the same topology and vector space structure), see \cref{top_vec_spaces_not_lin_k_spaces}. What does hold is that the $k$-ification of a topological vector space is always a linear $k$-space (see \cref{k_ification_lin_k_sp}).
\end{warning}

\subsection{Examples of linear $hk$-spaces.}\label{ex_lin_ksp}

\subsubsection{The linear $hk$-space $C(X,V)$ of continuous functions} First of all, the base field $\mathbb{K}$ itself is a linear $hk$-space. Moreover, if $X$ is any $k$-space and $V$ is any linear $k$-space, then the space $C(X,V)=V^X$ of continuous maps from $X$ to $V$ is again a linear $k$-space (and a linear $hk$-space if $V$ is $k$-Hausdorff). Continuity of addition and scalar multiplication follows from cartesian closure of $\ktop$, since these maps are given by:
    $$ +: V^X \times V^X \to V^X, \;\; (f,g) \mapsto (x \mapsto f(x)+g(x)), $$
    $$ \cdot: \mathbb{K} \times V^X \to V^X, \;\; (\lambda, f) \mapsto (x \mapsto \lambda\cdot f(x)). $$
When $V=\mathbb{K}$, we will write $C(X):= C(X, \mathbb{K})$. 

\subsubsection{Closed linear subspaces, products, limits} Closed linear subspaces of linear $hk$-spaces inherit this structure, as well, simply by restricting the vector space operations. Moreover, if $(V_i)_{i\in I}$ is a family of linear $hk$-spaces, then its product, 
    $$ \prod_{i\in I} V_i, $$
is a linear $hk$-space, with the vector space operations given component-wise and the topology being that of the ($\ktop$-)product. We can summarise this as follows.

\begin{prop}\label{lin_hk_spaces_complete}
    The category $\vect$ is complete (i.e.~it has all limits), with limits being computed as in $\spaces$ (equivalently, in $\ktop$). 
\end{prop}
\begin{proof}
    Limits in $\spaces$ are formed as closed subspaces (i.e.~certain equalisers) of products. Since these operations inherit a linear $hk$-space structure in the same way as for algebraic vector spaces, it follows immediately that they satisfy the necessary universal properties also in $\vect$.
\end{proof}

\subsubsection{The linear $k$-space $L(V,W)$ of continuous linear maps} Combining the previous two examples yields a natural definition of a linear $k$-space of continuous linear maps.

\begin{defn}\label{defn_lvw}
    Let $V,W$ be linear $k$-spaces. Define the linear $k$-space of continuous linear maps $V\to W$ as,
        $$ L(V,W) := \{f\in C(V,W)\mid \forall x,y\in V \:\forall \lambda\in \mathbb{K}: f(\lambda x + y)= \lambda f(x)+ f(y)\} \subseteq C(V,W), $$
    endowed with the $k$-subspace topology induced from $C(V,W)$. 
\end{defn}

When $W$ is $k$-Hausdorff, then $L(V,W)$ is closed in $C(V,W)$ (being an intersection of closed subsets). In this case, $L(V,W)\subseteq C(V,W)$ carries the (usual) subspace topology and is again $k$-Hausdorff (see \cref{perm_prop_haus}).

\begin{lem}\label{L_VW_QCB}
    If $V,W$ are linear QCB spaces, then $L(V,W)$ is a QCB space as well.
\end{lem}
\begin{proof}
    As a $k$-space, $L(V,W)$ is the equaliser of the continuous maps,
        $$ C(V,W) \to C(\mathbb{K} \times V \times W, W), \;\; f \mapsto ((\lambda, x, y) \mapsto f(\lambda x + y), $$
    and, 
        $$ C(V,W) \to C(\mathbb{K} \times V \times W, W), \;\; f \mapsto ((\lambda, x, y) \mapsto \lambda f(x) + f(y). $$
    Equalisers are limits, so the claim follows from  \cref{QCB_cart_closed}.
\end{proof}

\subsubsection{The natural dual \texorpdfstring{$V^{\wedge}$}{V^}} As a particular case of the above, we define:

\begin{defn}\label{defn_natural_dual}
    Let $V$ be a linear $k$-space. The \emph{natural dual} of $V$, 
        $$ V^\wedge := L(V, \mathbb{K}), $$
    is the space of continuous linear functionals on $V$, topologised as a closed subspace of the space of continuous maps $C(V)$.
\end{defn}

As an immediate corollary to \cref{L_VW_QCB}, we obtain:

\begin{cor}\label{nat_dual_QCB}
    The natural dual $V^\wedge$ of a linear QCB space is a linear QCB space.
\end{cor}

\subsubsection{Metrisable Hausdorff topological vector spaces, Fréchet spaces} Despite \cref{warning_topvec_vs_linksp}, every \emph{metrisable} Hausdorff topological vector space is a linear $hk$-space. This is because every metrisable space is a $k$-space and the $\ktop$-product of metrisable spaces agrees with the $\topsp$-product (and is again metrisable). Hence, on the class of metrisable spaces, there is no difference between the notions of linear $k$-space and topological vector space. In particular, every Fréchet space and (the underlying topological vector space of) every Banach space is a linear $hk$-space. 

\subsubsection{The \texorpdfstring{$k$-ification}{k-ification} of a topological vector space is a linear $k$-space} Many further examples are provided by the following proposition.

\begin{prop}\label{k_ification_lin_k_sp}
    Let $E$ be a topological vector space. Then $V:=kE$ is a linear $k$-space (with respect to the same vector space structure). If $E$ is Hausdorff, then $V$ is $k$-Hausdorff. 
\end{prop}
\begin{proof}
    It suffices to note that identity $V\times V = k(kE\times_{\topsp} kE) \to E\times_{\topsp} E$ is continuous. The second part of the statement is a special case of \cref{kification_khausdorff}.
\end{proof}

\subsubsection{Free Linear $hk$-Spaces}

A further class of examples comes from the following observation.

\begin{prop}\label{free_lin_hk_sp_existence}
    The forgetful functor $U:\vect \to \spaces$ has a left adjoint, 
        $$ F_{\mathbb{K}}: \spaces \to \vect. $$
\end{prop}
\begin{proof}
    The category $\spaces$ of $hk$-spaces is complete by  \cref{lin_hk_spaces_complete}, and by construction of these limits, the forgetful functor $U$ preserves them. By the adjoint functor theorem, it remains to show that $U$ satisfies the \emph{solution set condition}. \par Let $X$ be an $hk$-space. We need to find a family $(V_i)_{i\in I}$ of linear $hk$-spaces and a family of continuous maps, 
        $$ f_i: X \to V_i, \qquad (i\in I) $$
    such that every continuous map $g: X \to W$ (with $W$ some further linear $hk$-space) factors through some continuous linear map $h: V_i \to W$ (for some index $i\in I$). \par 
    Let $|\mathbb{K}^{(|X|)}|$ be the free algebraic vector space on $X$, viewed as a discrete set, and let $I$ be the set of all possible linear $hk$-space structures on subsets of $|\mathbb{K}^{(|X|)}|$. For each $i\in I$, let $V_i$ be the resulting linear $hk$-space. Hence, every linear $hk$-space of cardinality less than $|\mathbb{K}^{(|X|)}|$ is isomorphic to some $V_i$ (for some $i\in I$). Now, if $g:X\to W$ is an arbitrary continuous map to some further linear $hk$-space $W$, then $g$ factors through the inclusion of the closed linear span of the image of $g$, 
        $$ Z:= \overline{ \vspan \im(g)} \subseteq W, $$
    which we may topologise as a closed subspace of $W$, resulting in a linear $hk$-space $Z$. The cardinality of $Z$ is less or equal to that of $|\mathbb{K}^{(X)}|$, so there exists an index $i\in I$ such that $Z \cong V_i$ as linear $hk$-spaces. Hence, $g$ factors through the inclusion of $Z\cong V_i$ into $W$ and the solution set condition is verified.  
\end{proof}

\begin{defn}
    As appropriate for the left adjoint to a forgetful functor, we call $F_{\mathbb{K}}$ the \emph{free-linear-$hk$-space functor}. Accordingly, given an  $hk$-space $X$, we refer to $F_{\mathbb{K}}(X)$ as the \emph{free linear $k$-space} on $X$. 
\end{defn}

\subsubsection{A linear $hk$-space which is not a topological vector space} A linear $hk$-space need not be a topological vector space with respect to the same topology and vector space structure (as mentioned before in \cref{warning_topvec_vs_linksp}):

\begin{example}\label{ex_lin_k_space_not_top_vec_sp}
    Let $X$ be an uncountable Tychonov $k$-space. Then the free linear $hk$-space $F_{\mathbb{K}}(X)$ on $X$ is \emph{not} a topological vector space. 
\end{example}
\begin{proof}
    Suppose to the contrary that $F_{\mathbb{K}}(X)$ is a topological vector space. Then $F_{\mathbb{K}}(X)$ is Hausdorff, as the origin is closed in $F_{\mathbb{K}}(X)$ (by the $k$-Hausdorff property). Moreover, in addition to being the free linear $hk$-space on $X$, it is also the free \emph{Hausdorff topological} vector space on $X$, i.e.~for every Hausdorff topological vector space $E$ and continuous map $f: X\to E$, there is a unique continuous linear map $\tilde{f}: F_{\mathbb{K}}(X)\to E$ making the diagram
    \[\begin{tikzcd}
	{F_{\mathbb{K}}(X)} & E \\
	X
	\arrow["{\tilde{f}}", dashed, from=1-1, to=1-2]
	\arrow[from=2-1, to=1-1]
	\arrow["f"', from=2-1, to=1-2]
    \end{tikzcd}\]
    commute. This is because a continuous map from a $k$-space to $E$ is equivalently a continuous map to $kE$ (which is a linear $hk$-space), so the universal property shown in the diagram does indeed reduce to the one of the free linear $hk$-space. \par 
    But this contradicts the result of \cite[Fact 4.18]{gabriyelyan2017free} which states that the free Hausdorff topological vector space on an uncountable Tychonov space is never a $k$-space. Therefore, $F_{\mathbb{K}}(X)$ is not a topological vector space.
\end{proof}

\subsubsection{A topological vector space which is not a linear $k$-space} An example in the other direction is the following. 

\begin{example}\label{top_vec_spaces_not_lin_k_spaces}
    Let $H$ be an infinite-dimensional Hilbert space. Then the weak-$*$ dual of $H$ is not a $k$-space, see \cite[Proposition 1]{frolicher1972topologies} for a proof. 
\end{example}

\begin{remark}
    For most purposes, the appropriate substitute for the weak-$*$ dual in the context of $k$-spaces seems to be the natural dual $V^\wedge = L(V,\mathbb{K})$. When $V$ is a Fréchet space, convergence of sequences in $V^\wedge$ is exactly weak-$*$ convergence (\cref{conv_in_dual_of_fre}). Moreover, still assuming that $V$ is a Fréchet space, a subset of $V^\wedge$ is compact if, and only if, it is compact in the weak-$*$ topology (see \cref{prop_dual_fre_carries_co_topology}). 
\end{remark}

In the introduction, we mentioned that the canonical evaluation pairing between a locally convex Hausdorff topological vector space and its dual is in some sense ``unavoidably discontinuous’’ (see \cref{sec_continuity_problems}). In conjunction with the fact that the dual of a Fréchet space is a linear $hk$-space when equipped with the compact-open topology (which we will show later, see \cref{prop_dual_fre_carries_co_topology}), this leads to a further example, showing also that the product in $\topsp$ does not generally coincide with the one in $\spaces$.

\begin{prop}\label{prop_ex_cont_pairing}
    Let $V$ be a locally convex Hausdorff topological vector space (short: LCTVS). Let $V'$ be its dual, equipped with any topology that turns it into an LCTVS. 
    Suppose that the map,
    $$ \mathsf{ev}: V\times_{\topsp} V' \to \mathbb{K}, \;\; (x, \phi) \mapsto \phi(x), $$
    is continuous (with respect to the product $\times_{\topsp}$ in $\topsp$!).
    Then $V$ is normable.
\end{prop}
\begin{proof}
    First, by continuity of $\mathsf{ev}$ and local convexity of $V$ and $V'$, there exist convex neighbourhoods $U_V\subseteq V$, $U_{V'}\subseteq V'$ of $0$ such that
    \begin{equation}\label{eq_inclusion_cont_example}
        \mathsf{ev}(U_V\times U_{V'}) \subseteq \{\lambda \in \mathbb{K}\mid |\lambda|\leq 1\}.
    \end{equation}
    This implies that $U_V$ is weakly bounded, since for every $\phi\in V'$, there exists some $t>0$ such that $t \phi \in U_{V'}$ (because $U_{V'}$ is a neighbourhood of the origin and therefore absorbing), so by \eqref{eq_inclusion_cont_example}  $|\phi(x)|< t^{-1}$, for all $x\in U_V$. Hence, $U_V$ is bounded also in the original topology of $V$ (by \cite[Theorem~3.18]{rudin1991functional}) and therefore, $V$ is is normable, as it contains a convex bounded neighbourhood of the origin (this is Kolmogorov's normability criterion \cite[Theorem~1.39]{rudin1991functional}). 
\end{proof}

\begin{example}\label{example_VxVdash_not_k_space}
    Let $V$ be a non-normable Fréchet space. Then the natural dual $V^\wedge$ agrees with the compact-open dual (\cref{prop_dual_fre_carries_co_topology}), which is an LCTVS. By \cref{prop_ex_cont_pairing}, the evaluation pairing $V \times_{\topsp}V^\wedge \to V$ is discontinuous, while being continuous when regarded as a map $V\times V \to V$ (by cartesian closure of $\spaces$). Hence, $V \times_{\topsp}V^\wedge$ is not homeomorphic to its $k$-ification $k(V \times_{\topsp}V^\wedge)= V\times V^\wedge$ and therefore the LCTVS $V \times_{\topsp}V^\wedge$ is not a $k$-space, illustrating also the difference between the products of $\topsp$ and $\spaces$.
\end{example}

\subsection{The topology of spaces of continuous linear maps}\label{sec_topology_of_sp_of_lin_cont_maps}

We now turn to investigating some of the features of the topology of the spaces $L(V,W)$. 

\subsubsection{For Fréchet spaces, convergence of sequences in $L(V, W)$ is strong convergence} What does convergence of sequences in $L(V,W)$ amount to? In the case of Fréchet spaces, the answer is particularly simple.

\begin{prop}\label{convergence_in_LVW}
    Let $V$ be a Fréchet space and $W$ be a linear $k$-space which is also a Hausdorff topological vector space in the same topology. Then convergence of sequences of continuous linear maps in $L(V,W)$ is given equivalently by strong (i.e. pointwise) convergence. 
\end{prop}
\begin{proof}
    The hypotheses that $V$ is Fréchet and $W$ is a Hausdorff topological vector space ensure that the uniform boundedness principle applies (see \cite[p. 42, 2.6 Theorem]{rudin1991functional}). Hence, a pointwise convergent sequence in $L(V,W)$ is equicontinuous and hence also converges compactly and therefore in $L(V,W)$ (because the convergent sequences in a topological space and its $k$-ification coincide).
\end{proof}

Concerning convergence in the natural dual $V^\wedge$ of a 
Fréchet space $V$, we record the following important 
particular case of \cref{convergence_in_LVW}.

\begin{cor}\label{conv_in_dual_of_fre}
    Let $V$ be a Fréchet space. Then a sequence converges in $V^\wedge$ if, and only if, it converges pointwise.
\end{cor}

\subsubsection{For separable Fréchet spaces, the topology on $L(V,W)$ is the sequential strong topology} When $V$ and $W$ are \emph{separable} Fréchet spaces, we can characterise the topology of $L(V,W)$ completely in terms of strongly convergent sequences. 

\begin{lem}\label{lem_LVW_seq}
    Let $V, W$ be separable Fréchet spaces. Then $L(V,W)$ is a QCB space, and hence in particular sequential space.
\end{lem}
\begin{proof}
    Since QCB spaces are sequential spaces and separable Fréchet spaces are second countable, this is a special case of \cref{L_VW_QCB}.
\end{proof}

\subsubsection{The natural dual of a Fréchet space is its compact-open dual} In general, the topology on $V^\wedge$ is \emph{not} the compact-open topology, but only its $k$-ification. However, in case of Fréchet spaces, no $k$-ification is necessary:

\begin{prop}\label{prop_dual_fre_carries_co_topology}
    Let $V$ be a Fréchet space. Then $V^\wedge$ carries the compact-open topology (i.e.~the topology of uniform convergence on compact subsets). Moreover, the topology of $V^\wedge$ agrees with the weak-$*$ topology on every compact subset of $V^\wedge$, and a subset of $V^\wedge$ is compact if, and only if, it is closed and equicontinuous (which is the case if, and only if it is compact in the weak-$*$ topology).
\end{prop}
\begin{proof}
    Write $V^{*[c.o.]}$ for the compact-open dual of $V$. According to the Banach-Dieudonné theorem, the topology on $V^{*[c.o.]}$ is the finest topology which agrees with the weak-$*$ topology on equicontinuous subsets (see \cite[p.151, 6.3 Theorem, Corolary 2]{schaefer1971topological}). 
    The finest topology which agrees with the weak-$*$ topology on equicontinuous subsets is equivalently the finest topology which agrees with the weak-$*$ topology on compact subsets of $V^{*[c.o.]}$, since the equicontinuous subsets of $V^{*[c.o.]}$ are precisely the relatively compact subsets, by the Arzela-Ascoli theorem (\cref{thm_arzela_ascoli}) and the fact that equicontinuous sets of functionals are automatically pointwise bounded \cite[p. 83, Corollary]{schaefer1971topological}.
    Hence, the Banach-Dieudonné theorem entails that the topology of $V^{*[c.o.]}$ is the final topology with respect to all inclusion maps $K\to V^{*[c.o.]}$ of compact subsets $K$. In other words, $V^{*[c.o.]}$ is the filtered colimit
        $$ V^{*[c.o.]} = \colim_{K \subseteq V^{*[c.o.]}} K $$
    over its compact subsets $K$ along the inclusion maps. As a colimit of compact Hausdorff spaces, $V^{*[c.o.]}$ is a $k$-space and therefore, $V^\wedge=k V^{*[c.o.]}=V^{*[c.o.]}$, i.e.~$V^\wedge$ carries the compact-open topology. In conjunction with the aforementioned fact that the compact-open topology agrees with the weak-$*$ topology on compact subsets and the Alaoglu-Bourbaki theorem \cite[p. 84,~4.3, Corollary]{schaefer1971topological}, this also shows the second part of the claim.
\end{proof}

\subsubsection{The natural dual of a Fréchet space is a Brauner space} Recall from the introduction that a \emph{Brauner space} is a linear $hk$-space which is hemicompact (a sequential colimit of compact subsets along inclusion maps, see \cref{defn_brauner_space}) and a complete locally convex topological vector space with respect to the same underlying topology.

\begin{proof}\label{prop_dual_of_fre_bra}
    Let $V$ be a Fréchet space. Then $V^\wedge$ is a Brauner space. 
\end{proof}
\begin{proof}
    By \cref{prop_dual_fre_carries_co_topology}, $V^\wedge$ carries the compact-open topology, which is a complete
    locally convex vector space topology. It remains to show that $V^\wedge$ is hemicompact. Since $V$ is a Fréchet space, it has a countable base $(U_n)$ of neighbourhoods of the origin. Since 
        $$ \bigcap_{n\in \mathbb{N}} U_n = \{0\}, $$
    we have that 
        $$ \bigcup_{n\in \mathbb{N}} U_n^\circ = \{0\}^\circ = V^\wedge, $$
    where
        $$ S^\circ := \{\phi \in W^\wedge \mid \forall x\in S: |\phi(x)|\leq 1\}$$ 
    denotes the \emph{polar} of a subset $S$ (see \cite[p. 125, 1.3]{schaefer1971topological}). Each $K_n := U_n^\circ$ is compact, so $(K_n)$ is an increasing sequence of compact subsets of $V^\wedge$ and we claim that 
        $$ V^\wedge = \colim_{n\in \mathbb{N}} K_n $$
    is the sequential colimit (in $\ktop$) along the inclusion maps. By \cref{seq_colims_in_spaces}, it suffices to show that every compact subset of $V^\wedge$ is contained in one of the $K_n$. So let $L\subseteq V^\wedge$ be compact. Then $L$ is equicontinuous (see \cref{prop_dual_fre_carries_co_topology}), which implies that there is a basic open neighbourhood $U_n$ of the origin in $V$ such that $L \subseteq U_n^\circ=K_n$, so we are done.
\end{proof}

\subsubsection{The space $L(V,W)$ is a Fréchet space for $V$ Brauner and $W$ Fréchet} In the other direction, we have the following result.

\begin{prop}\label{prop_LVW_Frechet_if_V_bra_and_W_fre}
    Let $V$ be a Brauner space and let $W$ be a Fréchet space. Then $L(V,W)$ carries the compact-open topology and is a Fréchet space.
\end{prop}
\begin{proof}
    By assumption $V$ is a Brauner space, so we may write, 
        $$ V = \colim_{n\in\mathbb{N}} K_n, $$
    where $(K_n)$ is an increasing sequence of compact subsets of $V$. Now, given a countable family of seminorms $(\|\cdot\|_n^W)$ inducing the topology on $W$, a countable family of seminorms on $L(V,W)$ inducing the compact-open topology is given by
        $$ \|f\|_{n,m}^{L(V,W)} := \sup_{x\in K_n} \|f\|_m^{V}. \qquad (n,m \in \mathbb{N}, \, f\in L(V,W))$$
    Hence, the compact-open topology on $L(V,W)$ is a Fréchet topology. In particular, it is metrisable and therefore coincides with the topology on $L(V,W)$. 
\end{proof}

Let us record a direct corollary.

\begin{cor}\label{prop_dual_of_bra_fre}
    If $V$ is a Brauner space, then $V^\wedge$ is a Fréchet space and carries the compact-open topology. 
\end{cor}

\subsubsection{The natural dual of a \emph{nuclear} Fréchet space is its strong dual} On the dual of a nuclear Fréchet space (such as the space $\mathcal{S}'(\mathbb{R}^n)$ of tempered distributions), one usually considers the \emph{strong topology}, which is given by the topology of uniform convergence on bounded subsets (and not to be confused with the strong \emph{operator} topology which is given by the topology of \emph{pointwise} convergence). It turns out that this topology coincides with that of the natural dual. \par 
Recall that a Fréchet space has the \emph{Heine-Borel property} if every closed bounded set of $V$ is compact. This always holds in nuclear Fréchet spaces such as $\mathcal{S}(\mathbb{R}^n)$, see \cite[p. 101, Corollary 2]{schaefer1971topological}.

\begin{prop}\label{prop_heine_borel_implies_dual_strong_top}
    Let $V$ be a Fréchet space with the Heine-Borel property. Then $V^\wedge$ carries the strong topology.
\end{prop}
\begin{proof}
    By \cref{prop_dual_fre_carries_co_topology}, the topology on $V^\wedge$ is the compact-open topology, i.e.~the topology of uniform convergence on compact subsets. Since $V$ has the Heine-Borel property, uniform convergence on compact subsets is equivalent to uniform convergence on bounded subsets. Hence, the topology on $V^\wedge$ is the topology of uniform convergence on bounded subsets, i.e.~the strong topology.
\end{proof}

\subsection{Quotients and epimorphisms of linear $k$-Spaces}

\subsubsection{Epimorphisms of linear $k$-Hausdorff spaces are morphisms with dense image.} We have the following characterisation of epimorphisms in the category of linear $hk$-spaces, which we will need at a later point.

\begin{lem}\label{epis_in_lin_k_hausdorff_spaces}
    Let $d:V \to W$ be a continuous linear map between linear $hk$-spaces $V,W$. Then $d$ has dense image if, and only if, it is an epimorphism in $\hvect$, i.e.~for every pair of maps $f,g: W \to Z$ from $W$ to some further $k$-Hausdorff linear $k$-space $Z$, $f \circ d = g \circ d$ implies $f=g$.
\end{lem}
\begin{proof}
    For the ``if'' direction, take $Z= W/\overline{\text{im }d}$ to be the quotient of $W$ by the closed image of $d$, $f:W\to Z$ to be the projection and $g:W\to Z$ to be the zero map. Then $f \circ d = 0 = g \circ d$, and hence, by assumption $f=g=0$. Since $f$ is the projection onto $Z=W/\overline{\text{im }d}$ this implies that $\overline{\text{im }d}=Z$, so $d$ does indeed have dense image. \par 
    For the ``only if'' direction, let $h := f-g$. Then what we want to show is (equivalently) that $h \circ d = 0$ implies $h = 0$. So suppose that $h \circ d = 0$. Then $\text{im }d \subseteq \ker h$. Since $Z$ is $k$-Hausdorff, $\ker h = h^{-1}(\{0\})\subseteq W$ is closed. Using our assumption that $\overline{\text{im }d} = W$, this implies that 
        $$ W = \overline{\text{im }d} \subseteq \overline{\ker h} = \ker h \subseteq W. $$
    Therefore, $\ker h = W$ and thus, $h=0$, which is what we wanted to show.
\end{proof}

A direct and useful consequence is:

\begin{lem}\label{linear_map_vanishing_on_dense_subspace}
    Let $V, W$ be linear $hk$-spaces, let $f:V\to W$ be a continuous linear map and suppose that $f$ vanishes on some dense linear subspace $D\subseteq V$. Then $f=0$.
\end{lem}

\subsubsection{Quotients of linear $k$-spaces} The purpose of this section is to define a notion of quotient for linear $k$-spaces and to show that, similar to the case of topological vector spaces, quotient maps of linear $k$-spaces are always open maps (\cref{lem_quot_maps_open}).

\begin{defn}[Quotient linear $k$-spaces]\label{lin_quot_spaces}
    Let $V$ be a linear $k$-space.
    \begin{enumerate}[1.]
        \item Let $W\subseteq V$ be a linear $k$-subspace. We define the quotient of $V$ by $W$, written $V/W$, by endowing the (algebraic) quotient vector space with the quotient topology with respect to the projection $V\to V/W$.
        \item A \emph{linear quotient map} is a continuous linear map $p:V\to Z$ (with $Z$ some linear $k$-space) that is also a (topological) quotient map.  
    \end{enumerate}
\end{defn}

\begin{lem}
    The quotient $V/W$ of a linear $k$-space $V$ by a linear $k$-subspace $W$ is again a linear $k$-space with respect to the induced vector space structure.
\end{lem}
\begin{proof}
    It is clear that $V/W$ is a $k$-space, since (topological) quotients of $k$-spaces are $k$-spaces. \par 
    Next, let us show the continuity of 
        $$ +_{V/W} : V/W \times V/W \to V/W.$$
    $W\times W$ is contained in the kernel of the composite  
    \[\begin{tikzcd}
	{V\times V} & V & {V/W}.
	\arrow["{+_V}", from=1-1, to=1-2]
	\arrow["{p}", from=1-2, to=1-3]
    \end{tikzcd}\]
    Hence, using the universal property of the (topological) quotient, this composite factors through the projection $V\times V \to (V\times V)/(W\times W)$, yielding the map
        $$ f: (V\times V)/(W\times W) \to V/W, \;\;\; (x, y) + W\times W \mapsto x+y+W. $$
    Moreover, the map
        $$ i:(V\times V)/(W\times W) \to V/W \times V/W, \;\;\; (x,y)+ W\times W \mapsto (x+W, y+W) $$
    is a homeomorphism, since products (in $\khaus$) of quotients maps are quotient maps are again quotient maps (see \cref{khausdorff_quots}. Now, we have that $+_{V/W}=f\circ i^{-1}$ and therefore, it is continuous. A similar argument shows the continuity of scalar multiplication.
\end{proof}

Linear quotient maps, as introduced in \cref{lin_quot_spaces}, satisfy the expected universal property:

\begin{lem}[Universal property of linear quotient maps]\label{univ_prop_lin_quots}
    Let $V,W$ be linear $k$-spaces and let $p:V\to W$ be a linear quotient map. Suppose that $f:V\to Z$ is some further continuous map (with $Z$ some linear $k$-space) with the property that $\ker p \subseteq \ker f$, as illustrated in the commutative diagram: 
    \[\begin{tikzcd}
	{\ker p} & V & W \\
	& Z
	\arrow[from=1-2, to=2-2]
	\arrow["p"', from=1-2, to=1-3]
	\arrow[hook, from=1-1, to=1-2]
	\arrow["0"{description}, from=1-1, to=2-2]
    \end{tikzcd}\]
    Then there is a unique map $\overline{f}$ making the following diagram commute:
    \[\begin{tikzcd}
	{\ker p} & V & W \\
	& Z
	\arrow[from=1-2, to=2-2]
	\arrow["p", from=1-2, to=1-3]
	\arrow[hook, from=1-1, to=1-2]
	\arrow["0"{description}, from=1-1, to=2-2]
	\arrow["{\overline{f}}"', dashed, from=2-2, to=1-3]
    \end{tikzcd}\]
\end{lem}
\begin{proof}
    The universal property of quotient maps of topological spaces guarantees the unique existence of $\overline{f}$ as a continuous map, while the universal property of surjective linear maps between (algebraic) vector spaces guarantees unique existence of $\overline{f}$ as a linear map. Combining both therefore proves the claim.
\end{proof}

This also shows that linear quotients maps are always, up to a unique isomorphism, the projection onto some quotient linear $k$-space, as defined in \cref{lin_quot_spaces}:

\begin{cor}\label{iso_thm_linksp}
    Let $p:V \to W$ be a linear quotient map. Then there is a unique isomorphism $i$ making the following diagram commute: 
    \[\begin{tikzcd}
	V & W \\
	{V/\ker p}
	\arrow["q", two heads, from=1-1, to=2-1]
	\arrow["p", from=1-1, to=1-2]
	\arrow["i"', dashed, from=2-1, to=1-2]
    \end{tikzcd}\]
\end{cor}
\begin{proof}
    This follows from applying the universal property of linear quotient maps \linebreak (\cref{univ_prop_lin_quots}) twice, first using that $q$ is a linear quotient map and $\ker q \subseteq \ker p$, and then using that $p$ is also a linear quotient map and $\ker p \subseteq \ker q$.
\end{proof}

One important feature of linear quotient maps that distinguishes them from general quotient maps is that they are always open:

\begin{lem}\label{lem_quot_maps_open}
    Let $V,W$ be linear $k$-spaces and let $p:V\to W$ be a surjective continuous linear map. Then $p$ is a linear quotient map if, and only if, it is an open map.
\end{lem}
\begin{proof}
    Open continuous surjective maps are always quotient maps. For the other direction, by \cref{iso_thm_linksp}, we may assume that $W=V/Z$ for some linear $k$-subspace $Z\subseteq V$ and that $p:V\to V/Z$ is the canonical projection. Now, if $U\subseteq V$ is open, then 
        $$ p(U) = U + W = \bigcup_{w\in W} U + w. $$
    Since translation $(\cdot + x)$ is a homeomorphism for any $x\in V$ (with continuous inverse $(\cdot - x)$), this shows that $p(U)$ is open, as claimed.
\end{proof}

\subsection{Direct sums and cocompleteness of linear $hk$-spaces } 

The goal of this section is to construct coproducts of linear $hk$-spaces, which will entail that $\vect$ has all colimits. We will first give a direct construction and then show that this satisfies the required universal property in \cref{coprods_in_vect}.

\begin{construction}[Direct sum of linear $k$-spaces]\label{cref_direct_sums_of_lin_hk_sp}
    \text{}
    \begin{enumerate}
        \item For a \emph{finite} family $(V_i)_{i\in \{1,\dots,n\}}$ of linear $k$-spaces, define its \emph{direct sum} as the product
            $$ \bigoplus_{i=1}^{n} V_i := \prod_{i=1}^n V_i. $$
        \item For an \emph{arbitrary} family $(V_i)_{i\in I}$ of linear $k$-spaces, define its direct sum as the linear $k$-space given by the filtered colimit (in $\spaces$) over the direct sum of finite subfamilies:
            $$ \bigoplus_{i\in I} V_i := \colim_{J\subseteq I\text{ finite}} \bigoplus_{j\in J} V_j$$
    \end{enumerate} 
\end{construction}

\begin{remark}
    The underlying set of the direct sum of a family of linear $k$-spaces is the (underlying set of the) algebraic direct sum (since the forgetful functor from $\spaces$ to $\sets$ preserves colimits). The addition and scalar multiplication maps obtained this way are continuous, so the direct sum is indeed a linear $k$-space, as suggested.
\end{remark}

\begin{prop}\label{coprods_in_vect}
    The direct sum of linear $k$-spaces from \cref{cref_direct_sums_of_lin_hk_sp} is the coproduct in $\vect$.
\end{prop}
\begin{proof}
    Let $(V_i)_{i\in I}$ be a family of linear $k$-spaces. We need to verify that its direct sum verifies the following universal property: for any family of continuous linear maps
        $$ f_i : V_i \to W, \qquad(i\in I),$$
    with $W$ some further linear $k$-space, there is a unique continuous linear map 
        $$ \oplus_{i\in I} f_i: \bigoplus_{i\in i} V_i \to W, $$
    making the following diagram commute, for each $j\in I$,
    \[\begin{tikzcd}
	{V_j} & W \\
	{\bigoplus_{i\in I} V_i}
	\arrow["{f_j}", from=1-1, to=1-2]
	\arrow[hook, from=1-1, to=2-1]
	\arrow["{\bigoplus_{i\in I}f_i}"', from=2-1, to=1-2]
    \end{tikzcd}\]
    The unique existence of a linear (not necessarily continuous) such map is guaranteed by the universal property of the direct sum (coproduct) of algebraic vector spaces. Its continuity then follows from the fact that, in order for a map on a filtered colimit limit of $k$-spaces to be continuous, it suffices that its composition with each inclusion is continuous.
\end{proof}

\begin{cor}\label{vect_cocomplete_complete}
    $\vect$ is a complete and cocomplete pre-abelian category.
\end{cor}
\begin{proof}
    We have seen that $\vect$ is complete (\cref{lin_hk_spaces_complete}). Cocompleteness follows from the existence of arbitrary coproducts (\cref{coprods_in_vect}) and coequalisers of arbitrary pairs $f,g:V\to W$ of morphisms (which are given by the cokernel of $f-g$, i.e. the quotient $W/\text{im}(f-g)$). Finally, that $\vect$ is a pre-abelian category follows now from the fact that finite products and coproducts coincide, which is how we have constructed them. 
\end{proof}

\section{Smith Duality}\label{sec_smithduality}

An intriguing phenomenon concerning the natural dual space of a Fréchet space is what we refer to as \emph{Smith duality}, first observed in the context of Banach spaces by Smith in \cite{smith1952pontrjagin}. Let us write 
    $$ V^{\wedge\wedge} := (V^\wedge)^\wedge $$
for the ``natural double dual'' of a linear $k$-space $V$. We then have:

\begin{thm}[Smith Duality Theorem]\label{Smith_duality}
    Let $V$ be a Fréchet space, \emph{or} a Brauner space. Then the canonical evaluation map 
        $$\eta: V \to V^{\wedge\wedge}, \;\;\; x\mapsto (\phi \mapsto \phi(x))$$
    is an isomorphism (i.e.~a linear homeomorphism). 
\end{thm}
\begin{proof} 
    That $\eta$ is continuous is a consequence of cartesian closure of $\ktop$. By \cref{prop_dual_fre_carries_co_topology} and \cref{prop_dual_of_bra_fre}, $V^\wedge$ carries the compact-open topology which, by completeness of $V$, also coincides with the topology of uniform convergence on absolutely convex compact subsets (since, in complete spaces, the closed absolutely convex hull of a compact set is compact). That $\eta$ is a bijection then follows from the Mackey-Arens theorem (for a statement of the Mackey-Arens and the necessary preliminaries, see \cite[p. 131, 3.2 Theorem]{schaefer1971topological}). \par 
    It remains to show that $\eta$ is also an open map. Let $U\subseteq V$ be an open neighbourhood of the origin. Since $V$ is locally convex, we may suppose that $U$ is closed and absolutely convex (meaning that ). Then the \emph{polar} of $U$, 
        $$ U^\circ := \{\phi \in V^\wedge \mid \forall x \in U: |\phi(x)|\leq 1 \},$$
    is equicontinuous and hence compact (by the Arzelà-Ascoli theorem). Now, the \emph{double polar} (or, \emph{bipolar}), 
        $$ (U^\circ)^\circ := \{\chi \in V^{\wedge\wedge} \mid \forall \phi \in U^\circ: |\chi(\phi)|\leq 1 \},$$
    is a neighbourhood of the origin in the compact-open topology of $V^{\wedge\wedge}$ and therefore in $V^{\wedge\wedge}$. The \emph{bipolar theorem} (see \cite[p. 126, 1.5 Theorem]{schaefer1971topological}) now says precisely that 
        $$ (U^\circ)^\circ = \eta(U), $$
    so the image $\eta(U)$ of $U$ under $\eta$ is open. Since $U$ was arbitrary, $\eta$ is an open map, which is what we wanted to show.
\end{proof}

\begin{remark}
    The original statement and proof of the duality between Fréchet and Brauner spaces is the subject of \cite{brauner1973duals}. Our proof is essentially an adaption of the original proof of the Banach space case in \cite{smith1952pontrjagin} (which can even be adapted to classes of locally convex topological vector spaces more general than Fréchet spaces, see, for example, \cite{waterhouse1968dual}).
\end{remark}

\begin{remark}
    In light of \cref{rem_qcb_canonical} and \cref{L_VW_QCB}, when $V$ is a separable, even if we had chosen to work over a different (reasonable) category of topological spaces, $V^\wedge$ would still carry the \emph{same topology}. In other words, for separable Fréchet spaces, Smith duality does not depend on the choice of base category. From this perspective, Smith duality is a particularly natural phenomenon.  
\end{remark}

\begin{remark}
    An informal interpretation of the Smith Duality Theorem might suggest the paradoxical consequence that ``all Banach spaces are reflexive'' (paradoxical because Banach spaces are typically not reflexive, in the established sense of the word ``reflexive''). This is no contradiction, of course, as the topology on the natural dual of a Banach space (in the sense of \cref{defn_natural_dual}) is weaker than the norm topology (and always strictly so in the infinite-dimensional case).
\end{remark}

\begin{remark}
    The proof of \cref{Smith_duality} uses a large part of the canon of theorems of locally convex functional analysis: the uniform bounded principle and the Banach-Dieudonné theorem (both in form of \cref{prop_dual_fre_carries_co_topology}), as well as the Mackey-Arens theorem (which in turn uses the Alaoglu and Hahn-Banach theorems). Considering that the formulation of \cref{Smith_duality} takes the perspective of linear $k$-spaces, this seems like a detour, and raises the question of whether there is a proof more close to the language of linear $k$-spaces. While we will not pursue this question here, we remark that a result in this direction is the proof of (the Banach space case of) Smith duality in the language of \emph{condensed mathematics} \cite[p. 21, Theorem 4.2]{scholze2019condensed}. 
\end{remark}

\begin{remark}
    As a final remark, we point out that Smith duality is strongly related to, and can be viewed as an instance of, \emph{Pontryagin duality}. This was Smith's original motivation (see \cite{smith1952pontrjagin} for any further details). 
\end{remark}

\subsubsection{As an (anti-)equivalence of categories} We can also formulate Smith duality as an equivalence of categories.

\begin{cor}
    The natural dual space functor, 
        $$ (-)^\wedge: \mathsf{Fre}^\op \to \mathsf{Bra}, \;\; V \mapsto V^\wedge, \;f\mapsto f^\wedge, $$
    where $\mathsf{Fre}$ and $\mathsf{Bra}$ are the full subcategories of $\hvect$ spanned by Smith and Brauner spaces, respectively, induces an equivalence of categories.
\end{cor}
\begin{proof}
    This follows directly from \cref{Smith_duality} together with the fact that $\eta$ defines two natural transformations,
        $$ \id_{\mathsf{Fre^\op}} \to (-)^{\wedge\wedge}, \; \id_{\mathsf{Bra}} \to (-)^{\wedge\wedge}.$$
    (The proof of naturality of $\eta$ is the same as for algebraic vector spaces.)
\end{proof}

%% file: Chapters/3_Monads_Tensor_Products_and_Duality_in_Closed_Symmetric_Monoidal_Categories.tex
\chapter{Monads, Tensor Products and Duality}

This chapter is divided into three parts. The first is a brief introduction to monads, including a number of examples specific to our cause. The second part, \cref{sec_monoidal_cats_intro}, will similarly deal with the basics of the theory of (closed, symmetric) monoidal categories. All this is done in service to the third part, \cref{sec_duality_sep_chu}, where we will introduce \emph{separated objects}, a general notion of an object ``embedded into its double dual''. We show how this is related to the \emph{extensional-separated Chu construction} of \cite{barr1998separated} and apply both notions to the category $\vect$ of linear $hk$-spaces. The separated objects in $\vect$ are the \emph{replete} linear $hk$-spaces, and the separated-extensional Chu construction leads to the notion of \emph{paired linear $hk$-space}. \par 
While the first two sections of this Chapter have a purely expository purpose, the developments concerning the \emph{separated objects} of \cref{sec_triple_dualisation_lem_and_sep_obj} are original, as is the relation to the separated-extensional Chu construction given in \cref{sepobjs_corefl_subcat_of_little_chu}, the construction of free objects for the latter category (see \cref{construction_thm_for_star_aut_cats_over_cccs}), and the application of these results to $hk$-spaces. This being said, the results of \cref{sec_triple_dualisation_lem_and_sep_obj} concerning \emph{separated objects} can be seen as special cases of more general results related to the \emph{idempotent monad associated to an arbitrary monad}, applied to the \emph{double dualisation monad}, see \cite{fakir1970monade} as well as \cite{lucyshyn2014completion} for a generalisation and an overview of related results, with functional-analytic examples. However, we do not use these general results, significantly simplifying the development of these ideas for the particular case that we are interested in, thus making it more accessible and streamlined to our cause.

\section{Monads}

\subsection{What is a monad?}

Conceptually, there are at least two points of view on what a monad is. According to the first, monads provide a ``consistent way of extending spaces to include [generalised] elements and [generalised] functions of a specific kind'' \cite[p.~132]{perrone2019notes}.
In this way, they allow us to model (various sorts of) non-determinism, as illustrated by the example of probability monads, mentioned in the introduction (see \cref{sec_prob_monads_on_cccs_of_spaces}, as well as \cite[Section 5.1]{perrone2019notes} for a detailed explanation). A second perspective is that monads are ``generalised algebraic theories'', the reason for which should become apparent by the end of this section (see also \cite[Section 5.2]{perrone2019notes} for an elementary exposition). With these ideas in mind, we turn to the definition:

\begin{defn}[Monad]\label{monad_defn}
    Let $\mathsf{C}$ be a category. A \emph{monad} $(T, i, m)$ on $\mathsf{C}$ is 
    \begin{enumerate}
        \item an endofunctor, 
            $$T:\mathsf{C}\to \mathsf{C}, $$ 
            together with 
        \item two natural transformations, the \emph{unit},
            $$ i: \mathsf{id}_{\mathsf{C}} \to T, $$
            and the \emph{multiplication},
            $$ m: T\circ T \to T, $$
            such that
        \item the \emph{monad laws} hold, meaning that following two diagrams, the \emph{unit law},
        \[\begin{tikzcd}
	       TX & TTX \\
	       TTX & TX
	       \arrow["{i^{TX}}", from=1-1, to=1-2]
	       \arrow["{\;\;\;m^{TX}}", from=1-2, to=2-2]
	       \arrow[Rightarrow, no head, from=1-1, to=2-2]
	       \arrow["{T(i^{X})\:\:\:}"', from=1-1, to=2-1]
	       \arrow["{m^{TX}\;}"', from=2-1, to=2-2]
        \end{tikzcd}\] 
        and the \emph{associativity law},
        \[\begin{tikzcd}
	           TTTX && TTX \\
	           TTX && TX
	           \arrow["{m^{TTX}\;\;\;\;}"', from=1-1, to=2-1]
	           \arrow["{T(m^{TX})}", from=1-1, to=1-3]
	           \arrow["{m^{TX}}"', from=2-1, to=2-3]
	           \arrow["{\;\;\;m^{TX}}", from=1-3, to=2-3]
        \end{tikzcd}\]
        commute, where $X$ is any object from $\mathsf{C}$ and $m^{TY}: TTY \to TY$, $i^Y: Y\to TY$ are the components at the object $Y$ of the natural transformations $m, i$. 
    \end{enumerate}
\end{defn} 

\begin{remark}
    The reason for the terminology ``unit'' and ``multiplication'' is that monads can be considered a special case of the concept of a monoid in a monoidal category (see \cite[p.~154]{riehl2017category}). The unit and associativity laws are precisely the corresponding laws for monoid, under this interpretation.
\end{remark}

\subsection{Examples of monads}\label{exmpls_of_monads}

\subsubsection{Adjunctions induce monads} From an adjunction one always obtains a monad in the following way.

\begin{example}\label{adj_give_rise_to_monads}
    Let $\mathsf{C},\mathsf{D}$ be categories and suppose that the functor 
        $$ F: \mathsf{C}\to \mathsf{D} $$
    is left adjoint to 
        $$ G: \mathsf{D} \to \mathsf{C}, $$
    i.e. we have a natural isomorphism 
        $$ \phi^{X,Y}: \Hom_{\mathsf{D}}(FX, Y) \cong \Hom_{\mathsf{C}}(X, GY). \qquad (X\in \mathsf{C}, Y\in \mathsf{D})$$
    From this, we obtain two natural transformations, 
        $$ i := \phi^{X,FX}(\mathsf{id}_{FX}) : X \to GFX, \qquad (X \in \mathsf{C}) $$
    and,
        $$ m := G((\phi^{GFX, FX})^{-1}(\mathsf{id}_{GFX})): GFGFX \to GFX. \qquad (X \in \mathsf{C})$$
    Then $(G\circ F, i, m)$ is a monad on $\mathsf{C}$ (see \cite[p. 138f]{mac2013categories}).
\end{example}

\begin{remark}
    As we will see, \emph{every} monad arises from an adjunction in this way (see \cref{free_T_modules}).
\end{remark}

\subsubsection{Monads induced from free-forgetful adjunctions} The following example is conceptually important, illustrating the role that monads play in the category-theoretic approach to universal algebra (which in turn motivates the definition of the Eilenberg-Moore category, 
see \cref{t_modules_as_generalised_algebras}). 

\begin{example}[Set-of-terms monads]\label{expl_algebraic_cats_and_monads}
     Let $\mathsf{A}$ be some category of algebraic structures, such as the category $\monoids$ of monoids, the category $\groups$ of groups, the category $\rings$ of rings, or the category $R\modules$ of $R$-modules  over some ring $R$. (The notion of ``category of algebraic stuctures'' can be made precise using the language of \emph{Lawvere theories}, see \cite{hyland2007category}.) In any such case, the forgetful functor $|\cdot|:\mathsf{A}\to \sets$ has a left adjoint $T_\mathsf{A}$ (the ``free $\mathsf{A}$-object on a set''). Given a set $X$, $|T_\mathsf{A}(X)|$ may be thought of as the set of terms built freely from elements of $X$ (serving as ``abstract variables'') and the given algebraic operations. For example, $|T_\monoids(X)|$ is the free monoid on $X$ (equivalently, the set of lists with entries from $X$, or the set of \emph{strings} from the \emph{alphabet} $X$); $|T_\groups(X)|$ is the set of (group theoretic) \emph{words} on $X$ (i.e.~the free group on $X$); $|T_\rings(X)| \cong \mathbb{Z}[X]$ is the polynomial ring on $X$ over $\mathbb{Z}$; and $|T_{R\modules}(X)|$ is the set of finite formal $R$-linear combinations of elements of $X$. \par 
    Since the endofunctor $|T_\mathsf{A}(-)|: \sets \to \sets $ is the composite of a functor with its left adjoint, by \cref{adj_give_rise_to_monads}, it carries the structure of a monad, the \emph{set-of-terms} monad associated to $\mathsf{A}$. In this case, the unit corresponds to the inclusion $X\to |T_\mathsf{A}(X)|$ of single-variable terms, and the multiplication $|T_\mathsf{A}(|T_\mathsf{A}(X)|)|\to |T_\mathsf{A}(X)|$ corresponds to a formal application of the given algebraic operations to terms.
\end{example}

The next example will play a crucial role later on. For example, it will be used to show the existence of a good notion of tensor product of linear $hk$-spaces.

\begin{example}[The free-linear-$hk$-space monad]
    By \cref{free_lin_hk_sp_existence}, the forgetful functor 
        $$ \vect \to \spaces $$
    has a left adjoint, the free-linear-$hk$-space functor,
        $$ F_{\mathbb{K}}: \spaces \to \vect. $$
    The monad on $\spaces$ arising from this adjunction will, by a very slight abuse of notation, also be denoted by $F_{\mathbb{K}}$.
\end{example}

A final example of a monad arising from a free-forgetful adjunction demonstrates that monads also come up naturally in situations which one might not \emph{a priori} judge as being of an algebraic nature.

\begin{example}[Ultrafilter monad]
    Let $|-|: \compacta \to \sets $ be the forgetful functor from the category of compact Hausdorff spaces to the category of sets. This functor has a left adjoint (the \emph{Stone-Čech compactification}), yielding a monad with underlying endofunctor 
        $$ \beta: \sets \to \sets, $$
    such that $\beta X$ (with $X$ any set) can be identified with the set of ultrafilters on $X$ (see \cite[p. 157, section 9]{mac2013categories}).
\end{example}

\subsubsection{Probability monads} As mentioned in the introduction, probability monads, being both rich in structure and applications, are one of our motivations for studying interactions between measure theory, functional analysis and monads (see also \cref{sec_prob_monads_on_cccs_of_spaces} of the introduction). 

\begin{example}[The Giry monad on Polish spaces]
    The Giry monad on the category $\mathsf{Pol}$ of Polish (i.e.~separable, completely metrisable) spaces (with continuous maps as morphisms) provides an example of a probability monad (see \cref{ex_giry_monad} for details).
\end{example}

\begin{example}[Radon monad, \cite{swirszcz1974monadic}]
    Let $\compacta$ be the category of compact Hausdorff spaces. For a compact Hausdorff space $X$, let $\mathcal{P}_r(X)$ be the space of Radon probability measures on $X$, with the topology of weak convergence (which is a compact Hausdorff topology). Then the resulting functor
        $$ \mathcal{P}_r: \compacta \to \compacta$$
    may be equipped with the structure of a monad (in exactly the same way as for the Giry monad on $\mathsf{Pol}$, see \cite[Proposition 2.5]{keimel2008monad} for a detailed construction).
\end{example}

In Chapter 4, we will construct a probability monad $\mathcal{P}$ on $\spaces$ which restricts to both the Giry monad on Polish spaces and the Radon monad on compact Hausdorff spaces, thus in some sense interpolating between the two.

\subsection{Modules over a monad (the Eilenberg-Moore category)}\label{sec_eilernberg_moore_cat} In the interpretation of monads as ``generalised algebraic theories'', the \emph{Eilenberg-Moore category} is the category of ``generalised algebraic structures'' that the given monad describes. 

\begin{defn}[Modules over a monad, Eilenberg-Moore category]\label{modules_over_monad_def}
    Let $\mathsf{C}$ be a category and let $(T, i, m)$ be a monad on $\mathsf{C}$. 
    \begin{enumerate}
        \item A $T$-module is an object $A$ of $\mathsf{C}$ together with a $\mathsf{C}$-morphism (``structure morphism''),
            $$ t^A : TA \to A, $$
        such that the following two diagrams,
        \[\begin{tikzcd}
	       A && {TA} && {TTA} && {TA} \\
	       \\
	       && A && {TA} && A
	       \arrow["{T(t^A)\;}"', from=1-5, to=3-5]
	       \arrow["{t^A}"', from=3-5, to=3-7]
	       \arrow["{t^A}", from=1-7, to=3-7]
	       \arrow["{m}", from=1-5, to=1-7]
	       \arrow["{i}", from=1-1, to=1-3]
	       \arrow["{\text{id}_A}"', from=1-1, to=3-3]
	       \arrow["{t^{A}}", from=1-3, to=3-3]
        \end{tikzcd}\]
        commute.
        \item Let $A$ and $B$ be $T$-modules (with structure maps $t^A, t^B$, which we will subsequently sometimes leave implicit). A \emph{morphism of $T$-modules} is a $\mathsf{C}$-morphism 
            $$ f: A \to B$$
        such that 
            $$f \circ t^A = t^B \circ Tf, $$
        i.e.~the following diagram commutes:
        \[\begin{tikzcd}
	    TA & TB \\
	    A & B
	    \arrow["Tf", from=1-1, to=1-2]
	    \arrow["f", from=2-1, to=2-2]
	    \arrow["{t^A}"', from=1-1, to=2-1]
	    \arrow["{t^B}", from=1-2, to=2-2]
        \end{tikzcd}\]
        \item We denote the category of $T$-modules (with morphisms of $T$-modules as morphisms), also called the \emph{Eilenberg-Moore category} of $T$, by $\mathsf{C}^T$.
    \end{enumerate}
\end{defn}

\begin{remark}\label{t_modules_as_generalised_algebras}
    $T$-modules are also known as ``algebras for the monad $T\:$'' or simply ``$T$-algebras''. This terminology comes from the usage of monads for the purpose of universal algebra. In this context, the term ``algebra'' is not reserved for ``algebras over a field (or ring)'', but is used for any kind of algebraic structure, so viewing monads as generalised algebraic theories renders this terminology quite coherent. However, since vector space play a dominant role in our context (i.e.~our context is closer to \emph{linear} as opposed to universal algebra), the term ``algebra'' referring to something other than an algebra over a field or ring may be confusing, which is why we will adhere to the name ``$T$-module''. 
\end{remark}

\subsubsection{Examples of categories of modules over a monad.} Modules over a monad generalise many common kinds of mathematical structures. Let us see what the Eilenberg-Moore category of some of the exemplary monads given in \cref{exmpls_of_monads} is.

\begin{enumerate}[1.]
    \item If $T:= |T_\mathsf{A}(-)|$ is a the set-of-terms monad associated to category $\mathsf{A}$ of algebraic structures (see \cref{expl_algebraic_cats_and_monads}), a $T$-module is equivalently just an object of $\mathsf{A}$: the Eilenberg-Moore category of $T$ is canonically equivalent to $\mathsf{A}$ itself. For example, a module over the free-group monad on $\sets$ is equivalently a group, and similarly, a module over the free-commutative-ring monad $\mathbb{Z}[-]$ is equivalently a commutative ring. 
    \item Analogously to the previous point, an $F_{\mathbb{K}}$-module (where $F_{\mathbb{K}}$ is the free-linear-$hk$-space monad) is equivalently a linear $hk$-space.
    \item A module over the ultrafilter monad $\beta$ on $\sets$ is equivalently a compact Hausdorff space  (see \cite[p. 157, Section 9, Theorem 1]{mac2013categories}). This explains why compact Hausdorff spaces behave categorically a lot like algebraic structures (e.g.~in that a morphism of compact Hausdorff spaces is an isomorphism if, and only if, it is bijective). 
    \item Modules over the Radon monad are equivalently compact convex subsets of locally convex Hausdorff topological vector spaces, with morphisms of modules corresponding to affine maps between compact convex sets (this is \cite[Theorem 3]{swirszcz1974monadic}).
\end{enumerate}

\subsubsection{The free module over a monad.}\label{free_T_modules} It turns out that \emph{every} monad arises from a ``free-forgetful'' adjunction as described in \cref{adj_give_rise_to_monads}. To see this, we first observe that $TX$ always admits an obvious $T$-module structure.

\begin{defn}[Free $T$-module]
    Let $\mathsf{C}$ be a category and let $(T, i, m)$ be a monad on $\mathsf{C}$. Given an object $X$ of $\mathsf{C}$, endow $TX$ with the structure of a $T$-module by defining the structure map,
        $$ t^{TX}: TTX \to TX, $$
    to be $t^{TX}:= m$ (the multiplication). (That this does indeed define a $T$-module follows directly from the monad laws.) With this $T$-module structure, we call $TX$ the \emph{free $T$-module on $X$}.
\end{defn}

The name ``free $T$-module'' is justified by the following universal property, which also shows the aforementioned claim that every monad arises from a free-forgetful adjunction.

\begin{prop}[Universal property of free $T$-modules]\label{univ_prop_free_T_modules}
    Let $\mathsf{C}$ be a category, let $(T, i, m)$ be a monad on $\mathsf{C}$ and let $X$ be an object of $\mathsf{C}$. Then for any $\mathsf{C}$-morphism,
        $$ f:X\to A, $$
    from $X$ to some $T$-module $A$, there is a unique morphism of $T$-modules,
        $$ \overline{f}: TX \to A, $$
    such that the following diagram commutes:
    \[\begin{tikzcd}
	X & A \\
	TX
	\arrow["f", from=1-1, to=1-2]
	\arrow["i"', from=1-1, to=2-1]
	\arrow["{\overline{f}}"', dashed, from=2-1, to=1-2]
    \end{tikzcd}\]
    In other words, we have an adjunction,
        $$ \Hom_\mathsf{C}(X, |A|) \cong \Hom_{\mathsf{C}^T}(TX, A), \qquad (X \in \mathsf{C}, A \in \mathsf{C}^T)$$
    where $|-|:\mathsf{C}^T\to \mathsf{C}$ denotes the forgetful functor.
\end{prop}
\begin{proof}
    Take $\overline{f} := t^A \circ Tf$. This makes the desired diagram commute and thus shows existence. For uniqueness, notice that $\overline{f}$ is necessarily of this form: being a morphism of $T$-modules,
        $$ \overline{f} \circ m = t^A \circ T\overline{f}. $$
    But then, 
        $$ \overline{f} \circ m \circ T i = t^A \circ T \overline{f} \circ T i.$$
    Using the unit law, the left hand side is just $\overline{f}$, while the right hand side equals $t^A \circ Tf$, by functoriality of $T$ and commutativity of the diagram. Thus, $\overline{f} := t^A \circ Tf$, showing uniqueness. 
\end{proof}

\section{Tensor Products: Monoidal Categories}\label{sec_monoidal_cats_intro}

What does it mean for a category to admit a ``good notion of tensor product''? And how does one construct such tensor products? Questions like these are naturally treated using the language of \emph{monoidal categories} (and their relatives).

\subsection{Monoidal categories}

\begin{defn}
    A \emph{monoidal category} is a category $\mathsf{C}$ together with 
    \begin{enumerate}[1.]
        \item a functor,
            $$ \otimes_{\mathsf{C}}: \mathsf{C} \times \mathsf{C} \to \mathsf{C}, $$
        called the \emph{tensor product}, 
        \item an object $I_{\mathsf{C}}$, called the \emph{tensor unit},
        \item a natural isomorphism,
            $$ \alpha^{A,B,C}: (A \otimes_{\mathsf{C}} B) \otimes_{\mathsf{C}} C \cong A \otimes_{\mathsf{C}} (B \otimes_{\mathsf{C}} C), \;\;\;\; (A,B,C \in \mathsf{C} $$
        called the \emph{associator},
        \item two further natural isomorphisms,
            $$ \lambda^A: I_{\mathsf{C}} \otimes_{\mathsf{C}} A \cong A, \;\; \rho^A : A \otimes{\mathsf{C}} I_{\mathsf{C}} \cong A, \;\;\; (A \in \mathsf{C})$$
        called the left and right \emph{unitors}, respectively, 
    \end{enumerate}
    subject to the commutativity of the following two diagrams, the \emph{coherence conditions}:
    \begin{enumerate}[a.]
        \item The triangle identity: 
        \[\begin{tikzcd}
	       {(A \otimes_{\mathsf{C}} I_{\mathsf{C}}) \otimes B} && {A \otimes_{\mathsf{C}} (I_{\mathsf{C}} \otimes B)} \\
	       & {A \otimes_{\mathsf{C}} B}
	       \arrow["{\alpha^{A,I_{\mathsf{C}},B}}", from=1-1, to=1-3]
	       \arrow["{\rho^A\otimes_{\mathsf{C}} I_B}"', from=1-1, to=2-2]
	       \arrow["{I_A \otimes_{\mathsf{C}} \lambda_B}", from=1-3, to=2-2]
        \end{tikzcd}\]
        \item The pentagon identity:
    \[\begin{tikzcd}
	& {(A \otimes_{\mathsf{C}} B) \otimes_{\mathsf{C}}(C \otimes_{\mathsf{C}}D)} \\
	{(((A \otimes_{\mathsf{C}} B) \otimes_{\mathsf{C}}C) \otimes_{\mathsf{C}}D)} && {A \otimes_{\mathsf{C}} (B \otimes_{\mathsf{C}}(C \otimes_{\mathsf{C}}D))} \\
	\\
	{(A \otimes_{\mathsf{C}} (B \otimes_{\mathsf{C}}C)) \otimes_{\mathsf{C}}D} && {A \otimes_{\mathsf{C}} ((B \otimes_{\mathsf{C}}C) \otimes_{\mathsf{C}}D)}
	\arrow["{\alpha^{A\otimes_{\mathsf{C}}B, C,D}}", from=2-1, to=1-2]
	\arrow["{\alpha^{A,B,C\otimes_{\mathsf{C}} D}}", from=1-2, to=2-3]
	\arrow["{\alpha^{A, B\otimes_{\mathsf{C}}C, D}}"', from=4-1, to=4-3]
	\arrow["{\alpha^{A,B,C}\otimes_{\mathsf{C}} D}"', from=2-1, to=4-1]
	\arrow["{A\otimes_{\mathsf{C}}\alpha^{B,C,D}}"', from=4-3, to=2-3]
    \end{tikzcd}\]
    \end{enumerate}
\end{defn}

\begin{warning}
    Despite being a crucial part of the structure of a monoidal category, we will sometimes leave the associator and unitors implicit, as they are often easily inferred from context. For example, we will use (common) phrases like ``let $(\mathsf{C}, \otimes, I)$ be a monoidal category'', specifying the associator and unitors only when they become relevant. This is (partially) justified by \emph{Mac Lane's Coherence Theorem} (see \cite[p. 165-170]{mac2013categories}) which is also the motivation for the otherwise perhaps somewhat obscure coherence conditions in the definition of a monoidal category.
\end{warning}

\begin{example}
    Let $\mathsf{C}$ be a category with finite products. Then $\mathsf{C}$ is a monoidal category with respect to the cartesian product $\times$ as the tensor product, the initial object $1$ as tensor unit, and the associator and unitors given in the evident way. 
\end{example}

\begin{example}
    Let $R$ be a commutative ring and let $R\modules$ be the category of modules over $R$. Then $R\modules$ becomes a monoidal category under the tensor product $\otimes_R$ of $R$-modules, with $R$ serving as the tensor unit. (As before, we omit an explicit description of the associator and unitors.) This example illustrates that monoidal \emph{structure} is not a mere \emph{property} of a category: the category of $R\modules$ is also a monoidal category with respect to the cartesian product.
\end{example}

\subsubsection{Monoidal functors and adjunctions} Monoidal categories come with several corresponding notions of \emph{monoidal functor}, varying in strength.

\begin{defn}[lax/strong/strict monoidal functors]
    A \emph{lax monoidal functor} between monoidal categories $\mathsf{C}$ and $\mathsf{D}$ is a functor $F:\mathsf{C}\to \mathsf{D}$ together with 
    \begin{enumerate}
        \item a morphism, the \emph{unit},
            $$ I_{\mathsf{D}} \to F(I_\mathsf{C}), $$
        \item and a natural transformation, the \emph{multiplication}
            $$ F(A) \otimes F(B) \to F(A\otimes B), \;\;\;\; (A,B \in \mathsf{C})$$
    \end{enumerate} 
    subject to a number of coherence conditions (which we omit, see \cite[pp. 255-56]{mac2013categories} for a detailed definition). If unit and multiplication are both isomorphisms, then $F$ is a \emph{strong monoidal functor}, and if, in addition, these isomorphisms are even \emph{identities}, then $F$ is called a \emph{strict monoidal functor}. Even though a monoidal functor is a functor with additional \emph{structure}, we will, as is common, sometimes abuse terminology and treat this structure as though it were a mere \emph{property}. For example, when we say that a functor $F$ ``is'' a monoidal functor, we mean that $F$ admits a specific structure of a monoidal functor which is left implicit. 
\end{defn}

Adjunctions also have a monoidal counterpart.

\begin{defn}
    A \emph{monoidal adjunction} is an adjunction
    \[\begin{tikzcd}
	{\mathsf{C}} & {\mathsf{D}}
	\arrow[""{name=0, anchor=center, inner sep=0}, "G"', shift right=2, from=1-1, to=1-2]
	\arrow[""{name=1, anchor=center, inner sep=0}, "F"', shift right=2, from=1-2, to=1-1]
	\arrow["\dashv"{anchor=center, rotate=-90}, draw=none, from=1, to=0]
    \end{tikzcd}\]
    between monoidal categories $\mathsf{C}, \mathsf{D}$ in which the left adjoint $F$ is a strong monoidal functor.
\end{defn}

\subsection{Symmetric monoidal categories} Demanding further structure which renders the tensor product commutative, or \emph{symmetric}, in a suitable sense, leads to the following definition.

\begin{defn}\label{symm_mon_cat_defn}
    A symmetric monoidal category is a monoidal category $\mathsf{C}$ together with a natural isomorphism
        $$ \sigma^{A,B}: A \otimes B \cong B \otimes A, \;\;\;\; (A, B \in \mathsf{C})$$
    called the \emph{braiding}, satisfying the identity 
        $$ \sigma^{B,A} \circ \sigma^{A,B} = \id_{A\otimes B}, $$
    as well as certain coherence conditions (which we omit, a detailed definition can be found in \cite[pp. 252-53]{mac2013categories}).
\end{defn}

\subsubsection{Symmetric monoidal functors and adjunctions} As for the case of bare monoidal categories, there are associated notions of \emph{symmetric monoidal functor} and \emph{symmetric monoidal adjunction}. We will only need the case of \emph{strong} symmetric monoidal functors.

\begin{defn}
    A \emph{strong symmetric monoidal functor} is a strong monoidal functor $F:\mathsf{C} \to \mathsf{D}$ between symmetric monoidal categories, which is compatible with the braiding $\sigma$ on both categories in that the following diagram commutes for all $A,B\in \mathsf{C}$:
    \[\begin{tikzcd}
	{FA \otimes FB} & {FB \otimes FA} \\
	{F(A\otimes B)} & {F(B\otimes A)}
	\arrow["\sigma", from=1-1, to=1-2]
	\arrow[from=1-1, to=2-1]
	\arrow[from=1-2, to=2-2]
	\arrow["{F(\sigma)}"', from=2-1, to=2-2]
    \end{tikzcd}\]
\end{defn}

\begin{defn}\label{symmmon_adj}
    A \emph{symmetric monoidal adjunction} is a monoidal adjunction 
    \[\begin{tikzcd}
	{\mathsf{C}} & {\mathsf{D}}
	\arrow[""{name=0, anchor=center, inner sep=0}, "G"', shift right=2, from=1-1, to=1-2]
	\arrow[""{name=1, anchor=center, inner sep=0}, "F"', shift right=2, from=1-2, to=1-1]
	\arrow["\dashv"{anchor=center, rotate=-90}, draw=none, from=1, to=0]
    \end{tikzcd}\]
    between \emph{symmetric} monoidal categories $\mathsf{C}, \mathsf{D}$ in which the left adjoint $F$ is a strong \emph{symmetric} monoidal functor.
\end{defn}

\subsection{Symmetric monoidal monads: commutative monads}

Adapting the notion of monad to the setting of symmetric monoidal categories leads to the notion of \emph{symmetric monoidal monad}. We will also use the term \emph{commutative monad} as a synonym for ``symmetric monoidal monad''. (Despite the definition of ``commutative monad'' usually being slightly different, they are equivalent concepts, see \cite{kock1972strong}.)

\begin{defn}[Commutative Monad]\label{comm_monad_def}
    Let $(C, \boxtimes, I)$ be a symmetric monoidal category. A \emph{symmetric monoidal monad}, or \emph{commutative monad}, $(T,i,m,\otimes)$ on $C$ is 
    \begin{enumerate}
        \item a monad $(T,i,m)$ on $C$, together with
        \item a natural transformation, 
            $$ \otimes^{X,Y}: TX \boxtimes TY \to T(X\boxtimes Y), \qquad (X,Y \in \mathsf{C})$$
        such that
        \item $(T,i,\otimes)$ is a monoidal functor, satisfying two further coherence conditions which we omit (see \cite[1.2]{seal2012tensors} for the complete definition), 
        \item such that the following diagram commutes, where $\sigma$ is the braiding (see \cref{symm_mon_cat_defn}):
        \[\begin{tikzcd}
	       {TX \boxtimes TY} & {TY \boxtimes TX} \\
	       {T(X \boxtimes Y)} & {T(Y \boxtimes X)}
	       \arrow["\sigma", from=1-1, to=1-2]
	       \arrow["T\sigma", from=2-1, to=2-2]
	       \arrow["\otimes"', from=1-1, to=2-1]
	       \arrow["\otimes", from=1-2, to=2-2]
        \end{tikzcd}\]    
    \end{enumerate}
\end{defn}

\begin{example}
    The free-$R$-module monad $R^{(-)}$ becomes a commutative monad on $(\sets, \times, 1)$ if we let 
        $$ \otimes: R^{(X)} \times R^{(Y)} \to R^{(X\times Y)}, \;\;\;\;\;\;\;(X,Y\in \sets)$$
    be the map defined by sending a pair of basis element $(x,y) \in R^{(X)} \times R^{(Y)}$ to the basis element $(x,y)\in R^{(X\times Y)}$ and using the universal property of free $R$-modules to extend by linearity. 
\end{example}

\begin{example}
    The free-linear-$hk$-space becomes a commutative monad on $(\spaces, \times, 1)$, in the same way as the free-$R$-module monad, \emph{mutatis mutandis} (using the universal property of free linear $hk$-spaces in place of the one for free $R$-modules).
\end{example}

\begin{remark}
    One can generalise the previous two examples to the case of $R$-module objects internal to a sufficiently well-behaved cartesian closed category, and even further to the case of modules over \emph{commutative algebraic theories} in place of $R$-modules, see \cite{zhenlin2012algebraic}.
\end{remark}

\subsubsection{Binary Morphisms and tensor products of modules over commutative monads.}

The following is a generalisation of the notion of a bilinear map of modules over a ring to the case of modules over a commutative monad:

\begin{defn}[Binary morphism, \cite{seal2012tensors}, 2.1]
    Let $(T, i, m, \otimes)$ be a commutative monad over a symmetric monoidal category $(\mathsf{C}, \boxtimes, I)$. A \emph{binary morphism} of $T$-modules is a $\mathsf{C}$-morphism $f: A\boxtimes B \to C$ (where $A,B,C$ are $T$-modules) such that the following diagram commutes:
    \[\begin{tikzcd}
	{TA\boxtimes TB} & {T(A\boxtimes B)} & TC \\
	{A\boxtimes B} && C
	\arrow["\otimes", from=1-1, to=1-2]
	\arrow["Tf", from=1-2, to=1-3]
	\arrow["f"', from=2-1, to=2-3]
	\arrow["{t^A\,\boxtimes\, t^B\;}"', from=1-1, to=2-1]
	\arrow["{t^C}", from=1-3, to=2-3]
    \end{tikzcd}\]
\end{defn}

\begin{remark}
    In \cite{seal2012tensors}, binary morphisms of $T$-modules are referred to as ``bimorphisms''. This term is however slightly ambiguous, as it is also sometimes used for morphisms which are both an epimorphism and a monomorphism.
\end{remark}

The above definition of binary morphism may seem somewhat obscure. We can shed some light on it by introducing a notion of ``morphism in one variable'' and then showing that a binary morphism is nothing but a ``morphism in \emph{both} variables''.

\begin{defn}[Morphism in a variable]
    Let $(T, i, m, \otimes)$ be a commutative monad over a symmetric monoidal category $(\mathsf{C}, \boxtimes, I)$, let $A,B,C$ be $T$-modules and let $f: A \boxtimes B \to C$ be a $\mathsf{C}$-morphism. 
    \begin{enumerate}
        \item $f$ is a \emph{morphism of $T$-modules in the left variable} if, and only if, the diagram 
        \[\begin{tikzcd}
	       {TA\boxtimes B} && {T(A\boxtimes B)} & TC \\
	       {A\boxtimes B} &&& C
	       \arrow["{\otimes \circ (\id_{TA} \boxtimes i^B)}", from=1-1, to=1-3]
	       \arrow["Tf", from=1-3, to=1-4]
	       \arrow["f"', from=2-1, to=2-4]
	       \arrow["{t^A \boxtimes \id_B\;}"', from=1-1, to=2-1]
	       \arrow["{t^C}", from=1-4, to=2-4]
        \end{tikzcd}\]
        commutes.
        \item $f$ is a \emph{morphism of $T$-modules in the right variable} if, and only if, the diagram 
        \[\begin{tikzcd}
	       {A \boxtimes TB} && {T(A\boxtimes B)} & TC \\
	       {A\boxtimes B} &&& C
	       \arrow["{\otimes \circ (i^A \boxtimes \id_{TB})}", from=1-1, to=1-3]
	       \arrow["Tf", from=1-3, to=1-4]
	       \arrow["{t^C}", from=1-4, to=2-4]
	       \arrow["f"', from=2-1, to=2-4]
	       \arrow["{\id_A \boxtimes t^B\;}"', from=1-1, to=2-1]
        \end{tikzcd}\]
        commutes.
    \end{enumerate}
\end{defn}

If this definition is still not quite transparent, considering the case when $\boxtimes$ is an ordinary cartesian product should clear it up:

\begin{lem}\label{concrete_bin_morphisms}
    Let $(T, i, m, \otimes)$ be a commutative monad over a concrete category $\mathsf{C}$ with concrete finite products (i.e.~$\mathsf{C}$ is equipped with a ``forgetful'' faithful functor $\mathsf{C}\to \sets$, which we will leave implicit, preserving finite products). Then a $\mathsf{C}$-morphism $f: A \times B \to C$ is a morphism of $T$-modules in the left (right, resp.) variable if, and only if, for any $b_0\in B$ (for any $a_0\in A$, resp.), the map
        $$ a \mapsto f(a, b_0), \;\;\; (b \mapsto f(a_0, b), \text{  resp.})$$
    is a morphism of $T$-modules.
\end{lem}
\begin{proof}
    Because of the symmetry of the situation, we consider only the case of a morphism of $T$-modules in the left variable. For the ``if'' direction, suppose that for all $b_0\in B$, $a \mapsto f(a, b_0)$ is a morphism of $T$-modules. Then for all $b_0\in B$, $\alpha\in TA$,
    \begin{align*}
        f \circ (t^A \times \id_B) (\alpha, b_0) 
        &= f(t^A(\alpha), b_0) \\
        &= t^C(\,T(f(-, b_0))(\alpha)\,)\\
        &= t^C(\,T(f)(\alpha \otimes i^B(b_0))\,)\\
        &= t^C \circ Tf \circ \otimes \circ  (\id_{TA}\times i^B) (\alpha, b_0).
    \end{align*}
    This shows that $f$ is a morphism of $T$-modules in the left variable. \par 
    For the ``only if'' direction, suppose that $f$ is a morphism of $T$-modules in the left variable and let $b_0 \in B$. Then 
    \begin{align*}
        f(t^A(\alpha), b_0) 
        &= f \circ (t^A \times \id_B) (\alpha, b_0) \\
        &= t^C \circ Tf \circ \otimes \circ  (\id_{TA}\times i^B) (\alpha, b_0) \\
        &= t^C(\,T(f)(\alpha \otimes i^B(b_0))\,) \\
        &= t^C(\,T(f(-, b_0))(\alpha)\,) 
    \end{align*}
    where we have used that $f$ is a morphism of $T$-modules in the left variable in going from the first to the second line. This shows that $f(-,b_0)$ is a morphism of $T$-modules, completing the proof. 
\end{proof}

We can now give a three equivalent characterisations of binary morphisms. The second confirms the intuition that a binary morphism is indeed a ``morphism in both variables'', while the third hints on how to construct the \emph{tensor product of $T$-modules}.

\begin{prop}[Characterisations of binary morphisms]
    As before, let $(T, i, m, \otimes)$ be a commutative monad over a monoidal category $(\mathsf{C}, \boxtimes, I)$, let $A,B,C$ be $T$-modules and let $f: A \boxtimes B \to C$ be a $\mathsf{C}$-morphism. Then the following are equivalent:
    \begin{enumerate}
        \item $f$ is a binary morphism of $T$-modules.
        \item $f$ is both a morphism of $T$-modules in left variable and a morphism of $T$-modules in the right variable.
        \item The diagram,
        \[\begin{tikzcd}
	       {T(TA\boxtimes TB)} &&& {T(A\boxtimes B)} && C,
	       \arrow["t^{T(A\times B)}\circ\, {T(\otimes)}", shift left=1, from=1-1, to=1-4]
	       \arrow["{T(t^A\boxtimes t^B)}"', shift right=1, from=1-1, to=1-4]
	       \arrow["{t^C\circ \,Tf}", from=1-4, to=1-6]
        \end{tikzcd}\]
        commutes.
    \end{enumerate}
\end{prop}
\begin{proof}
    The equivalence of the first two points is \cite[Proposition 2.1.2]{seal2012tensors}; the equivalence of the first and the last item is \cite[Lemma 2.3.2]{seal2012tensors}
\end{proof}

\subsubsection{The tensor product of $T$-modules} The following is a generalisation of the tensor product of modules over a commutative ring to the setting of commutative monads. 

\begin{defn}[Tensor product of $T$-modules]\label{tensor_product_of_t_mods}
    Let $(T, i, m, \otimes)$ be a commutative monad over a symmetric monoidal category $(\mathsf{C}, \boxtimes, I)$. The \emph{tensor product} $A \otimes_T B$ of two $T$-modules $A$ and $B$ is defined, if it exists, as the co-equaliser (in $\mathsf{M}$) of the two morphisms 
    \[\begin{tikzcd}
	{T(TA\boxtimes TB)} &&& {T(A\boxtimes B)}.
	\arrow["{T(\otimes)\circ t^{\,TA\boxtimes TB}}", shift left=1, from=1-1, to=1-4]
	\arrow["{T(t^A\boxtimes t^B)}"', shift right=1, from=1-1, to=1-4]
    \end{tikzcd}\]
\end{defn}

This construction satisfies the expected universal property \cite[Lemma 2.3.3]{seal2012tensors}:

\begin{prop}[Universal property of tensor product]\label{tens_prod_t_modules_univ_prop}
    As before, let $(T, i, m, \otimes)$ be a commutative monad over a symmetric monoidal category $(\mathsf{C}, \boxtimes, I)$ and let $A,B$ be objects of $\mathsf{C}$ and suppose that the tensor product $A\otimes B$ exists. Then, for every binary morphism of $T$-modules,
        $$ f: A \boxtimes B \to C, $$
    to some object $C$ of $\mathsf{C}$, there is a unique morphism of $T$-modules,
        $$ \overline{f}: A \otimes B \to C, $$
    such that the diagram, 
    \[\begin{tikzcd}
	{A\otimes_{\mathsf{C}} B} & C \\
	{A\boxtimes B}
	\arrow["{\overline{f}}", dashed, from=1-1, to=1-2]
	\arrow["{\otimes_{\mathsf{C}}}", from=2-1, to=1-1]
	\arrow["f"', from=2-1, to=1-2]
    \end{tikzcd}\]
    commutes, where the map $A\boxtimes B \to A \otimes_{\mathsf{C}} B$ is the composite of the maps
    \[\begin{tikzcd}
	{A\boxtimes B} & {T(A\boxtimes B)} & {A\otimes_MB,}
	\arrow["i", from=1-1, to=1-2]
	\arrow["q", from=1-2, to=1-3]
    \end{tikzcd}\]
    with $q$ the canonical map to the coequaliser.
\end{prop}

The category of modules over a symmetric monoidal monad is now itself symmetric monoidal \cite[Theorem 2.2.5]{seal2012tensors}:

\begin{thm}\label{thm_eilenberg_moore_monoidal}
    Let $\mathsf{C}$ a symmetric monoidal category, let $T$ be a commutative monad on $\mathsf{C}$ and suppose that the Eilenberg-Moore category $\mathsf{C}^T$ has all coequalisers. Then the tensor product of $T$-modules endows $\mathsf{C}^T$ with a symmetric monoidal structure.
\end{thm}

\begin{example}[Tensor product of $R$-modules]
    Let $R$ be a commutative ring and consider the category the free-$R$-module monad $R^{(-)}$ (which is a commutative monad on $(\sets,\times, 1)$). As discussed before, the  category of $R^{(-)}$-modules is then equivalently the category of $R$-modules. In this case, a binary morphism is precisely a bilinear map (this is implied by \cref{concrete_bin_morphisms}) and then it follows from the respective universal properties that the tensor product of $R^{(-)}$-modules agrees with the usual tensor product.
\end{example}

\begin{example}[Tensor product of linear $hk$-spaces]
    Since linear $hk$-spaces are equivalently $F_{\mathbb{K}}$-modules (where $F_{\mathbb{K}}$ is the free-linear-$hk$-space monad, a commutative monad),  \cref{tensor_product_of_t_mods} yields a natural notion of tensor product of linear $hk$-spaces. We denote this tensor product by $\otimes_{\mathbb{K}}$. By \cref{thm_eilenberg_moore_monoidal}, we conclude that the category $\vect$ of linear $hk$-spaces becomes a symmetric monoidal category under the tensor product $\otimes_{\mathbb{K}}$.
\end{example}

As we will see, $\vect$ is even a \emph{closed symmetric monoidal category}.

\subsection{Closed symmetric monoidal categories}\label{sec_closed_symm_mon_cats}

\begin{defn}\label{defn_closed_symmetric_mon_cat}
    A \emph{closed symmetric monoidal category} is a symmetric monoidal category $\mathsf{C}$, equipped with an \emph{internal hom} functor,
        $$ [-,-]_{\mathsf{C}} : \mathsf{C}^\op \times \mathsf{C} \to \mathsf{C}, $$
    and natural isomorphisms 
        $$ \tau^{A,B,C} : \Hom_{\mathsf{C}}(A \otimes_{\mathsf{C}} B, C) \cong \Hom_{\mathsf{C}}(A, [B,C]_{\mathsf{C}}), \qquad (A,C\in \mathsf{C}) $$
    exhibiting $ - \otimes B $ as left adjoint to $[B , -]_{\mathsf{C}}$, for every object $B\in \mathsf{C}$. The adjunction $- \otimes B \dashv [B, -]_{\mathsf{C}}$ is called the \emph{tensor-hom adjunction}. If unambiguous, we drop indices and simply write $[-,-]$ instead of $[-,-]_{\mathsf{C}}$.
\end{defn}

\begin{remark}
    Since adjoints are essentially unique, the internal hom of a closed symmetric monoidal category is already essentially determined by its monoidal structure. In light of this, a \emph{closed structure} on a symmetric monoidal category is ``property-like''. 
\end{remark}

\begin{remark}
    In \cite{barr1991autonomous} (as well as other writings of Barr), closed symmetric monoidal category are referred to ``autonomous categories'' (not to be confused with \emph{$*$-autonomous} categories, which are a particular sort of symmetric monoidal category, see \cref{defn_star_autnms_cat}).  
\end{remark}

\begin{remark}[Relation to ``external'' hom functor] 
    Note that, by the tensor-hom adjunction, 
        $$ \Hom_{\mathsf{C}}(I, [-,-,]) \cong \Hom_{\mathsf{C}}(I \otimes -,-) \cong \Hom_{\mathsf{C}}(-,-), $$
    where $I$ is the tensor unit. Hence, a ``global element'' $I\to [-,-]$ of the internal hom corresponds bijectively to an element of the ``external'' hom $\Hom_{\mathsf{C}}(-,-)$. 
\end{remark} 

\subsubsection{Examples of closed symmetric monoidal categories} When the tensor product is the cartesian product $\times$, the notion of symmetric monoidal category reduces to that of a cartesian closed category. In particular:

\begin{example}
    The categories of $k$-spaces, $hk$-spaces and QCB spaces are closed symmetric monoidal categories under the cartesian product $\times$ as the tensor product and the space $C(-,-)$ of continuous maps as the internal hom. 
\end{example}

The most basic non-cartesian example is the following.

\begin{example}
    The category of $R$-modules over a commutative ring $R$ is a closed symmetric monoidal category with respect to the tensor product of $R$-modules and the usual internal hom. 
\end{example}

More generally, commutative monads give rise to closed symmetric monoidal categories \cite[Corollary 2.5.7]{seal2012tensors}:

\begin{thm}\label{thm_t_modules_closec_symmetric_monoidal}
    Let $\mathsf{C}$ be a closed symmetric monoidal category with all equalisers and let $T$ be a commutative monad on $\mathsf{C}$ such that the Eilenberg-Moore category $\mathsf{C}^T$ has all coequalisers. Then $\mathsf{C}^T$ becomes a closed symmetric monoidal category under the tensor product of $T$-modules.
\end{thm}

Most importantly to our cause, this implies that linear $hk$-spaces form a closed symmetric monoidal category:

\begin{example}\label{ex_vect_closed_symm_mon}
    The category $\vect$ of linear $hk$-spaces is equivalently the category of $F_{\mathbb{K}}$-modules (with $F_{\mathbb{K}}$ the free-linear-$hk$-space monad) which satisfies the hypotheses of \cref{thm_t_modules_closec_symmetric_monoidal} (since $\spaces$ is a complete category and $\vect$ is cocomplete). Hence, $\vect$ is a closed symmetric monoidal category under the tensor product $\otimes_{\mathbb{K}}$ of $F_{\mathbb{K}}$-modules. The internal hom that arises this way agrees with the space $L(-,-)$ from \cref{defn_lvw}, as can be seen as follows. Let $V,W,Z$ be linear $hk$-spaces. By cartesian closure of $\spaces$, the set of continuous linear maps $V \to L(W,Z)$ is in natural bijection with the set of continuous bilinear maps $V\times W \to Z$. The latter can in turn be identified with the set of continuous linear maps $V\otimes_{\mathbb{K}} W \to Z$. Hence, $- \otimes_{\mathbb{K}} W$ is left adjoint to $L(W,-)$ and $L(-,-)$ is indeed the internal hom associated to the tensor product $\otimes_{\mathbb{K}}$. 
\end{example}

\subsubsection{The internal tensor-hom adjunction} The tensor-hom adjunction ``internalises'' to the following statement.

\begin{prop}[Internal tensor-hom adjunction]
    Let $\mathsf{C}$ be a closed symmetric monoidal category. Then we have a natural isomorphism 
        $$ [A \otimes B, C] \cong [A, [B, C]]. \;\;\;\; (A, B, C \in \mathsf{C})$$
\end{prop}
\begin{proof}
    Let $D$ be any object of $\mathsf{C}$. Applying the (``external'') tensor-hom adjunction several times, we obtain the following chain of natural isomorphisms.
    \begin{align*}
        \Hom_{\mathsf{C}}(D, [A, [B, C]]) &\cong \Hom_{\mathsf{C}}(D, [A, [B, C]]) \\
        &\cong \Hom_{\mathsf{C}}(D\otimes A, [B, C]) \\
        &\cong \Hom_{\mathsf{C}}(D\otimes A \otimes B, C) \\
        &\cong \Hom_{\mathsf{C}}(D, [A \otimes B, C]).
    \end{align*}
    The claim now follows from the Yoneda lemma. 
\end{proof}

\subsubsection{The hom-hom adjunction} Another adjunction arising from the tensor-hom adjunction, crucial for our later considerations on duality theory, is the following.

\begin{prop}[Hom-hom adjunction]\label{hom_hom_adj}
    Let $\mathsf{C}$ be a closed symmetric monoidal category. Then we have a natural isomorphism
        $$ \Hom_{\mathsf{C}}(A, [B,C]) \cong \Hom_{\mathsf{C}^\op}([A, C], B), \;\;\;\; (A, B, C \in \mathsf{C})$$
    giving an adjunction 
    \[\begin{tikzcd}
	{\mathsf{C}^\op} & {\mathsf{C}.}
	\arrow[""{name=0, anchor=center, inner sep=0}, "{[-,C]}", shift left=2, from=1-1, to=1-2]
	\arrow[""{name=1, anchor=center, inner sep=0}, "{[-, C]}", shift left=2, from=1-2, to=1-1]
	\arrow["\dashv"{anchor=center, rotate=-90}, draw=none, from=0, to=1]
    \end{tikzcd}\]
\end{prop}
\begin{proof}
    Using the tensor-hom adjunction and symmetry of the tensor product, 
    \begin{align*}
        \Hom_{\mathsf{C}}(A, [B, C]) &\cong \Hom_{\mathsf{C}}(A \otimes B, C) \\
        &\cong \Hom_{\mathsf{C}}(B\otimes A, C) \\
        &\cong \Hom_{\mathsf{C}}(B, [A, C]) \\
        &= \Hom_{\mathsf{C}^\op}([A, C], B).
    \end{align*}
\end{proof}

As an important corollary, we obtain:

\begin{cor}
    Let $\mathsf{C}$ be a closed symmetric monoidal category. Then the tensor product $\otimes$ preserves all colimits in both variables, and the internal hom functor [-,-] preserves limits in the second variable and sends colimits to limits in the first variable:
    \begin{align*}
        (\colim_i A_i) \otimes (\colim_j B_j) &\cong \colim_{i,j} (A_i \otimes B_j), \\
        [\colim_i A_i, B] &\cong \lim_i [A_i, B], \\
        [A, \lim_i B_i] &\cong \lim_i [A, B_i]. 
    \end{align*}
\end{cor}
\begin{proof}
    Let $A$ be any object. Then the functors $A \otimes -$, $- \otimes A$ are left adjoints and $[A, -]$ is a right adjoint (by the tensor-hom adjunction). Furthermore, by the hom-hom adjunction, \cref{hom_hom_adj}, 
        $$[-, A]: \mathsf{C}^\op \to \mathsf{C}$$
    is a left adjoint. The claim now follows from the fact that left adjoints preserve colimits and right adjoints preserve limits.
\end{proof}

\subsubsection{Day's reflection theorem} Reflective subcategories of closed symmetric monoidal categories inherit their closed monoidal structure under a number of equivalent simple conditions which are summarised by the following theorem \cite[Theorem 1.2]{day1972reflection}: 

\begin{thm}\label{day_refl_thm}
    Let $\mathsf{C}$ be a closed symmetric monoidal category and suppose that  
    \[\begin{tikzcd}
	{\mathsf{D}} & {\mathsf{C}}
	\arrow[""{name=0, anchor=center, inner sep=0}, shift right=1, hook, from=1-1, to=1-2]
	\arrow[""{name=1, anchor=center, inner sep=0}, "L"', shift right=3, from=1-2, to=1-1]
	\arrow["\dashv"{anchor=center, rotate=-90}, draw=none, from=1, to=0]
    \end{tikzcd}\]
    is a reflective subcategory with reflector $L$ (i.e.~left adjoint to the inclusion). Let 
        $$ \epsilon : A \to LA $$
    be the unit of the adjunction $(\hookrightarrow) \vdash L$ and consider the following natural transformations.
    \begin{align*}
        \epsilon^{[A,B]}&:  [A,B] \to L[A,B], &&(A\in \mathsf{C}, \, B\in \mathsf{D}) \\
        [\epsilon^A, B]&: [LA, B] \to [A,B], &&(A\in \mathsf{C}, \, B\in \mathsf{D})\\
        L(\epsilon \otimes B)&: L(A \otimes B) \to L(LA \otimes B), &&(A,B \in \mathsf{D})\\
        L(\epsilon \otimes \epsilon) &: L(A \otimes B) \to L(LA \otimes LB).  &&(A,B \in \mathsf{D})
    \end{align*}
    Then, if one of them is an isomorphism, so are the others. In this case, $\mathsf{D}$ becomes a closed symmetric monoidal category with tensor product $L(-\otimes-)$ and internal hom $[-,-]$, and $L$ is (or rather, admits the structure of) a strong symmetric monoidal functor, thus making the adjunction $(\hookrightarrow) \vdash L$ a symmetric monoidal adjunction. 
\end{thm}

\section{Duality: Separated Objects, $*$-Autonomous Categories and the Chu construction}\label{sec_duality_sep_chu}

We now discuss a general theory of duality in the setting of closed symmetric monoidal categories which we will subsequently apply to linear $hk$-spaces. 

\subsubsection{Evaluation morphisms} First, we need a definition of a general version of evaluation mappings and the canonical morphism $\eta: V\to V^{\wedge\wedge}$ to the double dual.  

\begin{defn}\label{evaluation_morphisms_defn}
    Let $\mathsf{C}$ be a closed symmetric monoidal category and let $A,B,D$ be objects of $\mathsf{C}$. 
    \begin{enumerate}
        \item The \emph{evaluation morphism},
            $$ e^{A, B} : [A,B] \otimes A \to B, $$
        is given by
        $$ e^{A,B} := (\tau^{[A,B], A, B})^{-1}(\id_{[A,B]}).$$
        \item The \emph{canonical morphism to the double dual of $A$ with respect to $D$} is the morphism 
            $$ \eta^{A, D}: A \to [[A,D],D]$$
        defined as 
            $$ \eta^{A, D} := \tau^{A, [A, D], D}(e^{A, D} \circ \sigma^{A, [A,D]}), $$
        where $\sigma$ is the braiding (see \cref{symm_mon_cat_defn}). 
    \end{enumerate}
\end{defn}

Examining the proof of the hom-hom adjunction (\cref{hom_hom_adj}), we see that $\eta^{-,D}$ is precisely the unit of this adjunction. In particular, 
    $$ \eta^{-,D}: \id_\mathsf{C} \to [[-,D],D] $$
is a natural transformation.

\subsection{Regular monomorphisms, regular images and (co)regular categories} At this point, we would like to ask whether $\eta^{A,D}$ is -- in a suitable sense -- an embedding, and what its image is. This would then yield a general formulation of what it means for the dual $[A,D]$ to ``separate the points'' of $A$ and enable the study of $A$ in terms of its dual $[A,D]$. The category-theoretic notion of embedding which we would like to adhere to is that of \emph{regular monomorphism} with the corresponding notion of image being the \emph{regular image}. Categories in which these notions are particularly well-behaved are known as \emph{coregular categories} (the dual notion of which is that of a \emph{regular category}).

\subsubsection{Regular mono- and epimorphisms} We begin with the notions of regular mono- and epimorphism. 

\begin{defn}
    Let $\mathsf{C}$ be a category. A \emph{regular monomorphism} is a morphism which is the equaliser of some pair of morphisms. A \emph{regular epimorphism} in $\mathsf{C}$ is a regular monomorphism in the opposite category $\mathsf{C}^\op$, i.e.~a morphism (in $\mathsf{C}$) which is the \emph{co}equaliser of some pair of morphisms. 
\end{defn}

As the terminology suggests:

\begin{prop}
    A regular monomorphism is always a monomorphism.
\end{prop}
\begin{proof}
    By definition, a regular monomorphism $m: A\to B$ is the equaliser of a pair of maps $k,l: B \to C$. For any two morphisms $f,g: D \to A$ with $m \circ f = m \circ g$, the following diagram commutes:
    \[\begin{tikzcd}
	A & B & C \\
	D
	\arrow["m", from=1-1, to=1-2]
	\arrow["k", shift left=1, from=1-2, to=1-3]
	\arrow["l"', shift right=1, from=1-2, to=1-3]
	\arrow["{m \circ g}", shift left=1, curve={height=6pt}, from=2-1, to=1-2]
	\arrow["{m \circ f}"', shift right=1, curve={height=6pt}, from=2-1, to=1-2]
    \end{tikzcd}\]
    Hence, by the universal property of the equaliser, there is a \emph{unique} morphism $u$ making the following diagram commute.
    \[\begin{tikzcd}
	A & B & C \\
	D
	\arrow["m", from=1-1, to=1-2]
	\arrow["k", shift left=1, from=1-2, to=1-3]
	\arrow["l"', shift right=1, from=1-2, to=1-3]
	\arrow["{m \circ g}", shift left=1, curve={height=6pt}, from=2-1, to=1-2]
	\arrow["{m \circ f}"', shift right=1, curve={height=6pt}, from=2-1, to=1-2]
	\arrow["u", curve={height=-6pt}, dashed, from=2-1, to=1-1]
    \end{tikzcd}\]
    Therefore, $f=u=g$. Since $f,g$ were arbitrary morphisms with $m \circ f = m \circ g$, this shows that $m$ is a monomorphism.
\end{proof}

Recall that a \emph{split epimorphism} is a morphism which admits a section (i.e.~a right inverse). Dually, a \emph{split monomorphism} is a morphism which admits a retraction (i.e.~a left inverse). 

\begin{lem}
    Let $\mathsf{C}$ be a category with equalisers. Then every split epimorphism (split monomorphism, resp.) is a regular epimorphism (monomorphism, resp.).
\end{lem}
\begin{proof}
    If $i: A \to B$ is is a split monomorphism and $p$ is a left inverse of $i$, then $i$ is the equaliser of $i\circ p$ and $\id_B$ and hence a regular monomorphism. The case of split epimorphisms is dual. 
\end{proof}

\subsubsection{Characterising (regular, split) mono- and epimorphisms in $\khaus$}\label{khaus_morphisms_subsec}  In order to be able to apply this general discussion to the examples that we are interested in, we also need to characterise the regular mono- and epimorphisms in the category of $hk$-spaces, where they correspond to important notions of general topology. The following can be found in \cite[Theorem 3.1]{strickland2009category}:

\begin{thm}[Morphisms in $\khaus$]\label{khaus_morphisms}
    Let $f:X\to Y$ be a continuous map of $hk$-spaces. Then we have the following equivalences:
    \begin{align*}
        f \text{ is a monomorphism }&\Leftrightarrow \;f \text{ is injective, }\\
        f \text{ is a regular monomorphism }&\Leftrightarrow \;f \text{ is a closed embedding, }\\
        f \text{ is a split monomorphism }&\Leftrightarrow \;\text{$X$ is a retract of $Y$ and $f:X\hookrightarrow Y$ is the inclusion, }
    \end{align*}
    and dually, 
    \begin{align*}
        f \text{ is an epimorphism }&\Leftrightarrow \;f \text{ has dense image, }\\
        f \text{ is a regular epimorphism }&\Leftrightarrow \;f \text{ is a quotient map, }\\
        f \text{ is a split epimorphism }&\Leftrightarrow \;\text{ $Y$ is a retract of $X$ and $f:X\twoheadrightarrow Y$ is the retraction. }
    \end{align*}
    Moreover, we have the following permanence properties, where all limits and colimits are taken in $\khaus$:
    \begin{enumerate}[1.]
        \item Arbitrary products, coproducts, pushouts or composites of regular monomorphisms (i.e.~closed embedding, by the above) are regular monomorphisms.  
        \item Arbitrary coproducts, \emph{finite} products, pullbacks and composites of regular epimorphisms (i.e~quotient maps, by the above) are regular epimorphisms. 
    \end{enumerate}
\end{thm}

Using this characterisation, we may similarly characterise the regular monomorphisms in $\vect$, which we will need in order to apply our general duality theoretic results to the case of linear $hk$-spaces.

\begin{cor}\label{cor_morphisms_in_vect}
    Let $V,W$ be linear $hk$-spaces and let $i:V\to W$ be a continuous linear map. Then $i$ is a regular monomorphism if, and only if, it is a regular monomorphism in $\spaces$, i.e.~a closed embedding.
\end{cor}
\begin{proof}
    For the ``if'' direction, note that if $i$ is a closed embedding, then it is the equaliser of the projection
        $$ p: W \to W/i(V) $$
    and the zero map $0: W \to W/i(V)$. \par 
    For the other direction, recall that the forgetful functor $\vect\to \spaces$ is a right adjoint (to the free linear $hk$-space functor $F_{\mathbb{K}}$) and that right adjoints preserve limits. Hence, if $i$ is a regular monomorphism in $\vect$, it is also a regular monomorphism in $\spaces$. 
\end{proof}

\subsubsection{The regular image} 

\begin{defn}\label{regular_image_defn}
    Let $\mathsf{C}$ be a category and let $f:A\to B$ be a morphism. 
    \begin{enumerate}
        \item The \emph{regular image} $\im(f)$ of $f$ is, if it exists, the equaliser of the cokernel pair of $f$, i.e.~the equaliser of the pair of morphisms $B \to B \sqcup_f B$ arising from the pushout diagram:
        \[\begin{tikzcd}
	       A & B \\
	       B & {B \sqcup_f B}
	       \arrow["f", from=1-1, to=1-2]
	       \arrow["f"', from=1-1, to=2-1]
	       \arrow[from=1-2, to=2-2]
	       \arrow[from=2-1, to=2-2]
        \end{tikzcd}\]
        \item Dually, the \emph{regular coimage} $\coim(f)$ of $f$ is, if it exists, the coequaliser of the kernel pair of $f$. Put differently, the regular coimage of $f$ is the image of $f$ considered as an morphism in the opposite category $\mathsf{C}$.
        \item The universal property of the pushout yields a unique morphism,
            $$ f|^{\im(f)}: A \to \im(f), $$
        the \emph{corestriction of $f$ to its image}, making the following diagram commute: 
        \[\begin{tikzcd}
	       A \\
	       & {\im(f)} & B \\
	       & B & {B \sqcup_f B}
	       \arrow[from=2-3, to=3-3]
	       \arrow[from=3-2, to=3-3]
	       \arrow["f"', curve={height=12pt}, from=1-1, to=3-2]
	       \arrow["f", curve={height=-12pt}, from=1-1, to=2-3]
	       \arrow["{f|^{\im(f)}}", from=1-1, to=2-2]
	       \arrow[hook', from=2-2, to=2-3]
	       \arrow[hook, from=2-2, to=3-2]
        \end{tikzcd}\]
        \item The factorisation 
        \[\begin{tikzcd}
	   A & {\im(f)} & B
	   \arrow["{f|^{\im(f)}}", from=1-1, to=1-2]
	   \arrow[hook, from=1-2, to=1-3]
        \end{tikzcd}\]
        is called the \emph{regular image factorisation} of $f$.
    \end{enumerate}
\end{defn}

\subsubsection{Regular and coregular categories} 

\begin{defn}[(Co)regular category, \cite{gran2021introduction}]\label{regular_cat_defn}
    A category $\mathsf{C}$ is \emph{coregular} if 
    \begin{enumerate}
        \item it has all finite colimits,
        \item the regular image of every morphism in $\mathsf{C}$ exists,
        \item and regular monomorphisms are stable under pushouts.
    \end{enumerate}
    $\mathsf{C}$ is \emph{regular} if $\mathsf{C}^\op$ is coregular, and \emph{biregular} if it is both regular and coregular.
\end{defn}

The following shows that in coregular categories, the regular image behaves as one would hope \cite[Theorem 1.11]{gran2021introduction}:

\begin{thm}\label{reg_im_in_reg_cats}
    Let $\mathsf{C}$ be a coregular category let $f$ be a morphism. Then:
    \begin{enumerate}
        \item The corestriction of $f$ to its image 
            $$ f|^{\im(f)}: A \to \im(f) $$
        is an epimorphism. Hence, the regular image factorisation factors $f$ into an epimorphism followed by a regular monomorphism:
        \[\begin{tikzcd}
	       A & {\im(f)} & B
	       \arrow["{f|^{\im(f)}}", from=1-1, to=1-2]
	       \arrow[hook, from=1-2, to=1-3]
        \end{tikzcd}\]
        \item Such factorisation is essentially unique in that for any further epimorphism $e$ and regular monomorphism $m$, there is a unique isomorphism $i:I\to \im(f)$ making the following diagram commute
        \[\begin{tikzcd}
	       && {\im(f)} \\
	       A &&&& B \\
	       && I
	       \arrow["{f|^{\im(f)}}", from=2-1, to=1-3]
	       \arrow["e"', from=2-1, to=3-3]
	       \arrow["m"', from=3-3, to=2-5]
	       \arrow[hook, from=1-3, to=2-5]
	       \arrow["i", dashed, from=3-3, to=1-3]
        \end{tikzcd}\]
    \end{enumerate}
\end{thm}

In fact, coregular categories are \emph{precisely} those categories in which every morphism can be factored into an epimorphism followed by a regular monomorphism in a well-behaved manner (see \cite[Theorem 1.14]{gran2021introduction}):

\begin{prop}
    Let $\mathsf{C}$ be a category with finite colimits in which regular monomorphisms are stable under pullback. Then $\mathsf{C}$ is coregular if, and only if, every morphism admits a factorisation into an epimorphism followed by regular monomorphism. 
\end{prop}

\subsubsection{The category of $hk$-spaces is biregular} Cartesian closure is not the only convenience that $hk$-spaces provide. As shown in \cite[Theorem 3.1 (f)]{strickland2009category}:

\begin{prop}
    The category $\spaces$ of $hk$-spaces is biregular, i.e.~both a coregular category and a regular category.
\end{prop}

\subsubsection{Properties of regular monomorphisms in coregular categories} As a final note on (co)regular categories, let us record the following fact (see \cite[Proposition 1.13, Lemma 1.3]{gran2021introduction} for a proof).

\begin{prop}\label{properties_of_reg_epis_in_reg_cat}
    Let $\mathsf{C}$ be a coregular category and let $f,g$ be morphisms. 
    \begin{enumerate}
        \item The composite $f\circ g$ of any pair of regular monomorphisms $f$, $g$ is a regular monomorphisms, as well.
        \item If $g\circ f$ is a regular monomorphism, then $f$ is also a regular monomorphism.
        \item If $f$ is both a regular monomorphism and an epimorphism, then $f$ is an isomorphism.
    \end{enumerate}
\end{prop}

\subsection{The triple dualisation lemma and separated objects}\label{sec_triple_dualisation_lem_and_sep_obj}

\subsubsection{The triple dualisation lemma} The following simple observation will be crucial.

\begin{lem}[triple dualisation lemma]\label{lem_triple_dualisation_principle}
    Let $\mathsf{C}$ be a closed symmetric monoidal category and let $D$ be any object of $\mathsf{C}$. Then 
        $$ [\eta^{C, D}, D] \circ \eta^{[C,D], D} = \id_{[C,D]}, $$
    i.e.~the canonical morphism to the ``triple dual" with respect to $D$ (see \cref{evaluation_morphisms_defn}),
        $$ \eta^{[C,D], D} : [C,D] \to [[[C,D],D],D], $$
    has a right inverse, given by 
        $$ [\eta^{C, D}, D] : [[[C,D],D],D] \to [C,D].$$
    In particular, $\eta^{[C,D], D}$ is a split (and hence regular) monomorphism. 
\end{lem}
\begin{proof}
    An adjunction can be characterised by the \emph{triangle identities} (see \cite[Theorem 3.1.5]{borceux1994handbook}, or any other category theory textbook). The triangle identities for the hom-hom adjunction (see \cref{hom_hom_adj}) reduce to a single identity (because of the ``self-adjoint'' nature of this adjunction), and this identity is given precisely by:   
        $$ [\eta^{C, D}, D] \circ \eta^{[C,D], D} = \id_{[C,D]}. $$
\end{proof}

\subsubsection{Separated objects in a closed symmetric monoidal category}

The triple dualisation lemma suggests the following definition.

\begin{defn}
    Let $\mathsf{C}$ be a closed symmetric monoidal category and let $D$ be an object. An object $A$ of $\mathsf{C}$ is \emph{$D$-separated} if the canonical map (see \cref{evaluation_morphisms_defn}).
        $$ \eta^{A, D}: A \to [[A,D],D]$$
    is a regular monomorphism.
\end{defn}

\begin{example}
    Let $K$ be a field. Then every (algebraic) $K$-vector space is $K$-separated (assuming the axiom of choice).
\end{example}

The example that will be most important to our context is the following.

\begin{example}\label{ex_repl_lin_sk_sp}
    The $\mathbb{K}$-separated objects in the closed symmetric monoidal category $\vect$ of linear $hk$-spaces are precisely the \emph{replete} linear $hk$-spaces. Recall that a replete linear $hk$-space is a linear $hk$-space for which the canonical map $V \to V^{\wedge\wedge}$ is a closed embedding (\cref{defn_repl_lin_hk_spaces}). Since regular monomorphisms in $\vect$ are exactly given by closed embeddings (\cref{cor_morphisms_in_vect}), this is exactly the definition of a $\mathbb{K}$-separated object in $\vect$.
\end{example}

\begin{remark}
    The notion of $D$-separated object can be seen as a special case of a construction related to the \emph{(enriched) idempotent core of a monad}, see \cite[Example 4.18]{lucyshyn2014completion}.
\end{remark}

\subsubsection{Some permanence properties of separated objects} It turns out that the category of $D$-separated objects enjoys excellent permanence properties, thanks to the triple dualisation lemma. The next proposition is a first step towards establishing this. 

\begin{prop}\label{duals_and_subobjects_sep}
    Let $\mathsf{C}$ be a coregular closed symmetric monoidal category and let $D$ be an object. Then:
    \begin{enumerate}
        \item For any object $A$, $[A,D]$ is $D$-separated.
        \item If $B$ is $D$-separated and $\iota: A \hookrightarrow B$ is a regular monomorphism, then $A$ is $D$-separated, as well. In particular, an object $A$ is $D$-separated if, and only if, it admits a regular monomorphism into \emph{some} dual $[B,D]$ with respect to $D$.
        \item For all objects $A,B\in \mathsf{C}$, if $B$ is $D$-separated, then $[A,B]$ is also $D$-separated.
    \end{enumerate}
\end{prop}
\begin{proof}
    \text{}
    \begin{enumerate}
        \item This follows directly from the triple dualisation lemma (\cref{lem_triple_dualisation_principle}).
        \item By naturality of $\eta$, the following diagram commutes.
        \[\begin{tikzcd}
	   B & {[[B,D],D]} \\
	   A & {[[A,D],D]}
	   \arrow["{\eta^{B,D}}", hook, from=1-1, to=1-2]
	   \arrow["\iota", hook', from=2-1, to=1-1]
	   \arrow["{\eta^{A,D}}"', from=2-1, to=2-2]
	   \arrow["{[[\iota,D],D]}"', from=2-2, to=1-2]
        \end{tikzcd}\]
        By \cref{properties_of_reg_epis_in_reg_cat}, the composition of regular monomorphisms is a regular monomorphism and hence, 
        $$h := \eta^{B,D} \circ \iota$$ 
        is a regular monomorphism. Since $h= [[\iota, D],D] \circ \eta^{A, D}$ (by commutativity of the above diagram), $\eta^{A, D}$ is a regular monomorphism as well, by \cref{properties_of_reg_epis_in_reg_cat}. This shows that $A$ is $D$-separated.
        \item Since $B$ is $D$-separated, $\eta^{B,D}$ is a regular monomorphism. Since the internal hom preserves limits in the second variable, 
            $$[A, \eta^{B,D}]: [A,B] \to [A, [[B,D],D]]$$ 
        is also a regular monomorphism. Hence, $[A,B]$ admits a regular monomorphism into a $D$-separated object and is therefore $D$-separated, by the previous point.
    \end{enumerate}
\end{proof}

\subsubsection{The separation functor} As we will see, $D$-separated objects form a reflective subcategory of the given category (if the latter is sufficiently well-behaved). Let us first define the would-be reflector, the \emph{$D$-separation}.

\begin{defn}[$D$-Separation]\label{sepfun_defn}
    Let $\mathsf{C}$ be a coregular closed symmetric monoidal category and let $D$ be an object. We define the \emph{$D$-separation functor} as the regular image (see \cref{regular_image_defn}),
        $$ \sepfun_D := \im(\eta^{-,D}), $$
    of the natural transformation $\eta^{-,D}$ in the functor category $\mathsf{C}^{\mathsf{C}}$. 
\end{defn}

Since limits and colimits in functor categories are computed pointwise (see, for example, \cite[Proposition 3.3.9]{riehl2017category}), the $D$-separation $\sepfun_D(A)$ of an \emph{object} $A$ is the regular image in $\mathsf{C}$ of the canonical morphism to the double dual of $A$ with respect to $D$. The regular image factorisation of $\eta^{-,D}$ furthermore yields two natural transformations,
    $$ A \to \sepfun_D(A) \hookrightarrow [[A,D],D], \;\;\;\; (A \in \mathsf{C} $$
the corestriction of $\eta^{A,D}$ to its regular image, and the inclusion of $\sepfun_D(A)$ into $[[A,D],D]$. \par
Of course, the term ``$D$-separation'' is only justified if the $D$-separation of an object is $D$-separated. This is indeed the case.

\begin{prop}\label{chars_of_D_separation}
    Let $\mathsf{C}$ be a coregular closed symmetric monoidal category and let $A, D$ be objects of $\mathsf{C}$. Then:
    \begin{enumerate}
        \item $\sepfun_D(A)$ is always $D$-separated.
        \item The following are equivalent:
        \begin{enumerate}
            \item $A$ is $D$-separated.
            \item The co-restriction of $\eta^{A,D}$ to its regular image,
            $$ \eta^{A,D}|^{\sepfun_D(A)}: A \to \sepfun_D(A), $$
            is an isomorphism.
            \item There is \emph{some} isomorphism 
                $$ A \cong \sepfun_D(A). $$
        \end{enumerate}
    \end{enumerate}
    In particular, we can use the second point to reformulate the first as saying:
        $$ \sepfun_D(\sepfun_D(A)) \cong \sepfun_D(A). $$
\end{prop}
\begin{proof}
    \text{}
    \begin{enumerate}
        \item $\sepfun_D(A)$ admits a regular monomorphism into a dual with respect to $D$, namely the inclusion
            $$ \sepfun_D(A) \hookrightarrow [[A,D],D]. $$
        Hence, the claim follows from \cref{duals_and_subobjects_sep}.
        \item If $A$ is $D$-separated, i.e. $\eta^{A,D}$ is a regular monomorphism, then $\eta^{A,D}|^{\sepfun_D(A)}$ is a regular monomorphism, as well (by \cref{properties_of_reg_epis_in_reg_cat}). Since $\eta^{A,D}|^{\sepfun_D(A)}$ is also an epimorphism, it is hence an isomorphism (by \cref{properties_of_reg_epis_in_reg_cat}).
        On the other hand, if $\eta^{A,D}|^{\sepfun_D(A)}$ is an isomorphism, then $\eta^{A,D}$ is a regular monomorphism, so $A$ is $D$-separated. \par 
        The remaining implications follow from the fact that $\sepfun_D(A)$ is always D-separated (this was the first point) and that the property of being $D$-separated is isomorphism-invariant.
    \end{enumerate}
\end{proof}

This also shows that the essential image of the $D$-separation functor $\sepfun_D$ is given precisely be the $D$-separated objects. The full subcategory of a coregular closed symmetric monoidal category $\mathsf{C}$ spanned by the $D$-separated objects is therefore appropriately denoted by 
    $$ \sepfun_D \mathsf{C} \hookrightarrow \mathsf{C}. $$

\begin{defn}\label{defn_repletion}
    For the case of replete linear $hk$-spaces and $D=\mathbb{K}$ (see \cref{ex_repl_lin_sk_sp}), we write
        $$ r := \sepfun_{\mathbb{K}} $$
    for the $\mathbb{K}$-separation functor and call this functor the \emph{repletion}. Accordingly, the category of replete linear $hk$-spaces will be denoted by 
        $$ \sclin := \sepfun_{\mathbb{K}}\vect. $$
\end{defn}

Unfolding definitions, the repletion $rV$ of a linear $hk$-space is precisely the closure of its image in the double dual $V^{\wedge\wedge}$ under the canonical map $V\to V^{\wedge\wedge}$. 
    
\subsubsection{$D$-separated objects form a reflective subcategory} By now, the reader might expect, as the name would suggest, that the $D$-separation is the reflector associated to the \emph{reflective} subcategory of $D$-separated objects. That this expectation is indeed met is the content of the following theorem.

\begin{thm}\label{thm_sepfun_left_adjoint}
    Let $\mathsf{C}$ be a coregular closed symmetric monoidal category and let $D$ be an object. Then the $D$-separation functor (co-restricted to its image),
        $$\sepfun_D: \mathsf{C} \to \sepfun_D \mathsf{C}, $$ 
    is left adjoint to the inclusion 
        $$ \sepfun_D\mathsf{C} \hookrightarrow \mathsf{C}. $$
\end{thm}
\begin{proof}
    In order to simplify notation, write
        $$ (-,-) := \Hom_{\mathsf{C}}(-,-), $$
        $$ \eta^A := \eta^{A,D}, $$
    and, 
        $$ \sepfun := \sepfun_D, \;\; j^A := \eta^A|^{\sepfun A}. $$
    In order to obtain an adjunction,
        $$ (\sepfun A, B) \cong (A,B), \;\;\;\;\;\; (A,B \in \mathsf{C})$$
    define two natural transformations $\Phi$, $\Psi$ with components
        $$ \Phi^{A,B} : (\sepfun A, B) \to (A, B), \;\; \Phi^{A,B} := (j^A, B) $$
        $$ \Psi^{A,B} : (A, B) \to (\sepfun A, B), \;\; \Psi^{A,B} := (\sepfun(A), (j^B)^{-1}) \circ \sepfun^{A,B},$$
    where 
        $$\sepfun^{A,B}: (A,B) \to (\sepfun(A), \sepfun(B)), \;\; f \mapsto \sepfun(f)$$ 
    is the action of the $D$-separation functor on morphisms $A\to B$.
    Fix two objects $A,B$ and write $\Phi := \Phi^{A,B}$, $\Psi := \Psi^{A,B}$. We claim that $\Phi$ and $\Psi$ are inverse to one another, which would complete the proof. We will show this in two steps.
    \begin{enumerate}
        \item First, we claim that 
        $$ j^{\sepfun(A)} = \sepfun(j^A). $$
        This can be seen as follows. By naturality of $j$, the diagram, 
        \[\begin{tikzcd}
	   A & {\sepfun(A)} \\
	   {\sepfun(A)} & {\sepfun(\sepfun(A))}
	   \arrow["{j^A}", from=1-1, to=1-2]
	   \arrow["{j^A}"', from=1-1, to=2-1]
	   \arrow["{\sepfun(j^A)}"', from=2-1, to=2-2]
	   \arrow["{j^{\sepfun(A)}}", from=1-2, to=2-2]
        \end{tikzcd}\]
        commutes, i.e. 
            $$ j^{\sepfun(A)} \circ j^A = \sepfun(j^A) \circ j^A. $$
        Since $j^A$ is an epimorphism (by \cref{reg_im_in_reg_cats}), this implies the (intermediate) claim.
        \item Now, using naturality of $j$, we find that 
    \begin{align*}
        \Phi \circ \Psi (f) &= (j^A, B) \circ (\sepfun(A), (j^B)^{-1}) \circ \sepfun^{A,B}(f) \\
        &= (j^A, B)((\sepfun(A), (j^B)^{-1})(\sepfun^{A,B}(f)) \\
        &= (j^A, B)((j^B)^{-1} \circ \sepfun^{A,B}(f)) \\
        &= (j^B)^{-1} \circ \sepfun^{A,B}(f) \circ j^A \\
        &= f,
    \end{align*}
    for all $f\in (A, B)$. Likewise, for every $g\in (\sepfun A, B)$,
    \begin{align*}
        \Psi \circ \Phi &= (\sepfun(A), (j^B)^{-1}) \circ \sepfun^{A,B} \circ (j^A, B) (g) \\
        &= (\sepfun(A), (j^B)^{-1}) ( \sepfun ( (j^A, B) (g) ) )\\
        &= (j^B)^{-1} \circ \sepfun ( g \circ j^A ) ) \\
        &= (j^B)^{-1} \circ \sepfun(g) \circ \sepfun(j^A) &&\text{(by functoriality of $\sepfun$)}\\
        &= (j^B)^{-1} \circ \sepfun(g) \circ j^{\sepfun(A)} &&\text{(by the previous step)} \\
        &= g. &&\text{(by naturality of $j$)}
    \end{align*}
    \end{enumerate}
    This completes the proof.
\end{proof}

The adjunction from \cref{thm_sepfun_left_adjoint} can be enhanced to an internal (or \emph{enriched}) version. 

\begin{cor}\label{internal_sepfin_adjunction}
    Let $\mathsf{C}$ be a coregular closed symmetric monoidal category and let $D$ be an object of $\mathsf{C}$. Then we have a natural isomorphism 
    \begin{align*}
        [A, B] \cong [\sepfun_D(A), B]. \;\;\;\;\;\;\;\; (A \in \mathsf{C}, B\in \sepfun_D \mathsf{C}) 
    \end{align*}
\end{cor}
\begin{proof}
    Let $C$ be any object. By \cref{duals_and_subobjects_sep}, for any $D$-separated object $B$, $[C,B]$ is $D$-separated, as well. Combining this fact with \cref{thm_sepfun_left_adjoint}, we get the following chain of natural isomorphisms.
    \begin{align*}
        \Hom_{\mathsf{C}}(C, [\sepfun(A), B]) &\cong \Hom_{\mathsf{C}}(\sepfun(A), [C,B])\\
        &\cong \Hom_{\mathsf{C}}(A, [C, B]) \\
        &\cong \Hom_{\mathsf{C}}(C, [A,B]).
    \end{align*}
    The claim now follows from the Yoneda lemma.
\end{proof}

\subsubsection{$D$-separated objects form a closed symmetric monoidal category}

Day's reflection theorem (\cref{day_refl_thm}) now immediately yields:

\begin{cor}\label{cor_d_sep_closed_symm_mon}
    Let $\mathsf{C}$ be a coregular closed symmetric monoidal category and let $D$ be any object. Then the category $\sepfun_D(\mathsf{C})$ of $D$-separated objects is a closed symmetric monoidal category with tensor product $\sepfun_D(-\otimes_{\mathsf{C}} -)$ and internal hom $[-,-]_\mathsf{C}$, and the $D$-separation functor is a strong symmetric monoidal functor 
        $$ \sepfun_D: \mathsf{C} \to \sepfun_D(\mathsf{C}). $$
\end{cor}

Applying this Corollary to the case of replete linear $hk$-spaces, we obtain:

\begin{cor}\label{cor_repl_closed_symm_mon}
    The category $\sclin$ of replete linear $hk$-spaces becomes a closed symmetric monoidal category under the internal hom $L(-,-)$ and tensor product
        $$ V \,\otimes_r W := \scompl(V \otimes W), \qquad (V,W \in \sclin)$$
    where $\otimes$ denotes the tensor product of linear $hk$-spaces (equivalently, $F_{\mathbb{K}}$-modules, see \cref{ex_vect_closed_symm_mon}) and $\scompl = \sepfun_{\mathbb{K}}$ is the repletion functor (\cref{defn_repletion}), i.e.~the $\mathbb{K}$-separation functor in the category of linear $hk$-spaces.
\end{cor}

\subsection{$*$-Autonomous categories}\label{sec_star_aut_cats}

\begin{defn}\label{defn_star_autnms_cat}
    A \emph{$*$-autonomous category} is a closed symmetric monoidal category $\mathsf{C}$ together with an object $D$, the \emph{dualising object}, such that 
        $$ \eta^{A,D}: A \to [[A, D], D], $$
    is an isomorphism for all $A\in \mathsf{C}$.
    In this case, we denote by $(-)^* := [-, D]$ the ``dualisation functor'', 
        $$ (-)^*: \mathsf{C}^\op \to \mathsf{C}, $$
    with respect to $D$, which is then an equivalence of categories.
\end{defn}

Hence, $*$-autonomous categories are ``self-dual'', 
    $$ \mathsf{C}^\op \cong \mathsf{C}, $$
a property with far-reaching consequences. For example, a $*$-autonomous category $\mathsf{C}$ which has all limits automatically admits all colimits as well (and vice versa). After all, colimits in $\mathsf{C}$ are just limits in the opposite category $\mathsf{C}^\op$ (which is equivalent to $\mathsf{C}$ by $*$-autonomy). Of course, self-duality entails more generally that for \emph{any} category-theoretic property of $\mathsf{C}$ (in the sense of being invariant under equivalence of categories), its dual property is also true of $\mathsf{C}$. 

\begin{example}
    The category of finite-dimensional vector spaces over a field $K$ is a $*$-autonomous category with respect to the standard tensor product $\otimes$ and internal hom, with $K$ as the dualising object.
\end{example}

\subsection{The Chu construction}\label{sec_chu_construction}

A general method to construct $*$-autonomous categories from closed symmetric monoidal categories is given by the \emph{Chu construction} (given originally in \cite{chu1978constructing}).

\begin{defn}[Chu construction, \cite{barr1991autonomous}]
    Let $\mathsf{C}$ be a closed symmetric monoidal category and let $D$ be an object. Then the category 
        $$ \mathsf{Chu}_D(\mathsf{C}) $$
    is defined as follows.
    \begin{enumerate}
        \item An object of $\mathsf{Chu}_D(\mathsf{C})$ is a triple $(A, A^\top, a)$, where
        \begin{enumerate}
            \item $A, A^\top$ are objects of $\mathsf{C}$, and 
            \item $a: A \otimes_{\mathsf{C}} A^\top \to D$ is a morphism in $\mathsf{C}$.
        \end{enumerate}
        \item A morphism $(A,A^\top,a)\to (B,B^\top, b)$ in $\mathsf{Chu}_D(\mathsf{C})$ is a pair $(f,f^\top)$, where 
        \begin{enumerate}
            \item $f: A \to B$, $f^\top: B^\top \to A^\top$ are morphisms in $\mathsf{C}$, 
            \item making the following diagram commute: 
            \[\begin{tikzcd}
	       {A \otimes_{\mathsf{C}} B^\top} & {A\otimes_{\mathsf{C}} A^\top} \\
	       {B\otimes_{\mathsf{C}} B^\top} & D
	       \arrow["{f\otimes B^\top}"', from=1-1, to=2-1]
	       \arrow["{A\otimes f^\top}", from=1-1, to=1-2]
	       \arrow[from=2-1, to=2-2]
	       \arrow[from=1-2, to=2-2]
            \end{tikzcd}\]
        \end{enumerate}
    \end{enumerate}
\end{defn}

To understand this definition, it is helpful to consider the case when $\mathsf{C}$ is the category of $\mathbb{K}$-vector spaces and the dualising object is $\mathbb{K}$. In this case, an object of $\mathsf{Chu}_{\mathbb{K}}(\mathsf{Vect})$ is a pair of vector spaces equipped with a bilinear pairing $\langle -, -\rangle$ to the base field $\mathbb{K}$, and a morphism, 
$$(V,V^\top, \langle-,-\rangle_V) \to (W, W^\top, \langle-,-\rangle_W),$$ 
is a pair of linear maps $(f, f^\top)$ that are ``adjoint'' to one another,
    $$ \langle f(v), w^\top \rangle_{W} = \langle v, f(w^\top) \rangle_V. \;\;\;\;\;\;\; (v\in V, w^\top \in W^\top)$$ \par
Let us now describe the dualising object, internal hom and tensor product on ${\mathsf{Chu}}_D(\mathsf{C})$.

\begin{defn}
    Let $\mathsf{C}$ be a closed symmetric monoidal category with pullbacks and let $\mathcal{A}=(A,A^\top,a)$, $\mathcal{B}=(B,B^\top,b)$ be two objects of $\mathsf{Chu}_D(\mathsf{C})$.  
    \begin{enumerate}
        \item Define dualising object of ${\mathsf{Chu}}_D(\mathsf{C})$ as 
        $$\mathcal{D}:= (D, I_{\mathsf{C}}, \rho),$$ 
        where $I_{\mathsf{C}}$ is the tensor unit and 
            $$\rho: D \otimes_{\mathsf{C}} I_{\mathsf{C}} \to D $$
        is the right unitor.
        \item The internal hom 
            $$[\mathcal{A}, \mathcal{B}]_{{\mathsf{Chu}}_D(\mathsf{C})}$$
        is given by the triple $([A, B]_0, A \otimes_{\mathsf{C}} B^\top, c)$, where:
        \begin{enumerate}
        \item $[A, B]_0$ is given by the pullback,
        \[\begin{tikzcd}
	   {[A, B]_0} & {[B^\top,A^\top]_{\mathsf{C}}} \\
	   {[A,B]_\mathsf{C}} & {[A\otimes B^\top,D]_\mathsf{C}}
	   \arrow[from=1-1, to=2-1]
	   \arrow[from=1-1, to=1-2]
	   \arrow["\phi", from=1-2, to=2-2]
	   \arrow["\psi"', from=2-1, to=2-2]
        \end{tikzcd}\]
        Here, $\phi$ is obtained via the tensor-hom adjunction from the composite, 
        \[\begin{tikzcd}
	       {A\otimes_{\mathsf{C}} B^\top \otimes_{\mathsf{C}} [B^\top, A^\top]_{\mathsf{C}}} && {A \otimes_{\mathsf{C}} A^\top} & D
	       \arrow["{\id_A \otimes_{\mathsf{C}} \mathsf{eval}}", from=1-1, to=1-3]
	       \arrow["a", from=1-3, to=1-4]
        \end{tikzcd}\]
        and, similarly, $\psi$ is obtained from the composite, 
        \[\begin{tikzcd}
	       {[A, B]_{\mathsf{C}} \otimes_{\mathsf{C}} A \otimes_{\mathsf{C}} B^\top} && {B \otimes_{\mathsf{C}} B^\top} & D.
	       \arrow["{ \mathsf{eval}\otimes_{\mathsf{C}} \id_{B^\top}}", from=1-1, to=1-3]
	       \arrow["b", from=1-3, to=1-4]
        \end{tikzcd}\]
        \item The definition of $[A, B]_0$ yields a map
            $$ [A, B]_0 \to [A \otimes_{\mathsf{C}} B^\top, D], $$
        which induces the pairing 
            $$c: [A, B]_0 \otimes (A \otimes_{\mathsf{C}} B^\top) \to D$$
        via the tensor-hom adjunction.
        \end{enumerate}
        \item Writing $(-)^*:=[-, \mathcal{D}]_{{\mathsf{Chu}}_D(\mathsf{C})}$, the tensor product of $\mathcal{A}$ and $\mathcal{B}$ is defined as
            $$ \mathcal{A} \otimes_{{\mathsf{Chu}}_D(\mathsf{C})} \mathcal{B} := ([\mathcal{A}, \mathcal{B}^*]_{{\mathsf{Chu}}_D(\mathsf{C})})^*. $$
    \end{enumerate}
\end{defn}

With this structure in place, we obtain \cite[p.~8, Theorem 4.3]{barr1991autonomous}: 

\begin{thm}\label{thm_chu_construction}
    Let $\mathsf{C}$ be a bicomplete closed symmetric monoidal category and let $D$ be an object of $\mathsf{C}$. Then ${\mathsf{Chu}}_D(\mathsf{C})$ is a bicomplete $*$-autonomous category. Moreover, the functor
        $$ \mathsf{C} \hookrightarrow {\mathsf{Chu}}_D(\mathsf{C}), \;\; A \mapsto (A, [A,D], e^{A,D}\circ \sigma^{A, [A,D]}), \; f \mapsto (f, [f,D]) $$
    is fully faithful, strongly symmetric monoidal, as well as left adjoint and right inverse to the functor,
        $$ {\mathsf{Chu}}_D(\mathsf{C}) \to \mathsf{C}, \;\; (A,A^\top, a) \mapsto A, \; (f, f^\top) \mapsto f, $$
\end{thm}

\subsubsection{The extensional-separated Chu construction}\label{sec_sep_ext_chu}

Denoting by $\mathsf{Vect}$ the category of vector spaces over a field $K$, the objects of ${\mathsf{Chu}}_K(\mathsf{Vect})$ are \emph{almost} what is known in the context of functional analysis as a \emph{dual system} (or \emph{paired vector space}, see \cref{defn_paired_vector_space}). The only difference is that for $(V,V^\top,\langle-,-\rangle)$ to be a dual system, one requires the pairing $\langle-,-\rangle$ to be a \emph{perfect pairing}. This means that both 
    $$ V \to (V^\top)',\;\;\; v \mapsto \langle v, - \rangle $$
and 
    $$ V^\top \to V', \;\;\; v^\top \mapsto \langle -, v^\top \rangle $$
are to be injective. The injectivity of the first of these maps expresses that $V^\top$ \emph{separates the points} of $V$. The injectivity of the second map, on the other hand, enables us to view elements of $V^\top$ as genuine functionals on $V$. Put differently, it expresses \emph{function extensionality} for elements of $V^\top$. \par 
These separation and extensionality requirements can be extended to the general case of the Chu construction to yield a further way of constructing $*$-autonomous categories.

\begin{defn}[Extensional/separated objects]\label{defn_extnl_sep_objects}
    Let $\mathsf{C}$ be a coregular, closed symmetric monoidal category and let $D$ be an object of $\mathsf{C}$. An object $(A,A^\top,a)$ of ${\mathsf{Chu}}_D(\mathsf{C})$ is called
    \begin{enumerate}
        \item \emph{extensional} if the morphism,
            $$ \phi: A^\top \to [A, D]_{\mathsf{C}}, $$
        induced from $a$ (via the tensor-hom adjunction), is a regular monomorphism, and
        \item \emph{separated} if the morphism, 
            $$ \psi: A \to [A^\top, D]_{\mathsf{C}}, $$
        induced from $a$, is a regular monomorphism.
    \end{enumerate}
    Define the \emph{separation} (or \emph{pair separation} in distinction to the $D$-separation of \cref{sepfun_defn}) of $(A,A^\top,a)$ as the triple $(\im(\psi), A^\top, \tilde{a})$, where 
    \begin{enumerate}
        \item $\im(\psi)$ is the regular image of $\psi$, and
        \item $\tilde{a}$ is given by the composite,
            $$ \im(\psi) \otimes_{\mathsf{C}} A^\top \to [A^\top, D] \otimes_{\mathsf{C}} A^\top  \to D. $$
    \end{enumerate}
    Similarly, define the \emph{extensionalisation} as the triple $(A, \im(\phi), \overline{a})$, where 
    \begin{enumerate}
        \item $\im(\phi)$ is the regular image of $\phi$, and
        \item $\overline{a}$ is the composite,
            $$ A \otimes_{\mathsf{C}} \im(\phi) \to A \otimes_{\mathsf{C}} [A, D] \to D. $$
    \end{enumerate}
    Finally, denote full subcategories of extensional (separated, resp.) objects of ${\mathsf{Chu}}_D(\mathsf{C})$ by $e \mathsf{Chu}_D(\mathsf{C})$ ($s \mathsf{Chu}_D(\mathsf{C})$, resp.), and write 
        $$ \widehat{\mathsf{Chu}}_D(\mathsf{C})$$
    for the full subcategory of ${\mathsf{Chu}}_D(\mathsf{C})$ spanned by those objects which are both separated \emph{and} extensional. 
\end{defn}

\begin{remark}
    In \cite{barr1998separated}, Barr defines the separated-extensional Chu category $\widehat{\mathsf{Chu}}_D(\mathsf{C})$ with respect to arbitrary factorisation systems. Coregularity of the category $\mathsf{C}$ ensures the existence of unique regular image factorisations, so that our definition is that of Barr for the special case of the regular image factorisation system. Moreover, this factorisation system also has the following two properties that Barr requires of the given factorisation system $\mathcal{E}/\mathcal{M}$.
    \begin{enumerate}
        \item Every morphism in $\mathcal{E}$ is an epimorphism.
        \item For every object $A$, if $m\in \mathcal{M}$, then $[A,m]_{\mathsf{C}} \in \mathcal{M}$.
    \end{enumerate}
    For the regular image factorisation system in a coregular closed symmetric monoidal category, the first condition holds by definition of this factorisation system, while the other one follows from then fact that the internal hom preserves limits in the second variable. 
\end{remark}

\begin{example}
    The objects of $\widehat{\mathsf{Chu}}_{\mathbb{K}}(\mathsf{Vect})$ are now exactly dual systems, or ``paired vector spaces'', in the sense of functional-analytic duality theory (mentioned before in the beginning of \cref{sec_sep_ext_chu}, see \cref{defn_paired_vector_space} for a definition).
\end{example}

We will need the following facts on the relation between the extensionalisation and (pair) separation functors $e$ and $s$ \cite[Proposition 3.2, Proposition 3.4]{barr1998separated}:

\begin{prop}\label{separation_extn_prop}
    Lt $\mathsf{C}$ be a coregular closed symmetric monoidal category and let $D$ be an object of $\mathsf{C}$. The extensionalisation $e$ and separation $s$ extend to functors that are right (left, resp.) adjoint to the inclusions of the full subcategories of extensional (separated, resp.) objects. \par
    Moreover, the tensor unit of $\mathsf{Chu}_D(\mathsf{C})$ is extensional and the full subcategory $e\mathsf{Chu}(\mathsf{C})$ of extensional objects is closed under tensor products. Hence, $e\mathsf{Chu}(\mathsf{C})$ is a symmetric monoidal category and the inclusion functor from $\mathsf{C}$ into $\mathsf{Chu}(\mathsf{C})$ co-restricts to a strong symmetric monoidal functor 
        $$ \mathsf{C} \to e\mathsf{Chu}(\mathsf{C}). $$
\end{prop}

We now describe the closed monoidal structure on $\widehat{\mathsf{Chu}}_D(\mathsf{C})$.

\begin{defn}\label{cl_mon_str_on_chu}
    Define 
    \begin{enumerate}
        \item the dualising object and tensor unit of $\widehat{\mathsf{Chu}}_D(\mathsf{C})$ as $e\mathcal{D}$ and $s\mathcal{I}$, respectively (where $\mathcal{D}$ and $\mathcal{I}$ are those of ${\mathsf{Chu}}_D(\mathsf{C})$),
        \item the internal hom of $\widehat{\mathsf{Chu}}_D(\mathsf{C})$ as
            $$ [\mathcal{A}, \mathcal{B}]_{\widehat{\mathsf{Chu}}_D(\mathsf{C})} := e[\mathcal{A}, \mathcal{B}]_{{\mathsf{Chu}}_D(\mathsf{C})}, $$
        \item and the tensor product as 
            $$ \mathcal{A} \otimes_{\widehat{\mathsf{Chu}}_D(\mathsf{C})} \mathcal{B} := s(\mathcal{A} \otimes_{{\mathsf{Chu}}_D(\mathsf{C})} \mathcal{B}). $$
    \end{enumerate}
\end{defn}

\begin{thm}\label{thm_on_little_chu}
    Let $\mathsf{C}$ be a coregular bicomplete closed symmetric monoidal category and let $D$ be an object. Then, with the structure described in \cref{cl_mon_str_on_chu}, $\widehat{\mathsf{Chu}}_D(\mathsf{C})$ is a bicomplete $*$-autonomous category. Moreover, the functor 
        $$ \mathsf{C} \to \widehat{\mathsf{Chu}}_D(\mathsf{C}), \;\; A \mapsto s(A, [A,D], e^{A,D}\circ \sigma^{A, [A,D]}) $$
    is a strong symmetric monoidal functor.
\end{thm}
\begin{proof}
    For $*$-autonomy of $\widehat{\mathsf{Chu}}_D(\mathsf{C})$, see \cite{barr1998separated}. For the second part of the statement, notice the above functor is the composite of the functors
    \[\begin{tikzcd}
	{\mathsf{C}} & {e\mathsf{Chu}(\mathsf{C})} & {\widehat{\mathsf{Chu}}_D(\mathsf{C}),}
	\arrow[from=1-1, to=1-2]
	\arrow["s", from=1-2, to=1-3]
    \end{tikzcd}\]
    the first of which is strong symmetric monoidal by \cref{separation_extn_prop}. The second, on the other hand, is a \emph{strict} monoidal functor, by definition of the tensor product on $\widehat{\mathsf{Chu}}_D(\mathsf{C})$: 
        $$ \mathcal{A} \otimes_{\widehat{\mathsf{Chu}}_D(\mathsf{C})} \mathcal{B} = s(\mathcal{A} \otimes_{{\mathsf{Chu}}_D(\mathsf{C})} \mathcal{B}) = s(\mathcal{A} \otimes_{e\mathsf{Chu}_D(\mathsf{C})} \mathcal{B}) $$
    The claim now follows from the fact that a composite of strong symmetric monoidal functors is strong symmetric monoidal.
\end{proof}

One important difference between the categories $\widehat{\mathsf{Chu}}_D(\mathsf{C})$ and ${\mathsf{Chu}}_D(\mathsf{C})$ is the following.

\begin{prop}
    The functor 
        $$ \widehat{\mathsf{Chu}}_D(\mathsf{C}) \to \mathsf{C}, \;\;\; (A, A^\top, a) \mapsto A$$
    is faithful.
\end{prop}
\begin{proof}
    If $(f,f_0^\top),(f,f_1^\top):(A,A^\top, a) \to (B,B^\top, b)$ are two morphisms in $\widehat{\mathsf{Chu}}_D(\mathsf{C})$ inducing the same morphism $f$ in $\mathsf{C}$ under the above functor, then, since the diagram 
    \[\begin{tikzcd}
	{B^\top} & {A^\top} \\
	{[B,D]} & {[B,D]}
	\arrow["{f_0^\top}", shift left=1, from=1-1, to=1-2]
	\arrow[hook, from=1-1, to=2-1]
	\arrow[hook', from=1-2, to=2-2]
	\arrow["{[f,D]}"', from=2-1, to=2-2]
	\arrow["{f_1^\top}"', shift right=1, from=1-1, to=1-2]
    \end{tikzcd}\]
    commutes and the morphism $A^\top \to [A,D]$ is a monomorphism, $f_0^\top=f_1^\top$.
\end{proof}

We may hence interpret the above functor as a forgetful functor. Under this interpretation, the objects of $\widehat{\mathsf{Chu}}_D(\mathsf{C})$ are $\mathsf{C}$-objects with some additional structure (which the forgetful functor ``forgets''), and the morphisms in $\widehat{\mathsf{Chu}}_D(\mathsf{C})$ are precisely those $\mathsf{C}$-morphisms between (underlying $\mathsf{C}$-objects of) objects of $\widehat{\mathsf{Chu}}_D(\mathsf{C})$ which preserve this additional structure. \par 
In particular, if $\mathsf{C}$ is a concrete category (i.e.~equipped with a faithful functor to the category of sets), then so is $\widehat{\mathsf{Chu}}_D(\mathsf{C})$ in a canonical way. This is in sharp contrast to the larger category ${\mathsf{Chu}}_D(\mathsf{C})$. For example, the triple $(0, V, 0)$, where $V$ is \emph{any} (algebraic) vector space, is an object of ${\mathsf{Chu}}_D(\mathsf{Vect}(\mathsf{Set}))$. Such ``degenerate'' examples cannot occur in $\widehat{\mathsf{Chu}}_D(\mathsf{Vect}(\mathsf{Set}))$. \par

\subsubsection{Relation between the separated-extensional Chu construction and $D$-separated objects} A natural question is now which objects of $\mathsf{C}$ can arise as the underlying $\mathsf{C}$-object of an an object of $\widehat{\mathsf{Chu}}_D(\mathsf{C})$. In other words: which $\mathsf{C}$-objects admit the structure of an object of $\widehat{\mathsf{Chu}}_D(\mathsf{C})$? These are precisely the $D$-separated objects. 

\begin{prop}\label{sepobjs_corefl_subcat_of_little_chu}
    Let $\mathsf{C}$ be a coregular bicomplete closed symmetric monoidal category and let $D\in \mathsf{C}$. Then an object $A$ of $\mathsf{C}$ is $D$-separated if, and only if, it is the underlying $\mathsf{C}$-object of an object of $\widehat{\mathsf{Chu}}_D(\mathsf{C})$ (i.e.~there is an object of $\widehat{\mathsf{Chu}}_D(\mathsf{C})$ of the form $(A, A^\top, a)$). As a consequence, the functor 
        $$ \iota : \sepfun_D(\mathsf{C}) \hookrightarrow \widehat{\mathsf{Chu}}_D(\mathsf{C}), \;\; A \mapsto (A, [A,D], e^{A,D}\circ \sigma^{A, [A,D]}), \; f \mapsto (f, [f,D]) $$
    is well-defined, fully faithful and left adjoint as well as right inverse to
        $$ \widehat{\mathsf{Chu}}_D(\mathsf{C}) \to \sepfun_D(\mathsf{C}), \;\;\; (A, A^\top, a) \mapsto A. $$
    In other words, $\sepfun_D(\mathsf{C})$ is a coreflective subcategory of $\widehat{\mathsf{Chu}}_D(\mathsf{C})$ and we have the following diagram of functors, in which the solid square of inclusions commutes:
    \[\begin{tikzcd}
	{\sepfun_D(\mathsf{C})} & {\widehat{\mathsf{Chu}}_D(\mathsf{C})} \\
	{\mathsf{C}} & {\mathsf{Chu}_D(\mathsf{C})}
	\arrow[""{name=0, anchor=center, inner sep=0}, shift right=2, hook, from=1-1, to=1-2]
	\arrow[""{name=1, anchor=center, inner sep=0}, shift right=2, hook, from=1-1, to=2-1]
	\arrow[""{name=2, anchor=center, inner sep=0}, shift right=2, hook, from=2-1, to=2-2]
	\arrow[hook, from=1-2, to=2-2]
	\arrow[""{name=3, anchor=center, inner sep=0}, shift right=2, dashed, from=1-2, to=1-1]
	\arrow[""{name=4, anchor=center, inner sep=0}, shift right=2, dashed, from=2-2, to=2-1]
	\arrow[""{name=5, anchor=center, inner sep=0}, shift right=2, dashed, from=2-1, to=1-1]
	\arrow["\dashv"{anchor=center, rotate=90}, draw=none, from=0, to=3]
	\arrow["\dashv"{anchor=center, rotate=90}, draw=none, from=2, to=4]
	\arrow["\dashv"{anchor=center, rotate=-180}, draw=none, from=5, to=1]
    \end{tikzcd}\]
\end{prop}
\begin{proof}
    If $A\in \mathsf{C}$ is $D$-separated, then $(A,[A,D],e^{A,D}\circ \sigma^{A, [A,D]})$ is an object of $\widehat{\mathsf{Chu}}_D(\mathsf{C})$. Conversely, if $(A,A^\top,a) \in \widehat{\mathsf{Chu}}_D(\mathsf{C})$, then we have a regular monomorphism, 
        $$ A \to [A^\top,D], $$
    and by \cref{duals_and_subobjects_sep}, $A$ is $D$-separated. The other claims now follow from \cref{thm_chu_construction}.
\end{proof}

\subsection{Closed symmetric monoidal categories over cartesian closed categories}\label{sec_closed_symm_mon_cats_over_cart_cl_cats}

The following notion captures when a given closed symmetric monoidal (or $*$-autonomous, in particular) category stands in a good relationship with a cartesian closed ``base category''.

\begin{defn}\label{defn_aut_cat_over_ccc}
    A \emph{closed-symmetric-monoidal-over-cartesian-closed category} $|\cdot|\vdash T: \mathsf{D}\to \mathsf{C}$ is a closed symmetric monoidal category $\mathsf{D}$, together with a cartesian closed category $\mathsf{C}$ (the \emph{base category}) and a symmetric monoidal adjunction (\cref{symmmon_adj}), 
    \[\begin{tikzcd}
	{\mathsf{C}} & {\mathsf{D}}
	\arrow[""{name=0, anchor=center, inner sep=0}, "U"', from=1-1, to=1-2]
	\arrow[""{name=1, anchor=center, inner sep=0}, "T"', shift right=3, from=1-2, to=1-1]
	\arrow["\dashv"{anchor=center, rotate=-90}, draw=none, from=1, to=0]
    \end{tikzcd}\]
    By a bicomplete, coregular closed-symmetric-monoidal-over-cartesian-closed category $|\cdot|\vdash T: \mathsf{D}\to \mathsf{C}$ we mean a closed-symmetric-monoidal-over-cartesian-closed category for which both $\mathsf{D}$ and $\mathsf{C}$ are bicomplete (i.e.~have all limits and colimits) and coregular. A \emph{$*$-autonomous-over-cartesian-closed category} is a closed-symmetric-monoidal-over-cartesian-closed category $|\cdot|\vdash T: \mathsf{D}\to \mathsf{C}$ where $\mathsf{D}$ is $*$-autonomous. \par 
    We cal call $\mathsf{C}$ the base category and the adjunction $|\cdot|\vdash T$ the free-forgetful adjunction.
\end{defn}

\begin{remark}
    A definition similar to \cref{defn_aut_cat_over_ccc} is given in \cite[p. 12, Definition 6]{de2014categorical} under the name \emph{linear-non-linear category}, in the context of (the categorical semantics of) \emph{linear logic}.
\end{remark}

In our context, one should think of $\mathsf{D}$ as a category of linear spaces (e.g.~vector spaces or linear $hk$-spaces) living over a base category $\mathsf{C}$ of ``general'' spaces (e.g.~sets or $hk$-spaces). 

\par 
The free-forgetful adjunction defining a closed-symmetric-monoidal-over-cartesian-closed category is always enriched over the base category:

\begin{prop}\label{prop_free_forgetful_enriched}
    Let $|\cdot|\vdash T: \mathsf{D}\to \mathsf{C}$ be a closed-symmetric-monoidal-over-cartesian-closed category. Then we have a natural isomorphism, 
        $$ [X, V]_{\mathsf{C}} \cong |[TX, V]_{\mathsf{D}}|. \qquad (X\in \mathsf{C}, V\in\mathsf{D})$$
\end{prop}
\begin{proof}
    Using that the free-forgetful adjunction is a monoidal adjunction, we obtain a sequence of natural isomorphisms, 
    \begin{align*}
    \Hom_{\mathsf{C}}(Y, [X, V]_{\mathsf{C}}) & \cong \Hom_{\mathsf{C}}(Y \times X, V) \\
    & \cong \Hom_{\mathsf{D}}(T(Y\times X), V) \\
    & \cong \Hom_{\mathsf{D}}(TY\otimes_{\mathsf{D}} TX, V) \\
    & \cong \Hom_{\mathsf{D}}(TY, [TX, V]_{\mathsf{D}}) \\
    & \cong \Hom_{\mathsf{C}}(Y, |[TX, V]_{\mathsf{D}}|). \qquad (Y\in \mathsf{C})
    \end{align*} 
    The claim now follows from the Yoneda lemma.
\end{proof}

\subsubsection{Free $D$-separated objects} The notion of $D$-separated object is compatible with that of closed-symmetric-monoidal-over-cartesian-closed category, in the following sense.

\begin{prop}
    Let $|\cdot|\vdash T: \mathsf{D}\to \mathsf{C}$ be a closed-symmetric-monoidal-over-cartesian-closed category and let $D\in \mathsf{D}$ be an object. Then 
        $$ T_D := \sepfun_D \circ T $$
    is left adjoint to the ``forgetful'' functor
        $$ |\cdot| \circ \iota : \sepfun_D \mathsf{D} \to \mathsf{C}, $$
    where 
        $$ \iota: \sepfun_D \mathsf{D} \to \mathsf{D}$$
    is the inclusion functor, and 
        $$|\cdot|\circ \iota \vdash T_D: \mathsf{D}\to \mathsf{C}$$ 
    is again a closed-symmetric-monoidal-over-cartesian-closed category.
\end{prop}
\begin{proof}
    The $D$-separation $\sepfun_D$ is left adjoint to the inclusion functor $\iota$ and this is a symmetric monoidal adjunction, since $D$-separated objects form an exponential ideal in $\mathsf{D}$. Composing the two symmetric monoidal adjunctions $ \iota \vdash \sepfun_D$ and $|\cdot|\vdash T$ gives the desired 
        $$|\cdot|\circ \iota \vdash T_D.$$
\end{proof}

\subsubsection{The free replete linear $hk$-space} As a particular case, the category $\sclin$ of replete linear $hk$-spaces also admits \emph{free objects}.

\begin{defn}[Free replete linear $hk$-space]
    Let $X$ be an $hk$-space. Define the \emph{free linear $hk$-space} on $X$,
        $$ \mathcal{M}_c(X) := rF_{\mathbb{K}}(X), $$
    as the repletion of the free linear $hk$-space on $X$.
\end{defn}

\begin{remark}
    The notation $\mathcal{M}_c(X)$ will be justified in Chapter 4, where we will show that this object can be identified with the space of compactly supported Radon measures on $X$, when $X$ is a Hausdorff $k$-space.
\end{remark}

\begin{warning}\label{warning_free_replete_vs_free_paired}
    We will use the same notation for the free \emph{paired} linear $hk$-space (and have already done so in the introduction). This unproblematic, since these two objects agree, in the sense that the free replete linear $hk$-space $\mathcal{M}_c(X)$ carries a natural paired linear $hk$-space structure under which it becomes the free paired linear $hk$-space.
\end{warning}

\begin{cor}[Universal property of $\mathcal{M}_c(X)$]\label{cor_univ_prop_free_repl_sp}
    Let $X$ be an $hk$-space, let $V$ be a replete linear $hk$-space and $f:X\to V$ be a continuous map. Then there exists a unique continuous linear map $\tilde{f}: \mathcal{M}_c(X) \to V$ making the diagram 
    \[\begin{tikzcd}
	{\mathcal{M}_c(X)} & V \\
	X
	\arrow["{\tilde{f}}", dashed, from=1-1, to=1-2]
	\arrow["{\delta_\bullet}", from=2-1, to=1-1]
	\arrow["f"', from=2-1, to=1-2]
    \end{tikzcd}\]
    commute, where $\delta_\bullet$ is the composite of the canonical maps,
        $$ X \to F_{\mathbb{K}}(X) \to rF_{\mathbb{K}}(X). $$
\end{cor}

\subsubsection{A construction theorem for $*$-Autonomous-over-cartesian-closed categories}

The following is a general construction theorem which uses the separated-extensional Chu construction to obtain a $*$-autonomous-over-cartesian-closed category $U\vdash M_D: \widehat{\mathsf{Chu}}_D(\mathsf{D})\to \mathsf{C}$ from a general closed-symmetric-monoidal-over-cartesian-closed category $|\cdot|\vdash T: \mathsf{D}\to \mathsf{C}$ in such a way that if $|\cdot|: \mathsf{D}\to \mathsf{C}$ is a faithful functor, then so is  $U: \widehat{\mathsf{Chu}}_D(\mathsf{D})\to \mathsf{C}$. 

\begin{thm}\label{construction_thm_for_star_aut_cats_over_cccs}
    Let $|\cdot|\vdash T: \mathsf{D}\to \mathsf{C}$ be a bicomplete, coregular closed-symmetric-monoidal-over-cartesian-closed category $\mathsf{C}$, and let $D\in \mathsf{D}$ be an object. Then the forgetful functor,  
        $$ U: \widehat{\mathsf{Chu}}_D(\mathsf{D}) \to \mathsf{C}, \;\;\; (A,A^\top, \langle -,- \rangle)\mapsto |A|, $$
    has a strong symmetric monoidal left adjoint $M_D$. Hence, the resulting adjunction,
        \[\begin{tikzcd}
	   {\widehat{\mathsf{Chu}}_D(\mathsf{D})} & {\mathsf{C}},
	   \arrow["U"', shift right=2, from=1-1, to=1-2]
	   \arrow["{M_D}"', shift right=2, from=1-2, to=1-1]
        \end{tikzcd}\]
    defines a $*$-autonomous-over-cartesian-closed category. 
\end{thm}
\begin{proof}
    By \cref{thm_sepfun_left_adjoint} and \cref{sepobjs_corefl_subcat_of_little_chu}, we have the following adjunctions,
    \[\begin{tikzcd}
	{\mathsf{C}} & {\mathsf{D}} & {\sepfun_D(\mathsf{D})} & {\widehat{\mathsf{Chu}}_D(\mathsf{D})}
	\arrow[""{name=0, anchor=center, inner sep=0}, shift left=2, from=1-2, to=1-1]
	\arrow[""{name=1, anchor=center, inner sep=0}, shift left=2, from=1-3, to=1-2]
	\arrow[""{name=2, anchor=center, inner sep=0}, shift left=2, from=1-4, to=1-3]
	\arrow[""{name=3, anchor=center, inner sep=0}, "T", shift left=2, from=1-1, to=1-2]
	\arrow[""{name=4, anchor=center, inner sep=0}, "{\sepfun_D}", shift left=2, from=1-2, to=1-3]
	\arrow[""{name=5, anchor=center, inner sep=0}, "\iota", shift left=2, hook, from=1-3, to=1-4]
	\arrow["\dashv"{anchor=center, rotate=-90}, draw=none, from=3, to=0]
	\arrow["\dashv"{anchor=center, rotate=-90}, draw=none, from=5, to=2]
	\arrow["\dashv"{anchor=center, rotate=-90}, draw=none, from=4, to=1]
    \end{tikzcd}\]
    Hence, 
        $$ M_D := \iota \circ \sepfun_D \circ T $$
    is the desired left adjoint. It remains to show that $M_D$ is (or rather, admits the structure of a) strong symmetric monoidal functor. To see this, notice that $M_D$ also factors as the composite of two strong symmetric monoidal functors, 
    \[\begin{tikzcd}
	{\mathsf{C}} & {\mathsf{D}} & {\widehat{\mathsf{Chu}}_D(\mathsf{D}),}
	\arrow["T", from=1-1, to=1-2]
	\arrow[from=1-2, to=1-3]
    \end{tikzcd}\]
    where the functor on the right is the functor as described in \cref{thm_on_little_chu}.
\end{proof}

\subsubsection{Paired linear $hk$-spaces}\label{sec_paired_lin_ksp} Applying the extensional-separated Chu construction to the category of linear $hk$-spaces, we obtain \emph{paired linear $hk$-spaces}.

\begin{defn}\label{defn_paired_lin_hk_sp}
    A \emph{paired linear $hk$-space} is an object of $\widehat{\mathsf{Chu}}_{\mathbb{K}}(\vect)$, i.e.~a linear $hk$-space $V$ together with a further linear $hk$-space $V^*$ and a bilinear map 
        $$ \langle -, - \rangle: V \times V^* \to \mathbb{K} $$
    such that both 
        $$ V \to (V^*)^\wedge, \;\; x \mapsto \langle x, - \rangle $$
    and 
        $$ V^* \to V^\wedge, \;\; \phi \mapsto \langle - , \phi \rangle $$
    are regular monomorphisms in $\vect$.
\end{defn}

Since, by \cref{cor_morphisms_in_vect}, regular monomorphisms in $\vect$ are equivalently closed embeddings, we can reformulate this definition as follows, identifying $V^*$ with its image in $V^\wedge$.

\begin{defn}[Paired linear $hk$-space, more explicitly]
    A \emph{paired linear $hk$-space} is a linear $hk$-space $V$ together with a closed linear subspace $V^*\subseteq V^\wedge$ of the natural dual such that
        $$ V \to (V^*)^\wedge, \;\; x \mapsto \langle x, - \rangle $$
    is a closed embedding. A morphism of paired linear $hk$-spaces $V,W$ is a continuous linear map $f:V\to W$ such that $f^\wedge(W^*)\subseteq V^*$, where 
        $$ f^\wedge: W^\wedge \to V^\wedge, \;\; \phi \mapsto \phi \circ f $$
    is the adjoint of $f$. We denote the category of paired linear $hk$-spaces by 
        $$ \plin := \widehat{\mathsf{Chu}}_{\mathbb{K}}(\vect). $$
    We denote the internal hom and tensor product of paired linear $hk$-spaces simply by 
        $$ [V, W] := [V, W]_{\plin}, \;\; V \ptimes W := V \otimes_{\plin} W \qquad \qquad (V,W\in \plin)$$
    as long as there is no potential for ambiguity. The underlying linear $hk$-space of $[V, \mathbb{K}]$ is exactly $V^*$, so it is justified to also write  
        $$ V^* := [V, \mathbb{K}] \qquad\qquad (V\in \plin) $$
    for the dual, viewed as a paired linear $hk$-space.
\end{defn}

As a special case of \cref{construction_thm_for_star_aut_cats_over_cccs}, we immediately obtain that $\plin$ is a $*$-autonomous category with respect to the above tensor product and internal hom. 

\begin{remark}
    By construction, the internal hom $[V,W]$ of paired linear $hk$-spaces $V,W$ is the closed subspace 
        $$ [V, W] \subseteq L(V,W) $$
    of the space of continuous linear maps $L(V,W)$ (the internal hom in $\vect$ and also $\sclin$) consisting of those continuous linear maps which are morphisms of paired linear $hk$-spaces. 
\end{remark}

\subsubsection{The category of replete linear $hk$-spaces embeds into the category of paired linear $hk$-spaces} A wide range of examples of paired linear $hk$-spaces is supplied by considering replete linear $hk$-spaces as paired linear $hk$-spaces, as follows.

\begin{prop}
    The functor 
        $$ \iota: \sclin \hookrightarrow \plin, \;\; V \mapsto (V, V^\wedge), \; f\mapsto f^\wedge$$
    is fully faithful, displaying $\sclin$ as a coreflective subcategory of $\plin$.
\end{prop}
\begin{proof}
    This is a special case of \cref{sepobjs_corefl_subcat_of_little_chu}.
\end{proof}

As a consequence, when $W$ is any paired linear $hk$-space and $V$ is in the essential image of the inclusion functor,
        $$ \sclin \hookrightarrow \plin, $$
i.e. $V^*=V^\wedge$, then every continuous linear map $V\to W$ is a morphism of paired linear $hk$-spaces and we have, 
        $$ [V, W] = L(V,W). $$

\subsubsection{Fréchet and Brauner spaces are paired linear $k$-spaces in a unique way}

While for a general replete linear $hk$-space $V$, there may a priori be several paired linear $hk$-spaces whose underlying linear $hk$-space is $V$, this is not possible for Fréchet or Brauner spaces: their dual is, in some sense, uniquely determined.

\begin{prop}\label{Frechet_space_paired_linear_k_space_in_unique_way}
    Let $V$ be a paired linear $hk$-space with dual $V^*$ and suppose that the underlying linear $hk$-space of $V$ is a Fréchet space or a Brauner space. Then $V^*=V^\wedge$.
\end{prop}
\begin{proof}
    Notice that the assumption that
        $$ j: V \to (V^*)^\wedge, \;\;\;x \mapsto \langle x , - \rangle $$
    is a closed embedding implies that  
        $$ j^\wedge: (V^*)^{\wedge\wedge} \to V^\wedge, \;\;\; \chi \mapsto (x \mapsto \chi(\langle x , - \rangle)) $$
    is surjective, using the Hahn-Banach extension theorem (which is applicable because both $V$ and $(V^*)^\wedge$ are Fréchet/Brauner spaces). By Smith duality, the canonical map
        $$ \eta^{V^*}: V^* \to (V^*)^{\wedge\wedge} $$
    is an isomorphism. Hence, observing that 
        $$ j^\wedge \circ \eta^{V^*} : V^* \to V^\wedge, \;\; \phi \mapsto \phi, $$
    is the embedding of $V^*$ into $V^\wedge$, we see that this map is surjective and hence the identity.
\end{proof}

\subsubsection{The free paired linear $hk$-space}

By \cref{construction_thm_for_star_aut_cats_over_cccs}, the forgetful functor $\plin \to \spaces$ admits a left adjoint, and this left adjoint is given by the composition of the left adjoint $\mathcal{M}_c(X)$ to the forgetful functor $\sclin\to \spaces$, followed by the inclusion of $\sclin$ into $\plin$. Hence, we are justified in defining (see also \cref{warning_free_replete_vs_free_paired}): 

\begin{defn}\label{defn_free_plin_ksp}
    We denote the left adjoint to the forgetful functor $\plin \to \spaces$ by $\mathcal{M}_c$. For an $hk$-space $X$, we call $\mathcal{M}_c(X)$ the \emph{free paired linear $hk$-space} on $X$.
\end{defn}

\begin{remark}\label{remark_on_Mc_spc_of_meas_and_rel_to_vv_int}
    The notation $\mathcal{M}_c(X)$ will be justified in Chapter 4, where we will show that the elements of $\mathcal{M}_c(X)$ can be identified with certain compactly supported measures. Moreover, by \cref{prop_free_forgetful_enriched}, the free-forgetful adjunction between paired linear $hk$-spaces and $hk$-spaces is $\spaces$-enriched, 
        $$ [\mathcal{M}_c(X), V] \cong C(X, V), \qquad (X\in \spaces, V\in \plin)$$
    and we will see in Chapter 4 that this natural isomorphism is given by vector-valued integration (see \cref{remark_vv_integral_from_free_forget_adj}).
\end{remark}

\begin{remark}\label{comm_mon_str_of_Mc}
    Since by \cref{construction_thm_for_star_aut_cats_over_cccs}, the free-forgetful adjunction between paired linear $hk$-spaces and $hk$-spaces is a symmetric monoidal adjunction, it induces a symmetric monoidal, i.e.~commutative, monad on $\spaces$ with underlying endofunctor $X\mapsto \mathcal{M}_c(X)$. 
\end{remark}

%% file: Chapters/4_Monadic_Measure_and_Integration_Theory.tex
\chapter{Monadic Measure and Integration Theory}

In this chapter, we will harvest the fruits of the general abstract developments of Chapter 3. To this end, we will first define the linear $hk$-space $C_b(X)$ of continuous bounded functions on an $hk$-space $X$. Then, we will construct $\mathcal{M}(X)$ as a space of continuous linear functionals on $C_b(X)$ and endow $\mathcal{M}(X)$ as well as $C_b(X)$ with the structure of a paired linear $hk$-space, rendering them mutually dual. In \cref{sec_k_reg_measures_section}, a version of the Riesz representation theorem will allow us to identify the elements of $\mathcal{M}(X)$ with a certain class of measures which we call \emph{$k$-regular measures}. Next, we will construct the monad structure of $\mathcal{M}$ and the probability monad $\mathcal{P}$ in \cref{sec_M_comm_monad,sec_prob_mon_P}. Finally, we will show in \cref{sec_monadic_vec_int} how the formalism of (paired or replete) linear $hk$-spaces can be used to obtain a very coherent theory of vector-valued integration, \emph{monadic vector-valued integration}. 

\section{The Paired Linear $hk$-space $C_b(X)$ and its Dual $\mathcal{M}(X)$}

\subsection{The linear $hk$-space $C_b(X)$}

First, we need to define $C_b(X)$ as a linear $hk$-space. (The \emph{paired} linear structure on $C_b(X)$ will come later.)

\begin{defn}
    We define the linear $hk$-space $C_b(X)$ of bounded continuous ($\mathbb{K}$-valued) functions on an $hk$-space $X$ as the filtered colimit (in the category of $hk$-spaces),
        $$ C_b(X) := \colim_{n\in\mathbb{N}}\: (nD)^X, $$
    of the diagram,
        $$ D^X \hookrightarrow (2D)^X \hookrightarrow (3D)^X \hookrightarrow ... ,  $$
    where 
        $$ D = \{\lambda\in \mathbb{K} \mid |\lambda| \leq 1\} \subseteq \mathbb{K} $$
    is the unit disk.
\end{defn}

\begin{remark}
    Since the underlying set of this colimit can be identified with the union of all functions whose absolute value is bounded by some natural number, the elements of $C_b(X)$ correspond exactly to bounded continuous functions on $X$, justifying the notation. Moreover, (pointwise) addition and scalar multiplication are continuous with respect to the above colimit topology. Hence, $C_b(X)$ is indeed a linear $hk$-space.
\end{remark}

Convergent sequences in $C_b(X)$ are simple to describe.

\begin{lem}\label{convergent_sequences_in_Cb}
    Let $X$ be an $hk$-space. Then a sequence $(f_n)$ of continuous bounded functions converges to $f$ in $C_b(X)$ if, and only if, $(f_n)$ is uniformly bounded and converges in the compact-open topology to $f$.
\end{lem}
\begin{proof}
     A convergent sequence in $C_b(X)$ can be identified with a continuous map 
        $$ \mathbb{N}\cup\{\infty\} \to C_b(X) $$
    from the one-point compactification of the natural numbers to $C_b(X)$. The image of such map being compact, \cref{seq_colims_in_spaces} implies that every convergent sequence lies in one of the spaces $(nD)^X$ (for some $n\in \mathbb{N}$) and is hence uniformly bounded. Moreover, a sequence converges in $(nD)^X$ if, and only if, it converges in the compact-open topology (see \cref{lem_convergence_in_Y_X}).
\end{proof}

We will need the following lemma on the continuous linear functionals on $C_b(X)$ later.

\begin{lem}\label{continuity_on_Cb}
    Let $X$ be an $hk$-space and let $\phi$ be a (not necessarily continuous) functional on $C_b(X)$. Suppose that for every uniformly bounded net of continuous functions $(f_i)$ on $X$ converging in the compact-open topology to $0$, $\phi(f_i) \to 0$. Then $\phi$ is continuous.
\end{lem}
\begin{proof}
    By assumption, $\phi$ is continuous on each space of continuous functions $C_{c.o.}(X, nD)$ with the compact-open topology, implying that it is also continuous on each $(nD)^X$. The claim follows, since $C_b(X)$ is given as the colimit over these spaces. 
\end{proof}

\subsection{The linear $hk$-space $\mathcal{M}(X)$} 

At this point, as for $C_b(X)$, we introduce $\mathcal{M}(X)$ only as a linear $hk$-space. Later, $\mathcal{M}(X)$ will also be given a paired linear structure, rendering it dual to $C_b(X)$. Moreover, while the following definition of $\mathcal{M}(X)$ is in terms of functionals, we will later identify $\mathcal{M}(X)$ with a space of measures.

\begin{defn}
    Define 
     $$ \mathcal{M}(X) := \overline{\vspan \{\delta_x \in C_b(X)^\wedge \mid \, x\in X \}} \subseteq C_b(X)^\wedge $$
    as the closed linear subspace of the natural dual $C_b(X)^\wedge$ of $C_b(X)$ generated by the point evaluations (``Dirac functionals'', $\delta_x(f) := f(x)$).
\end{defn}

This definition ensures from the start that finitely supported measures will be dense in $\mathcal{M}(X)$ (once we have shown how to identify elements of $\mathcal{M}(X)$ with measures), a very convenient feature. (For instance, we will use this in the proof of \cref{M_bbd_meas_monad_commutative}, the commutativity of $\mathcal{M}$ as a monad.) Moreover, as a closed linear subspace of a replete linear $hk$-space, $\mathcal{M}(X)$ is a replete linear $hk$-space.

\subsection{The natural dual of $\mathcal{M}(X)$ is $C_b(X)$} 

The next lemma will be needed for showing that $\mathcal{M}(X)^\wedge \cong C_b(X)$, which we will then use to define the paired linear structure on both spaces in an obvious way.

\begin{lem}\label{lem_delta_cont_as_map_to_MX}
    Let $X$ be an $hk$-space. Then the map 
        $$ \delta_\bullet : X \to \mathcal{M}(X), \;\; x \mapsto \delta_x $$
    is continuous.
\end{lem}
\begin{proof}
    Since $\mathcal{M}(X)$ was defined as a subspace of $C_b(X)^\wedge$, it suffices to show that the map 
        $$ X \to C_b(X)^\wedge, \;\; x \mapsto \delta_\bullet $$
    is continuous. This map factors as the composite, 
    \[\begin{tikzcd}
	X & {C(X)^\wedge} && {C_b(X)^\wedge}
	\arrow["{\delta_\bullet}", from=1-1, to=1-2]
	\arrow["{(-)|_{C_b(X)}}", from=1-2, to=1-4]
    \end{tikzcd}\]
    so its continuity follows from cartesian closure of $\spaces$ and the fact that the inclusion map $C_b(X) \to C(X)$ is continuous. 
\end{proof}

\begin{prop}\label{dual_of_MX_is_Cb}
    Let $X$ be an $hk$-space. Then the linear maps,
        $$ \Phi: \mathcal{M}(X)^\wedge \to C_b(X), \;\;\; F \mapsto (x\mapsto F(\delta_x)),$$
    and,
        $$ \Psi: C_b(X) \to \mathcal{M}(X)^\wedge, \;\;\; f \mapsto (\mu \mapsto \mu(f)), $$
    are continuous and mutually inverse and hence constitute isomorphisms (natural in $X$):
        $$ \mathcal{M}(X)^\wedge \cong C_b(X).$$
\end{prop}
\begin{proof}
    Continuity of $\Psi$ is clear from cartesian closure of $\spaces$, as is continuity of $\Phi$ when taking \cref{lem_delta_cont_as_map_to_MX} into account. 
    Now, for any $f\in C_b(X)$, $x\in X$,
    \begin{align*}
        \Phi \circ \Psi (f)(x) = \Phi(\mu \mapsto \mu(f))(x)
        &= \delta_x(f) = f(x) = \id_{C_b(X)}(f)(x).
    \end{align*}
    Similarly, for all $F \in \mathcal{M}(X)^\wedge$, $x_0 \in X$,
    \begin{align*}
        \Psi \circ \Phi (F)(\delta_{x_0}) 
        &= \Psi(x\mapsto F(\delta_x))(\delta_{x_0})\\
        &= \delta_{x_0}(x \mapsto F(\delta_x)) \\
        &= F(\delta_{x_0}) = \id_{\mathcal{M}(X)^\wedge}(F)(\delta_{x_0}),
    \end{align*}
    which, since the span of all such $\delta_{x_0}$ is dense in $\mathcal{M}(X)$, implies the claim.
\end{proof}

\subsection{$C_b(X)$ and $\mathcal{M}(X)$ as paired linear $hk$-spaces}\label{sec_M_C_b_as_paired_lin_hk_sp}

Finally, we may now endow $C_b(X)$ and $\mathcal{M}(X)$ with (mutually dual) paired linear structures. 

\begin{defn}
    We define the paired linear $hk$-space $\mathcal{M}(X)$ as equipped with its natural dual, 
        $$ \mathcal{M}(X)^* := \mathcal{M}(X)^\wedge. $$
    Accordingly, we define the paired linear $hk$-space $C_b(X)$ as being equipped with the admissible dual 
        $$ C_b(X)^* := \mathcal{M}(X). $$
    We hence obtain two functors, 
    $$ C_b(X): \spaces^\op \to \plin, \;\;\; X \mapsto C_b(X), \; f \mapsto (- \circ f),  $$
    $$ \mathcal{M}(X): \spaces \to \plin, \;\;\; X \mapsto \mathcal{M}(X), \; f \mapsto f_*. $$
\end{defn}

\section{$k$-Regular Measures}\label{sec_k_reg_measures_section}

\subsection{Some measure-theoretic preliminaries}

Up to this point, we only know that the spaces $\mathcal{M}(X)$ and $\mathcal{M}_c(X)$ are spaces of \emph{functionals}, but we really want them to be spaces of genuine \emph{measures}. To finally bring measure theory into the picture, we remind ourselves of some notions from the theory of measures on topological spaces. We follow the terminology of \cite{bogachev2007measure}.  

\begin{warning}
    By default, ``measure'' will mean ``countably additive $\mathbb{K}$-valued measure of bounded variation''.
\end{warning}

\begin{defn}
    Let $X$ be a topological space. The \emph{Baire $\sigma$-algebra} is the $\sigma$-algebra generated by all continuous $\mathbb{K}$-valued functions on $X$ (it does not matter whether $\mathbb{K}=\mathbb{R}$ or $\mathbb{K}=\mathbb{C}$). A \emph{Baire measure} is a measure on the Baire $\sigma$-algebra. Similarly, recall that a \emph{Borel measure} is a measure defined on the \emph{Borel $\sigma$-algebra}, which, by definition, is generated by the open subsets of $X$. A Borel measure $\mu$ on $X$ is a \emph{Radon measure} if for every $\epsilon > 0$, there is a compact subset $K_\epsilon \subseteq X$ such that $|\mu|(X \mathbin{\backslash} K_\epsilon) < \epsilon$.  
\end{defn}

When $X$ is sufficiently well-behaved, the notions of Baire and Radon measure essentially coincide:

\begin{thm}\label{baire_measures_vs_radon_measures}
    On a Polish space, the Baire and Borel $\sigma$-algebras agree and every Baire (equivalently Borel) measure is a Radon measure. Moreover, on a $\sigma$-compact completely regular space, every Baire measure has a unique extension to a Radon measure.
\end{thm}
\begin{proof}
    For the statement about Polish spaces, see \cite[p. 70, Theorem 7.1.7]{bogachev2007measure}, for the one on compacta, see \cite[p. 81, Theorem 7.3.4]{bogachev2007measure}.
\end{proof}

Hence, if a space $X$ is either a Polish or compact, Baire measures on $X$ can be identified with Radon measures.

\subsection{A Riesz-Markov-type theorem} 

The proposition that follows is a first step towards our desired representation theorem for $\mathcal{M}(X)$ by representing functionals from the larger space $C_b(X)^\wedge$ as Baire measures. 

\begin{prop}\label{riesz_for_cb_general}
    Let $X$ be an $hk$-space. Then for every $\phi \in C_b(X)^\wedge$, there exists a unique Baire measure $\mu$ such that 
        $$ \phi(f) = \int f \diff \mu, $$
    for all $f\in C_b(X)$.
\end{prop}
\begin{proof}
    Let $\phi \in C_b(X)^\wedge$. By \cite[p.~111, Theorem 7.10.1]{bogachev2007measure}, the conclusion of the lemma follows if, and only, if for every monotonically decreasing sequence $(f_n)$ in $C_b(X)$ converging pointwise to zero, $\phi(f_n)\to 0$, as well. So let $(f_n)$ be such a sequence. By Dini's theorem, $(f_n)$ converges to $0$ uniformly on compact subsets of $X$. Moreover, since it is monotonically decreasing, $(f_n)$ is uniformly bounded by the constant $\sup_{x\in X} f_0(x)$. Hence, by \cref{convergent_sequences_in_Cb}, $(f_n)$ converges to $0$ in $C_b(X)$ and by continuity of $\phi$, $\phi(f_n) \to 0$, which is what we wanted to show.
\end{proof}

\subsection{$k$-Pre-regular measures}

The class of measures which represent functionals from $\mathcal{M}(X)$ will be referred to as \emph{$k$-regular measures} (this will be \cref{defn_k_reg_measures}). As a stepping stone towards this definition and towards understanding this class of measures, we first introduce \emph{$k$-pre-regular measures}.

\begin{defn}
    Let $X$ be an $hk$-space. A \emph{$k$-pre-regular measure} is a Baire measure $\mu$ on $X$ such that the map 
        $$ C_b(X) \to \mathbb{K}, \;\; f \mapsto \int f \mathrm{d} \mu $$
    is continuous.
    Denote the set of $k$-pre-regular measures on $X$ by $\mathcal{M}_{\mathrm{pre}}(X)$. 
\end{defn}

As an immediate corollary to \cref{riesz_for_cb_general}, we obtain:

\begin{cor}\label{cor_riesz_for_k_pre_reg}
    Let $X$ be an $hk$-space. Then the mapping, 
        $$ \mathcal{M}_{\mathrm{pre}}(X) \to C_b(X)^\wedge, \;\;\; \mu \mapsto  \int (-) \diff \mu, $$
    is a bijection, so we may identify functionals in $C_b(X)^\wedge$ with $k$-pre-regular measures under this bijection.
\end{cor}

\begin{defn}
    Let $X$ be an $hk$-space. Denote the closed subspace of $\mathcal{M}_{\mathrm{pre}}(X)=C(X)^\wedge$ consisting of $k$-pre-regular \emph{probability} measures (equivalently, positive normalised functionals) by $\mathcal{P}_{\mathrm{pre}}(X)\subseteq \mathcal{M}_{\mathrm{pre}}(X)$.
\end{defn}

In many cases, \emph{any} Baire measure will be $k$-pre-regular: 

\begin{prop}\label{radon_measures_k_regular}
    Let $X$ be an $hk$-space and let $\mu$ be a Baire measure $X$ that is the restriction of some Radon measure $\tilde{\mu}$ on $X$ to the Baire $\sigma$-algebra. Then $\mu$ is a $k$-pre-regular measure. In particular, if $X$ either a completely regular $\sigma$-compact $hk$-space (e.g.~a compact Hausdorff space, or a Brauner space), or a Polish space, then \emph{every} Baire measure is a $k$-pre-regular measure.
\end{prop}
\begin{proof}
     Let $(f_i)$ be a net of continuous functions on $X$ uniformly bounded by $C>0$ converging uniformly on compact subsets to $0$. By \cref{continuity_on_Cb}, it suffices to show that 
        $$ \int f_i \diff \mu \to 0. $$
    Let $\epsilon >0$. Then, since $\tilde{\mu}$ is a Radon measure, there exists a compact subset $K$ such that $|\tilde{\mu}|(X\mathbin{\backslash} K)<\epsilon$. Now, by the compact convergence of $(f_i)$, there is an index $i_0$ such that for all $i\geq i_0$, 
        $$ \sup_{x\in K} |f_i(x)| < \epsilon. $$
    Now, 
    \begin{align*}
        \Big| \int f_i \diff \mu \Big| \leq \int |f_i| \diff |\tilde{\mu}| &= 
        \int_K |f_i| \diff |\tilde{\mu}| + \int_{X\mathbin{\backslash} K} |f_i| \diff |\tilde{\mu}| \\
        &\leq \sup_{x\in K} |f_i(x)| |\tilde{\mu}|(K) + C |\tilde{\mu}|(X \mathbin{\backslash} K) \\
        &< (|\mu|(X) + C) \,\epsilon.
    \end{align*}
    This shows convergence and, in conclusion, that $\mu$ is $k$-pre-regular. Now, if $X$ is completely regular $\sigma$-compact $hk$-space or a Polish space, then every Baire measure admits a unique extension to a Radon measure (see \cref{baire_measures_vs_radon_measures}), so that indeed, every Baire measure is $k$-pre-regular.
\end{proof}

\subsection{$k$-Regular measures}\label{sec_k_reg_measures}

The space of $k$-pre-regular measures is still too large -- it corresponds to $C_b(X)^\wedge$ instead of the possibly smaller paired linear $hk$-space dual $C_b(X)^*=\mathcal{M}(X)$. As a first step, the notion of \emph{$k$-regular measure} solves this problem trivially by \emph{requiring} integration against a $k$-regular measure to define a functional in $\mathcal{M}(X)$ (so that the question becomes finding classes of examples of $k$-regular measures that include the practically relevant ones).

\begin{defn}\label{defn_k_reg_measures}
    Let $X$ be an $hk$-space. A \emph{$k$-regular measure} on $X$ is a $k$-pre-regular measure $\mu$ such that, 
        $$ \int (-) \diff \mu \in C_b(X)^* (= \mathcal{M}(X)). $$
    As a consequence of this definition and \cref{cor_riesz_for_k_pre_reg}, we may identify $k$-regular measures with the elements of $\mathcal{M}(X)$ and will therefore use the same notation $\mathcal{M}(X)$ also for the set of $k$-regular measures on $X$. Similarly, we will write 
        $$ \mathcal{P}(X) :=  \{ \mu \in \mathcal{M}(X) \mid \mu \geq 0, \: \mu(1) = 1\} $$ 
    for the closed subspace of $\mathcal{M}(X)$ consisting of positive, normalised functionals. In other words, $\mathcal{P}(X)$ is, under the identification of measures and functionals, the space of $k$-regular probability measures on $X$.
\end{defn}

The pressing question is now, of course, whether the class of $k$-regular measures is wide enough to encompass those examples which appear in applications. As we will see, arbitrary Baire probability measures on compact Hausdorff spaces are $k$-regular, and the same result holds for Polish spaces. This will require examining the question of convergence in $\mathcal{M}(X)$ and how it relates to the more familiar notion of \emph{weak convergence} of measures. 

\subsection{Convergence in spaces of $k$-regular measures}

\subsubsection{Weak convergence of measures} The following should be a familiar definition:

\begin{defn}\label{defn_weak_convergence_of_measures}
    Recall that a net $(\mu_i)$ of measures on a topological space $X$ is said to \emph{converge weakly} to $\mu$ if for all $f\in C_b(X)$, 
        $$ \int f \diff \mu_i \to \int f \diff \mu. $$
\end{defn}

When $(\mu_i)$ is the sequence of distributions $(P_{x_i})$ of some random variables $(x_i)$, weak convergence of $(P_{x_i})$ is exactly  \emph{convergence in distribution} of the $(x_i)$ in the sense of probability theory. \par 
The simplest case in which convergence of \emph{sequences} in $\mathcal{M}(X)$ (and of general nets in $\mathcal{P}(X)$) coincides with weak convergence is when $X$ is compact. In this case, the notion of $k$-regular measure reduces to that of Baire measure.

\begin{prop}\label{convergence_of_measures_on_compactum}
    Let $X$ be a compact Hausdorff space. Then every Baire measure on $X$ is $k$-regular (so that $\mathcal{M}(X)$ consists exactly of the Baire measures on $X$). Moreover, the topology of $\mathcal{P}(X)$ coincides with the topology of weak convergence of measures, and $\mathcal{P}(X)$ is a compact Hausdorff space. In particular, a net of measures converges in $\mathcal{M}(X)$ if, and only if, it converges weakly.
\end{prop}
\begin{proof}
    When $X$ is compact, $C_b(X)=C(X)$ is the \emph{Banach} space of continuous functions on $X$ and using Smith duality (\cref{Smith_duality}) together with \cref{dual_of_MX_is_Cb}, we see that
        $$ \mathcal{M}(X)=C(X)^\wedge. $$ 
    The left hand side is the space of $k$-regular measures, while the right hand side can be identified with the space of Baire (equivalently, Radon, by \cref{baire_measures_vs_radon_measures}) measures on $X$ (by the Riesz representation theorem, see \cite[Theorem 7.10.4]{bogachev2007measure}), so every Baire measure on $X$ is indeed $k$-regular. As a closed subspace of the compact unit ball of the space $\mathcal{M}(X)=C(X)^\wedge$, $\mathcal{P}(X)$ is compact and its topology coincides with the weak-$*$ topology, i.e.~the topology of weak convergence of measures (see \cref{prop_dual_fre_carries_co_topology}).
\end{proof}

\subsubsection{Uniformly tight families and Prokhorov spaces} We will subsequently need the following measure-theoretic concepts.

\begin{defn}
    Recall that a family of Radon measures $(\mu_i)$ on an $hk$-space $X$ is \emph{uniformly tight} if for every $\epsilon > 0$ there is a compact set $K$ such that $|\mu|(X\mathbin{\backslash} K) < \epsilon$. $X$ is a \emph{strongly sequentially Prokhorov space} if every weakly convergent sequence of Radon measures on $X$ is uniformly tight.
\end{defn}

\subsubsection{When does convergence in $\mathcal{P}(X)$ reduce to weak convergence?}

\begin{thm}\label{convergence_in_px_vs_weak_convergence}
    Let $X$ be a either a Polish space or a completely regular hemicompact $hk$-space. Then a sequence $(\mu_n)$ of probability measures converges in $\mathcal{P}(X)$ if, and only if, it converges weakly.  
\end{thm}
\begin{proof}
    Only the ``if'' direction is non-trivial, so suppose that $\mu_n \to \mu$ weakly. We want to show that $\mu_n \to \mu$ uniformly on compact subsets of $C_b(X)$. So let $L\subseteq C_b(X)$ be compact and $\epsilon >0$. The hypotheses imply that every Baire measure $\nu$ on $X$ extends uniquely to a Radon measure $\tilde{\nu}$ (see \cref{baire_measures_vs_radon_measures}) and that $X$ is strongly sequentially Prokhorov (see \cite[p. 220, Proposition 8.20.12]{bogachev2007measure} for the case of hemicompact $hk$-spaces, and \cite[p. 206, Theorem 8.6.8]{bogachev2007measure} for the case of Polish spaces). Therefore, the family $\{|\tilde{\mu}_n-\tilde{\mu}|\}$ is uniformly tight and there exists a compact subset $K_\epsilon\subseteq X$ such that for all $n\in \mathbb{N}$, $|\tilde{\mu}_n-\tilde{\mu}|(X\mathbin{\backslash} K_\epsilon) < \epsilon $. Since the family of functions $L$ is compact, it is uniformly bounded by some constant $C>0$. Moreover, given the hypotheses, the restrictions $\tilde{\mu}_n|_{K_\epsilon}, \tilde{\mu}|_{K_\epsilon}$ are Radon measures on the compact Hausdorff space $K_\epsilon$, and $\tilde{\mu}_n|_{K_\epsilon} \to \tilde{\mu}|_{K_\epsilon}$ weakly and hence uniformly on compact subsets of $C(K_\epsilon)$ (see \cref{convergence_of_measures_on_compactum}). The image of $L$ under restriction of functions to $K_\epsilon$,
        $$ L|_{K_\epsilon} := \{f|_{K_\epsilon} \mid f \in L \},$$
    is compact, since it is the image of the compact set $L$ under continuous map $f \mapsto f \circ \iota_{K_\epsilon}$, where $\iota_{K_\epsilon}: K_{\epsilon} \to X $ is the inclusion map. By compact convergence of $\tilde{\mu}_n|_{K_\epsilon}$, we have that for sufficiently large $n$, 
        $$ \sup_{f \in L|_{K_\epsilon}} \Big| \int f \diff \tilde{\mu}_n|_{K_\epsilon} - \int f \diff \tilde{\mu}|_{K_\epsilon} \Big| < \epsilon. $$
    Putting all of these observations together, we obtain that for all $f\in L$ and sufficiently large $n$, 
    \begin{align*}
        \Big| \int f \diff \mu_n - \int f \diff \mu \Big| 
        &\leq \int_{X\mathbin{\backslash} K_\epsilon} |f| \diff |\tilde{\mu}_n - \tilde{\mu}| + \Big| \int_{K_\epsilon} f \diff \tilde{\mu}_n - \int_{K_\epsilon} f \diff \tilde{\mu}_n \Big| \\
        &\leq  C \epsilon + \Big| \int_{K_\epsilon} f \diff \tilde{\mu}_n|_{K_\epsilon} - \int_{K_\epsilon} f \diff \tilde{\mu}|_{K_\epsilon} \Big|\\
        &\leq (C+1) \,\epsilon,  
    \end{align*}
    which is what we wanted to show.
\end{proof}

\begin{thm}\label{cor_prob_measures_on_sep_lin_spaces}
    Let $V$ be a separable Fréchet space. Then the Borel and Baire $\sigma$-algebras on the natural dual $V^\wedge$ coincide and every Baire measure on $V^\wedge$ is a Radon measure and $k$-regular. Hence, $\mathcal{P}(V^\wedge)$ consists exactly of the Radon probability measures on $V^\wedge$. Moreover, a sequence of probability measures in $\mathcal{P}(V^\wedge)$ converges if, and only if, it converges weakly. Finally, $\mathcal{P}(V^\wedge)$ is a QCB space (in particular, a sequential space) and consequently, convergence of sequences of probability measures completely characterises the topology of $\mathcal{P}(V^\wedge)$.
\end{thm}
\begin{proof}
    The natural dual $V^\wedge$ of $V$ is a separable Brauner space and therefore in particular a hemicompact QCB space.
    Since $V^\wedge$ is a topological vector space, it is also completely regular and by \cref{baire_measures_vs_radon_measures}, every Baire measure has a unique extension to a Radon measure. Furthermore, since $V^\wedge$ is a QCB space, it is hereditarily Lindelöf (see \cref{QCB_separable_lindelof}), so by \cite[p.~13, Corollary 6.3.5]{bogachev2007measure}, the Baire and Borel sigma algebras coincide. Therefore, every Baire measure on $V^\wedge$ is a Radon measure. \par 
    By \cref{radon_measures_k_regular}, every Baire measure on $V^\wedge$ is a $k$-preregular measure. To show that, moreover, every Baire (equivalently, Radon) measure on $V^\wedge$ is $k$-regular, it suffices to show that for every Radon measure $\mu$ on $V^\wedge$, there is a sequence of finitely supported measures converging to $\mu$ in $\mathcal{P}_{\mathrm{pre}}(V^\wedge)=C_b(V^\wedge)^\wedge$. Since, by \cref{convergence_in_px_vs_weak_convergence}, convergence in $\mathcal{P}(V^\wedge)$ is equivalent to weak convergence and $\mathcal{P}(V^\wedge) \subseteq \mathcal{P}_{\mathrm{pre}}(V^\wedge)$ is closed, it even suffices to show that for every Radon measure $\mu$ on $V^\wedge$ there is a sequence of finitely supported measures converging weakly to $\mu$. So let $\mu$ be a Radon measure on $V^\wedge$. \par  
    Let $K_n\subseteq V^\wedge$ be an increasing sequence of compact subsets such that 
        $$ V^\wedge = \bigcup_{n\in\mathbb{N}} K_n. $$
    For each $n\in\mathbb{N}$, the restriction $\mu|_{K_n}$ is a Radon measure on $K_n$. By \cref{convergence_of_measures_on_compactum}, there exists a finitely supported measure $\nu_n$ on $K_n$ such that for all $f\in C_b(V^\wedge)$, 
        $$ \left| \int_{K_n} f\: \diff (\nu_n - \mu) \right| < \frac{1}{n}. $$
    We claim that the sequence $(\nu_n)$ converges weakly to $\mu$. Let $\epsilon>0$ and choose $N\in \mathbb{N}$ such that $|\mu|(V^\wedge \mathbin{\backslash} K_n) < \epsilon$ and $1/N < \epsilon$. Then for all $n\geq N$,
    \begin{align*}
        \left| \int_{V^\wedge} f\: \diff (\nu_n - \mu) \right| &\leq 
        \left| \int_{V^\wedge \mathbin{\backslash} K_n} f\: \diff (\nu_n - \mu) \right| + \left| \int_{K_n} f\: \diff (\nu_n - \mu) \right| \\
        &< \|f\|_{\sup}\, \epsilon + \epsilon,
    \end{align*}
    proving the claim that $\nu_n$ converges weakly to $\mu$, which is hence a $k$-regular measure. \par 
    Finally, the fact that $\mathcal{P}(V^\wedge)$ is a QCB space follows from the fact that QCB spaces are closed under the formation of countable (co-)limits and exponentials (i.e.~spaces of continuous maps). 
\end{proof}

\subsection{The Bochner-Minlos-Lévy-Fernique homeomorphism} Let us point out that, in conjunction with cartesian closure and sequentiality of QCB spaces, \cref{cor_prob_measures_on_sep_lin_spaces} allows for a very concise simultaneous formulation of the Bochner-Minlos and Lévy-Fernique theorems: if $V$ is a separable nuclear Fréchet space, then the Fourier transform is a homeomorphism between $\mathcal{P}(V^\wedge)$ and the space of continuous normalised positive-definite functions on $V$. Let us make this more precise.

\begin{defn}
    Let $V$ be a linear $hk$-space. Following \cite[Definition 7.13.3]{bogachev2007measure}, a continuous function  $f:V\to \mathbb{C}$ is \emph{positive-definite} if for all $n\in\mathbb{N}$, $c_1,\dots, c_n\in \mathbb{C}$, $x_1, \dots, x_n\in V$,
        $$ \sum_{i,j=1}^n c_i \overline{c_j} f(x_i-x_j) \geq 0. $$
    We say that $f$ is \emph{normalised} if $f(0)=1$.
    Finally, we denote the set of continuous normalised positive-definite functions on $V$ by $C_{1}^{+}(V)$ and endow $C_{1}^{+}(V)$ with the subspace topology that it inherits from $C(V)$ as a closed subspace. 
\end{defn}

\begin{defn}
    Let $V$ be a linear $hk$-space and let $\mu\in \mathcal{P}(V^\wedge)$ be a probability measure. The \emph{characteristic function} of $\mu$ is the function $\widehat{\mu}: V \to \mathbb{C}$ given by,
        $$ \widehat{\mu}(x) := \int_{V^\wedge} e^{2\pi i \phi(x)} \diff \mu(\phi). \qquad (x\in V)$$
    The map $\mu \mapsto \widehat{\mu}$ is called the \emph{Fourier transform}.
\end{defn}

With these definitions in place, we may formulate: 

\begin{thm}\label{thm_bochner_minlos_levy_fernique}
    Let $V$ be a separable nuclear Fréchet space (for example, $V=\mathcal{S}(\mathbb{R}^n)$). Then the Fourier transform defines a \emph{homeomorphism},
        $$\mathcal{F}: \mathcal{P}(V^\wedge) \to C_{1}^{+}(V), \;\; \mu \mapsto \int_{V^\wedge} e^{2\pi i \phi(x)} \diff \mu(\phi). $$
\end{thm}
\begin{proof}
    The statement that $\mathcal{F}$ is a well-defined bijection is (a special case of) the Bochner-Minlos theorem (see \cite[Théorème 1]{meyer18theoreme}), taking into account that by \cref{cor_prob_measures_on_sep_lin_spaces}, $\mathcal{P}(V^\wedge)$ consists exactly of the Radon measures on $V^\wedge$. That $\mathcal{F}$ is continuous follows from the fact that integration of continuous bounded functions against $k$-regular measure is continuous (by definition) and the cartesian closure of $\spaces$. Finally, as a closed subspace of a QCB space $C(V)$, $C_{1}^{+}(V)$ is a QCB space (see \cref{closed_subspaces_QCB}). In particular, $C_{1}^{+}(V)$ is a sequential space. Therefore, the continuity of the inverse Fourier transform $\mathcal{F}^{-1}$ is equivalent to its sequential continuity, and this follows from the Lévy-Fernique continuity theorem (\cite[Théorème 2]{meyer18theoreme}). 
\end{proof}

\begin{remark}
    \cite[Théorème 2]{meyer18theoreme} also entails that a sequence converges in $C_{1}^{+}(V)$ if, and only if, it converges pointwise to a function that is continuous at the origin.
\end{remark}

\subsection{$k$-Regular measures on Polish spaces} Our final observation concerning $k$-regular measures is that on Polish spaces, their theory completely reduces to the classical one.

\begin{prop}\label{radon_measures_Polish_k_regular}
    Let $X$ be a Polish space. Then following \emph{every} Baire measure is a $k$-regular measure and hence, the notions of Baire, Radon and $k$-regular measure all coincide (see \cref{baire_measures_vs_radon_measures}). Moreover, the topology on $\mathcal{P}(X)$ is given by the topology of weak convergence of measures and $\mathcal{P}(X)$ is again a Polish space.
\end{prop}
\begin{proof}
    Let $\mu$ be a Baire (equivalently, Radon, see \cref{baire_measures_vs_radon_measures}) probability measure on $X$. By \cref{radon_measures_k_regular}, $\mu$ is $k$-pre-regular. Now, let $\mathcal{P}_w(X)$ be the space of Radon measures with the topology of weak convergence, and recall that $\mathcal{P}_{\mathrm{pre}}(X)\subseteq \mathcal{M}_{\mathrm{pre}}(X)$ is the space of $k$-pre-regular probability measures with the subspace topology. As sets, $\mathcal{P}_w(X)=\mathcal{P}_{\mathrm{pre}}$ (again, by \cref{radon_measures_k_regular}). We want to show that the respective topologies coincide, as well. First, note that the identity map $\mathcal{P}_{\mathrm{pre}}(X) \to \mathcal{P}_w(X)$ is continuous, since a convergent net in $\mathcal{P}(X)$ is also weakly convergent. By \cite[p.~213, Theorem 8.9.4]{bogachev2007measure}, our assumption that $X$ is a Polish space implies that $\mathcal{P}_w(X)$ is a Polish space, too. Hence, sequential continuity of the identity map $\mathcal{P}_w(X) \to \mathcal{P}_{\mathrm{pre}}(X)$ (which is provided by \cref{convergence_in_px_vs_weak_convergence}) implies continuity, so the two topologies coincide and $\mathcal{P}_{\mathrm{pre}}(X)=\mathcal{P}_w(X)$ also as topological spaces. Finally, finitely supported measures are dense in $\mathcal{P}_w(X)$ (see \cite[p.~213, Theorem 8.9.4 (ii)]{bogachev2007measure}) and therefore 
        $$ \mathcal{P}(X)=\mathcal{P}_{\mathrm{pre}}(X)=\mathcal{P}_w(X), $$
    which completes the proof.
\end{proof}

The above guarantees that most probability measures that occur in applications are $k$-regular. (For a concrete example, take the law of a standard Brownian motion, understood as Baire probability measure on the separable Fréchet space $C([0,\infty), \mathbb{R}^n)$.)

\section{$\mathcal{M}$ as a commutative monad on $\spaces$}\label{sec_M_comm_monad}

\subsection{The monad structure of $\mathcal{M}$} 

We now describe the monad structure of $\mathcal{M}$. Unit and multiplication are given as follows.

\begin{defn}\label{defn_unit_mult_MX}
    Define the (families of) map(s),
        $$ \delta_\bullet : X \mapsto \mathcal{M}(X), \;\; x \mapsto \delta_x, \;\;\;\;\; (X \in\spaces) $$
    and 
        $$ \smallint{}: \mathcal{M}(\mathcal{M}(X)) \to \mathcal{M}(X), \smallint{}(\pi)(f) := \int_{\mathcal{M(X)}}\mu(f)\diff \pi(\mu). \;\;\;\;\; (X \in\spaces) $$
\end{defn}

Note that under the identification of measures with functionals, we have that 
\begin{equation}\label{eqn_smallint_as_higher_order_function}
    \smallint{} (\pi)(f) = \pi( \mu \mapsto \mu(f) ). \qquad (\pi \in \mathcal{M}(\mathcal{M}(X)), f \in C_b(X))
\end{equation}

\begin{lem}
    The families of maps $\delta_\bullet$ and $\smallint{}$ constitute well-defined natural transformations.
\end{lem}
\begin{proof}\text{}
    \begin{enumerate}
        \item \emph{The map $\delta_\bullet$ is well-defined and continuous:} This was \cref{lem_delta_cont_as_map_to_MX}. 
        \item \emph{Naturality of $\delta_\bullet$:} We need to show that for all $hk$-spaces $X,Y$ and every continuous map $f:X\to Y$ the following diagram commutes:
\[\begin{tikzcd}
	X & Y \\
	{\mathcal{M}(X)} & {\mathcal{M}(Y)}
	\arrow["f", from=1-1, to=1-2]
	\arrow["{\delta_\bullet}"', from=1-1, to=2-1]
	\arrow["{\delta_\bullet}", from=1-2, to=2-2]
	\arrow["{f_*}"', from=2-1, to=2-2]
\end{tikzcd}\]
        But this just means that for all $x\in X$,
            $$ f_*\delta_x = \delta_{f(x)}, $$
        which holds.
        \item \emph{The map $\smallint{}$ is well-defined and continuous:} First notice that $\smallint{}$ is well-defined, since for every $f\in C_b(X)$, the map $\mu \mapsto \mu(f)$ is continuous and bounded (by $\|f\|_{\sup}\cdot |\mu|(X)$) on $\mathcal{M}(X)$. Moreover, being given by a certain evaluation-type map (see \cref{eqn_smallint_as_higher_order_function}), $\smallint{}$ is continuous by cartesian closure of $\spaces$.
        \item \emph{The map $\smallint{}$ is natural:} We need to show that the following diagram commutes:
        \[\begin{tikzcd}
	       {\mathcal{M}(\mathcal{M}(X))} & {\mathcal{M}(\mathcal{M}(Y))} \\
	       {\mathcal{M}(X)} & {\mathcal{M}(Y)}
	       \arrow["{(f_*)_*}", from=1-1, to=1-2]
	       \arrow["{\smallint{}}"', from=1-1, to=2-1]
	       \arrow["{\smallint{}}", from=1-2, to=2-2]
	       \arrow["{f_*}"', from=2-1, to=2-2]
        \end{tikzcd}\]
        To prove this, we calculate, using \cref{eqn_smallint_as_higher_order_function}, that for all $\pi\in \mathcal{M}(\mathcal{M}(X))$ and all $h\in C_b(X)$,
        \begin{align*}
            f_*\left(\smallint{}(\pi)\right)(h) &= \smallint{}(\pi)(h \circ f)\\
                                     &= \pi(\mu \mapsto \mu(h \circ f)) \\
                                     &= \pi(\mu \mapsto f_*\mu(h))\\
                                     &= \pi(f_* \circ(\mu \mapsto \mu(h)))\\
                                     &= (f_*)_*\pi(\mu \mapsto \mu(h))\\
                                     &= \smallint{}((f_*)_*\pi)(h).
        \end{align*}
    \end{enumerate}
\end{proof}

\begin{thm}
    $(\mathcal{M}, \delta_\bullet, \smallint{})$ is a monad on $\spaces$.
\end{thm}
\begin{proof}
    We need to verify the monad laws, which amounts to the following three calculations, where $X$ is an $hk$-space, $\mu \in \mathcal{M}(X)$, $f\in C_b(X)$ and $\pi \in \mathcal{M}(\mathcal{M}(X))$.
    \begin{enumerate}
    \item First unit law:
        $$\smallint{}(\delta_\mu)(f) = \int \nu(f) \diff \delta_\mu(\nu) = \mu(f). $$ 
    \item Second unit law:
    $$(\smallint{} (\delta_\bullet)_*\mu)(f) = \int \delta_x(f) \diff \mu = \int f \diff \mu = \mu(f)$$
    \item Associativity law: We need to show that for all $\Pi\in \mathcal{M}(\mathcal{M}(\mathcal{M}(X)))$ and all $h \in C_b(X)$,
        $$\smallint{} \left(\smallint{} (\Pi)\right)(h)= \smallint{} \left(\smallint{}\right)_*\Pi. $$
     To do so, we repeatedly use \cref{eqn_smallint_as_higher_order_function} to obtain,
    \begin{align*}
        \smallint{} \left(\smallint{} (\Pi)\right)(h) 
        &= \left(\smallint{}(\Pi)\right)( \mu \mapsto \mu(h) ) \\
        &= \Pi(\nu \mapsto \nu(\mu \mapsto \mu(h))) \\
        &= \Pi(\nu \mapsto \left(\smallint{} (\nu)\right)(h)) \\
        &= \Pi\left([\mu \mapsto \mu(h)]\circ \smallint{}\right) \\
        &= \left(\smallint{}\right)_*\Pi (\mu \mapsto \mu(h)) \\
        &= \smallint{} \left(\smallint{}\right)_*\Pi,
    \end{align*}
    completing the proof.
    \end{enumerate}
\end{proof}

\subsection{Product measures (the commutative monad structure of $\mathcal{M}$)}

\begin{defn}\label{defn_product_measures}
    Let $X,Y$ be $hk$-spaces and let $\mu \in \mathcal{M}(X)$, $\nu \in \mathcal{M}(Y)$ be $k$-regular measures. Define the product measure $\mu \otimes \nu \in \mathcal{M}(X \times Y)$ on $X\times Y$ as (the $k$-regular measure representing the functional), 
        $$ (\mu \otimes \nu)(f) := \int \int f(x,y) \diff \mu(x) \diff \nu(y). \;\;\;\;\; (f \in C_b(X\times Y))$$
\end{defn}

The family of maps  
    $$ \otimes: \mathcal{M}(X) \times \mathcal{M}(Y) \to \mathcal{M}(X\times Y), \;\; (\mu, \nu) \mapsto \mu\otimes \nu, \;\;\;\; ((X,Y)\in \spaces \times \spaces)$$
is now a natural transformation (this follows from a change of variables in the definition) and we have that:

\begin{thm}\label{M_bbd_meas_monad_commutative}
    $(\mathcal{M}, \delta_\bullet, \smallint{}, \otimes)$ is a commutative monad on $\spaces$.
\end{thm}
\begin{proof}
    We need to show that the diagram 
    \[\begin{tikzcd}
	{\mathcal{M}(X)\times\mathcal{M}(Y)} && {\mathcal{M}(X\times Y)} \\
	{\mathcal{M}(Y)\times\mathcal{M}(X)} && {\mathcal{M}(Y\times X)}
	\arrow["\otimes", from=1-1, to=1-3]
	\arrow["\sigma"', from=1-1, to=2-1]
	\arrow["\otimes"', from=2-1, to=2-3]
	\arrow["{\sigma_*}", from=1-3, to=2-3]
    \end{tikzcd}\]
    commutes (where $\sigma$ is the braiding), which amounts to showing that for all $f\in C_b(X)$,
        $$ \int f \diff \mu \otimes \nu  = \int f \diff \sigma_*(\nu \otimes \mu). $$
    After a change of variables and expanding, this becomes,
        \begin{equation}\label{fubini_for_Cb}
            \int \int f(x,y) \diff \mu(x) \diff \nu(y)  = \int \int f(x,y) \diff \nu(y) \diff \mu(x).
        \end{equation}
    We might now be tempted to invoke a version of Fubini's theorem for signed measures, but let us take the oppurtunity to give an independent proof which will then instead \emph{entail} this version of Fubini's theorem. \par 
    Note that both sides of \eqref{fubini_for_Cb} are continuous in $(\mu, \nu)$ and that equality holds when $\mu$ and $\nu$ are finitely supported (since then, \eqref{fubini_for_Cb} reduces to an interchange of finite sums). The claim now follows from density of finitely supported measures in $\mathcal{M}(X\times Y)$.
\end{proof}

\section{The Probability Monad $\mathcal{P}$ on the Category of $hk$-Spaces}\label{sec_prob_mon_P}

Because the pushforward of a probability measure remains a probability measure, $\mathcal{P}(-)$ inherits the structure of a functor from $\mathcal{M}$. Similarly, we can equip $\mathcal{P}$ with a unit and multiplication coming from the monad structure of $\mathcal{M}$.

\begin{defn}
    Define the family of continuous maps,
        $$ \delta_\bullet : X \mapsto \mathcal{P}(X), \;\; x \mapsto \delta_x, \;\;\;\;\; (X \in\spaces) $$
    as well as, 
        $$ \smallint{}: \mathcal{P}(\mathcal{P}(X)) \to \mathcal{P}(X), \smallint{}(\pi)(f) := \int_{\mathcal{P}(X)}\mu(f)\diff \pi(\mu). \;\;\;\;\; (X \in\spaces) $$
\end{defn}

Note that this definition is exactly the same as \cref{defn_unit_mult_MX} if we replace $\mathcal{P}$ with $\mathcal{M}$. Moreover, when $\mu$ and $\nu$ probability measures then so is the product measure $\mu \otimes \nu$ from \cref{defn_product_measures} and we obtain a family of maps  
    $$ \otimes: \mathcal{P}(X) \times \mathcal{P}(Y) \to \mathcal{P}(X\times Y), \;\; (\mu, \nu) \mapsto \mu\otimes \nu, \;\;\;\; ((X,Y)\in \spaces \times \spaces)$$
We can therefore, the proof being exactly as before, conclude:

\begin{theorem}
    The families $\delta_\bullet$, $\smallint{}$ and $\otimes$ constitute natural transformations and $(\mathcal{P}, \delta_\bullet, \smallint{}, \otimes)$ is a commutative monad on $\spaces$. 
\end{theorem}

\begin{remark}
    Now that we know that $\mathcal{P}$ is a monad, we may rephrase the results of  \cref{convergence_of_measures_on_compactum} and \cref{radon_measures_Polish_k_regular} as saying that $\mathcal{P}$ restricts to the Giry monad on Polish spaces, and to the Radon monad on compact Hausdorff spaces, recovering (and in some sense ``interpolating'' between) two well-known probability monads. Moreover, since the category $h\mathsf{QCB}$ of $k$-Hausdorff QCB spaces is closed under the formation of limits and exponentials in $\spaces$, $\mathcal{P}$ also restricts to a probability monad on $h\mathsf{QCB}$. 
\end{remark}

\section{The Free Paired Linear $hk$-Space $\mathcal{M}_c(X)$ as a Space of Compactly Supported Measures}\label{sec_free_repl_lin_is_spc_of_measures}

In Chapter 3, we defined $\mathcal{M}_c(X)$ to be the free paired linear $hk$-space (equivalently, the free replete linear $hk$-space) on $X$ (see \cref{defn_free_plin_ksp}). We now justify this notation by showing that $\mathcal{M}_c(X)$ can be identified with with the space of compactly supported $k$-regular measures on $X$. In contrast to $\mathcal{M}(X)$, of which we needed to show ``by hand'' that it defines a commutative monad, this is automatic for $\mathcal{M}_c(X)$ from the general considerations of Chapter 3 (see \cref{comm_mon_str_of_Mc}). \par 
Let us start by observing that we may regard $\mathcal{M}_c(X)$ as a subset of $\mathcal{M}(X)$.

\begin{lem}\label{lem_istar_injective}
    The map 
        $$ i^*: \mathcal{M}_c(X) \to \mathcal{M}(X) $$
    obtained as the adjoint of the inclusion map $i: C_b(X) \to C(X)$ is injective.
\end{lem}
\begin{proof}
    Let $\mu, \nu \in \mathcal{M}_c(X)$ and suppose that $i^*(\mu)=i^*(\nu)$. Then, for all $f\in C_b(X)$, $\nu(f)=\mu(f)$. Since $C_b(X)$ is dense in $C(X)$, this implies that $\mu=\nu$.
\end{proof}

\begin{defn}
    Let $X$ be an $hk$-space. By a compactly supported $k$-regular measure on $X$ we mean a $k$-regular measure $\mu\in \widetilde{\mathcal{M}}(X)$ which is the pushforward of a Baire measure defined on a compact subspace of $X$ under the inclusion. (Of course, by \cref{convergence_of_measures_on_compactum}, $k$-regularity of $\mu$ is automatic in this case.) We denote the set of compactly supported $k$-regular measures by $\widetilde{\mathcal{M}}_c(X)$.
\end{defn}

\begin{remark}\label{rem_csupp_k_reg}
    Why do we not just define compactly supported $k$-regular measures as ``$k$-regular measures whose support is compact''? This is for a technical reason: Remember that on a general topological space, the Baire $\sigma$-algebra may be a proper subset of the Borel $\sigma$-algebra, which is why the usual definition of the support of a \emph{Radon} (or, more generally, \emph{$\tau$-additive}) measure cannot be applied here. \par 
    However, if $X$ is a \emph{Hausdorff} $hk$-space, then every tight (in particular, every compactly supported $k$-regular) Baire measure on $X$ admits a unique extension to a Radon measure \cite[Theorem 7.3.10]{bogachev2007measure}. Hence, for all practical purposes, the set of compactly supported $k$-regular measures can be identified with the set of compactly supported Radon measures.
\end{remark}

Before we can provide an identification between compactly supported measures in $\widetilde{\mathcal{M}}_c(X)$ and functionals in $\mathcal{M}_c(X)$, as promised, we need a lemma. 

\begin{lem}\label{functionals_on_cco}
    Let $X$ be an $hk$-space and let $\phi$ be a continuous linear functional on the space $\cco(X)$ of continuous functions \emph{with the compact-open topology}. Then there exists a compactly supported $k$-regular measure $\mu$ on $X$ such that for all $f\in C(X)$, 
        $$ \phi(f) = \int_X f(x) \diff \mu(x).$$
\end{lem}
\begin{proof}
    Put 
        $$ P := \prod_{K\subseteq X \text{ compact}}\!\!\!\!\!\!\!\!\!\!\!{}^{\topsp} \;\;C(K), $$
    where $\prod^{\topsp}$ denotes the product in $\topsp$ (equivalently, in the category of locally convex topological vector spaces).
    Since $X$ is $k$-Hausdorff, the compact-open topology is the topology of uniform convergence on compact subsets and hence, the map 
        $$ \cco(X) \hookrightarrow P, \;\;\; f \mapsto (f|_K) $$
    is an embedding of locally convex topological vector spaces. By the Hahn-Banach extension theorem, $\phi$ has an extension $\tilde{\phi}$ to $P$. Every linear functional on a product of locally convex topological vector spaces is a finite sum of functionals factoring through one of its factors, so there exist $K_1,\dots,K_n\subseteq X$ compact and $\phi_1\in C(K_1)^\wedge, \dots \phi_n \in C(K_n)^\wedge$ such that 
        $$ \tilde{\phi} = \sum_{i=1}^n \phi_i \circ \pi_i, $$
    where $\pi_i: P \to C(K_i)$ is the canonical projection. By the  Riesz–Markov–Kakutani theorem for the case of compact Hausdorff spaces, there exist Baire measures $\mu_1, \dots, \mu_n$ on $K_1,\dots, K_n$ such that 
        $$ \phi_i = \int_{K_i} (-) \diff \mu_i, $$
    for every $i\in\{1,\dots, n\}$. Hence, we obtain that 
        $$ \phi = \sum_{i=1}^n \int_{K_i} (-) \diff \mu_i. $$
    Let $\iota_i: K_i \hookrightarrow X$ be the inclusion and let 
        $$ \mu := \sum_{i=1}^n (\iota_i)_*\mu_i \in \mathcal{M}_c(X). $$
    Then we have that 
        $$ \phi = \int_X (-) \diff \mu,$$
    which is what we wanted to show.
\end{proof}

\begin{thm}\label{thm_riesz_type_repr_thm_Mc}
    Let $X$ be an $hk$-space. Then the map
        $$ \Phi:  \widetilde{\mathcal{M}}_c(X) \to \mathcal{M}_c(X), \;\;\; \mu \mapsto  \int_X (-) \diff \mu$$
    is a well-defined natural bijection between the set $\widetilde{\mathcal{M}}_c(X)$ of compactly supported $k$-regular measures on $X$ and the set $C(X)^*\cong \mathcal{M}_c(X)$. In particular, when $X$ is Hausdorff $\mathcal{M}_c(X)$ can be identified with the set of compactly supported Radon measures on $X$ (see \cref{rem_csupp_k_reg}). 
\end{thm}
\begin{proof}
    \text{}
    \begin{enumerate}[1.]
        \item \emph{$\Phi$ is well-defined:} 
        \begin{enumerate}[a.]
            \item \emph{$\Phi(\mu)$ is continuous on $C(X)$ (for any $\mu\in \widetilde{\mathcal{M}}_c(X)$:} it is even continuous on $\cco(X)$, since 
                $$ \Big|\int_X f(x) \diff \mu(x) \Big| \leq \mu(X) \,\sup_{x\in K} |f(x)|, $$
            for all $f\in C(X)$, where $K := \supp \mu$ is the support of $\mu$.
            \item \emph{$\im \Phi \subseteq \mathcal{M}_c(X)$:} Let $\mu\in \widetilde{\mathcal{M}}_c(X)$. We want to show that 
                $$ \tilde{\mu}:=\int_X (-) \diff \mu \in \mathcal{M}_c(X). $$
            Let $K:=\supp \mu$. Since $K$ is compact, $\mathcal{M}_c(K) = C(K)^\wedge$. Now, 
                $$ \tilde{\mu}|_K := \int_K (-) \diff \mu \in \mathcal{M}_c(K) = C(K)^\wedge. $$
            Let $\iota: K \hookrightarrow X$ be the inclusion. Then 
                $$ \tilde{\mu} = \iota_*(\tilde{\mu}|_K) \in \mathcal{M}_c(X). $$
        \end{enumerate}
        \item \emph{$\Phi$ is injective:} Since $\mathcal{M}_c(X)$ can be identified with a subspace of $\mathcal{M}(X)$ (\cref{lem_istar_injective}), we already know that $\Phi$ is injective from \cref{riesz_for_cb_general}.
        \item \emph{$\Phi$ is surjective:} As $\Phi$ has dense image (since linear combinations of point masses are dense in $\mathcal{M}_c(X)$), it suffices to show that $\im \Phi \subseteq \mathcal{M}_c(X)$ is closed. $\mathcal{M}_c(X)$ is a closed subspace of $C(C(X))$, by definition, and a closed subset of $\cco(C(X))$ is also closed in $C(C(X))$. It is therefore sufficient to prove that $\im \Phi$ is closed as a subset of $\cco(C(X))$. Moreover, $\cco(\cco(X))$ is a subspace (in the sense of carrying the subspace topology) of $\cco(C(X))$ because the compact subsets of $C(X)$ and $\cco(X)$ coincide. This in turn reduces the claim to showing that $\im \Phi$ is closed in $\cco(\cco(X))$, which follows from the fact that $\im \Phi$ consists precisely of those continuous functions in $\cco(\cco(X))$ which are linear, by \cref{functionals_on_cco}.
    \end{enumerate}
\end{proof}

\section{Monadic Vector-Valued Integration}\label{sec_monadic_vec_int}

\subsection{The vector-valued integral}

We now show how the preceding developments can be applied to obtain a particularly simple and coherent theory of vector-valued integration. Implementing a general idea of Kock \cite{kock2011commutative} which was refined and further developed by Lucyshyn-Wright in his doctoral thesis \cite{lucyshyn2013riesz} (see there for further references), we realise our vector valued integral by viewing paired linear $hk$-spaces as modules over the commutative monad $\mathcal{M}_c$. All the desired properties of the resulting vector-valued integral follow naturally from this, without much further work, which is the reason for the term \emph{monadic vector-valued integration}. In this setting, the two notational components of the integral, the integral sign $\smallint{}$ and what comes after it (``$f(x) \diff \mu(x)$''), both have separate precise meanings which we now define.

\begin{defn}\label{defn_Mc_module_str_on_plin_ksp}
    Let $V$ be a paired linear $hk$-space. Since $\mathcal{M}_c(V)$ is the free paired linear $hk$-space on $V$, the identity $V\to V$ extends uniquely to a morphism  $\mathcal{M}_c(V)\to V$ of paired linear $hk$-spaces. We denote this map by 
        $$ \smallint{}: \mathcal{M}_c(V)\to V \;\;\;\;\; \text{or} \;\;\;\; \int: \mathcal{M}_c(V)\to V $$
    Moreover, let us introduce the following notation for the pushforward:
        $$ f(x) \diff \mu(x) := f_*\mu. $$
\end{defn}

That this definition agrees with the usual Lebesgue integral in the scalar case is guaranteed by the following fact.

\begin{prop}
    Let $f: \mathbb{K} \to \mathbb{K}$ be continuous and let $\mu\in \mathcal{M}_c(\mathbb{K})$. Then,
        $$ \smallint{}(f_*\mu) = \int_{\mathbb{K}} f(x) \diff \mu(x) = \int_{\mathbb{K}} f(x) \diff \mu(x), $$
    where the right hand side denotes a Lebesgue integral in the usual sense, and the $\smallint{}$ on the left hand side denotes the map,
        $$ \smallint{}: \mathcal{M}_c(\mathbb{K})\to \mathbb{K}, \;\; \nu \mapsto \nu(\id_{\mathbb{K}}), $$
    from \cref{defn_Mc_module_str_on_plin_ksp}. 
\end{prop}
\begin{proof}
    This is simply the change-of-variables formula for the Lebesgue integral:
     $$ \int_{\mathbb{K}} f(x) \diff \mu(x) = \int_{\mathbb{K}} x \diff f_*\mu(x) = (f_*\mu)(\id_{\mathbb{K}}) = \smallint{}(f_*\mu).$$
\end{proof}

\begin{remark}\label{why_restrict_to_comp_supp}
    We restrict our attention to integration of continuous maps against compactly supported measures for two reasons. On the one hand, this is the setting which naturally arises from the monad structure of $\mathcal{M}_c$, allowing us to contemplate the phenomena specific to this monadic vector-valued integration theory. 
    On the other hand, this is a sufficient level of generality for many important cases of vector-valued integration such as the integration of continuous curves (which is certainly covered by this setting). For example, integrals of continuous curves are crucial in infinite-dimensional calculus (see Part I of \cite{hamilton1982inverse}), as well as the theory of vector-valued holomorphic functions (see \cite[Chapter 15]{garrett2018modern}).
\end{remark}

\begin{remark}\label{remark_vv_integral_from_free_forget_adj}
    We can now interpret the ($\spaces$-enriched) free-forgetful adjunction between paired linear $hk$-spaces and $hk$-spaces (see Chapter 3, \cref{remark_on_Mc_spc_of_meas_and_rel_to_vv_int}) in terms of vector-valued integration as follows. Consider for any $hk$-space $X$ and paired linear $hk$-space $V$ the map
        $$ I: C(X,V) \to [\mathcal{M}_c(X), V]. \;\; f \mapsto \left(\mu \mapsto \int f(x) \diff \mu(x) \right) $$
    For each $f\in C(X,V)$, the morphism $I(f)$ makes the diagram 
    \[\begin{tikzcd}
	{\mathcal{M}_c(X)} & V \\
	X
	\arrow["{I(f)}", dashed, from=1-1, to=1-2]
	\arrow["{\delta_\bullet}", from=2-1, to=1-1]
	\arrow["f"', from=2-1, to=1-2]
    \end{tikzcd}\]
    commute, since $I(f)(\delta_x) = \smallint{}(f_*\delta_x)= f(x)$. By the uniqueness of such morphism, $I$ is exactly the natural homeomorphism, 
        $$ I: C(X,V) \to [\mathcal{M}_c(X), V], $$
    constituting the free-forgetful adjunction between paired linear $hk$-spaces and $hk$-spaces. 
\end{remark}

\subsubsection{The change-of-variables formula} Functoriality of $\mathcal{M}_c$ immediately yields the change-of vairables formula for the vector-valued integral:

\begin{prop}
    Let $X,Y$ be $hk$-spaces, let $V$ be a paired linear $hk$-space and let $f:Y \to V$, $g: X \to Y$ be continuous maps. Then: 
        $$ f(g(x)) \diff \mu(x) = f(x) \diff g_*\mu. $$
\end{prop}
\begin{proof}
    By definition, and functoriality of $\mathcal{M}_c$,
    $$ f(g(x)) \diff \mu(x) = (f\circ g)_*\mu = f_*(g_*\mu) = f(x) \diff g_*\mu. $$
\end{proof}

\subsubsection{Interchange with linear operators} We may freely interchange integral signs and morphisms of paired linear $hk$-spaces:
\begin{prop}\label{prop_exchange_of_int_and_operators}
    Let $X$ be an $hk$-space and let $f:V\to W$ be a morphism of paired linear $hk$-spaces. Then for every continuous map $g:X\to V$ and every $\mu\in \mathcal{M}_c(X)$, 
         $$ f \Big( \int g(x) \diff \mu(x) \Big) = \int f(g(x)) \diff \mu(x). $$
\end{prop}
\begin{proof}
    First note that the claim holds when $\mu=\delta_{x_0}$ is a Dirac delta measure at $x_0\in X$. The general case now follows from the fact that $f$ is linear and continuous, and that linear combinations of Dirac measures are dense in $\mathcal{M}_c(X)$.
\end{proof}

\subsubsection{Continuous bilinearity} One feature of the present theory that needs to be emphasised is that the integration pairing, 
    $$ \int: C(X,V) \times \mathcal{M}_c(X) \to V, \;\; (f, \mu) \mapsto \int f(x) \diff \mu(x), $$
is continuous, which is a direct consequence of cartesian closure together with the continuity of $f \mapsto f_*$, and crucially depends on this setup. If were considering $C(X,V)$ and $\mathcal{M}_c(X)$ as locally convex topological vector spaces, and used the $\topsp$-product $\times_{\topsp}$, this could not possibly hold (see \cref{sec_continuity_problems}).

\subsection{A vector-valued Fubini theorem}

\subsubsection{Associativity of the integral} We would like to be able to form double integrals and formulate a Fubini theorem. However, the term 
    $$ \int \int f(x, y) \diff \mu(x) \diff \nu(y), $$
which representing in such double integral is \emph{a priori} ambiguous in our notation, as this double integral could -- \emph{a priori} -- have two distinct interpretations, depending on parentheses. That a rearrangement of such parentheses does not change the result might be called the \emph{associativity} of the integral (a fact, which, interestingly, does not seem to have a ``classical'' counterpart as it relies on an independent precise meaning of expressions of the form ``$(f(x, y) \diff \mu(x)) \diff \nu(y)$''). 

\begin{prop}[Associativity of integration]
    Let $X, Y$ be $hk$-spaces and let $V$ be a paired linear $hk$-space. Let $f: X\times Y \to V$ be a continuous map, $\mu\in \mnd (X), \nu \in \mathcal{M}_c(Y)$. Then 
    \begin{equation}\label{integral_assoc}
        \int \Big( \int f(x, y) \diff \mu(x) \Big) \diff \nu(y) = \int \Big(\int \Big( (f(x, y) \diff \mu(x)) \diff \nu(y) \Big)\Big) 
    \end{equation}
    Therefore, we may drop parentheses in these expressions all together, simply writing
        $$ \int \int f(x, y) \diff \mu(x) \diff \nu(y)  $$ 
    for either side of equation \eqref{integral_assoc}. 
\end{prop}
\begin{proof}
    We simply rewrite, 
    \begin{align*}
         \int \Big( \int f(x, y) \diff \mu(x) \Big) \diff \nu(y)
        &=  \int [ \smallint{} \circ (f(x,-)\diff \mu(x)) ]_*\nu\\
        &=  \int ((\smallint{})_* \circ (f(x,-)\diff \mu(x))_*)(\nu)\\
        &=  \int (\smallint{})_* ( f(x,y)\diff \mu(x) \diff \nu(y) )\\
        &=  \int \Big(\int \Big( f(x, y) \diff \mu(x) \diff \nu(y) \Big)\Big).
    \end{align*}
\end{proof}

\subsubsection{Commutativity of the integral (Fubini)} With unambiguous, associative double integrals in place, we now come to a vector-valued Fubini theorem. 

\begin{prop}
    Let $X, Y$ be $hk$-spaces and let $V$ be a paired linear $hk$-space. Let $f: X\times Y \to V$ be a continuous map, $\mu\in \mnd X, \nu \in \mathcal{M}_c(Y)$. Then: 
        $$ \int \int f(x, y) \diff \mu(x) \diff \nu(y) = \int \int f(x, y) \diff \mu(x) \diff \nu(y) = \int \int f(x, y) \diff (\mu\otimes\nu)(x,y). $$
\end{prop}
\begin{proof}
    Clearly, the identity holds for all finitely supported $\mu,\nu$ and both sides are continuous in $(\mu, \nu)$. Since, finitely supported measures are dense in $\mathcal{M}_c(-)$, this implies the claim.
\end{proof}

\subsection{The fundamental theorem of calculus} As mentioned before in \cref{why_restrict_to_comp_supp}, one of the most prominent applications of vector-valued integration is infinite-dimensional calculus. To show that our setting fulfils a minimal requirement for such purposes, we provide a version of the fundamental theorem of calculus for curves in paired linear $hk$-spaces. We first need the following straightforward definition of a continuously differentiable curve in a paired linear $hk$-space.

\begin{defn}
    Let $V$ be a replete linear $hk$-space. A continuous curve $\gamma:[0,1]\to V$ is said to be \emph{continuously differentiable} if for all $t\in (0,1)$, the limit, 
        $$ \gamma'(t) := \frac{\gamma(t+h)-\gamma(t)}{h}, $$
    exists and the map $t\mapsto \gamma'(t)$ is continuous with a (necessarily unique) extension to a continuous curve $\gamma':[0,1]\to V$. 
\end{defn}

\begin{prop}[Fundamental theorem of calculus]
    Let $V$ be a paired linear $hk$-space and let $\gamma:[0,1]\to V$ be a continuously differentiable curve. Then for all $t\in [0,1]$,
        $$ \int_0^t \gamma'(s) \diff s := \int^V \gamma'(s) \diff \lambda_{[0,t]}(s) = \gamma(t)-\gamma(0). $$
    Conversely, let $\gamma:[0,1]\to V$ be a continuous curve. Then
        $$ t \mapsto \int_0^t \gamma(s) \diff s $$
    is continuously differentiable and 
        $$ \frac{\diff}{\diff t} \int_0^t \gamma(s) \diff s = \gamma(t). $$
\end{prop}
\begin{proof}
    For the first part, let $\lambda [a,b]\in \mathcal{M}([0,1])$ be the Lebesgue measure restricted to the interval $[a,b]$ (where $0<a<b<1$). Observe that for all $t\in (0,1)$,
        $$ \lim_{h\to 0}\frac{1}{h}\lambda{[t,t+h]} = \delta_t. $$
    Using continuity of the integral, we see that for all $t\in (0,1)$,
    \begin{align*}
        \int_0^t \gamma'(s) \diff s &= \int \Big(\lim_{h\to 0} \frac{1}{h} (\gamma(s+h)-\gamma(s))\Big)\diff \lambda [0,t](s)\\
        &= \lim_{h\to 0}\, \Big[\int \!\gamma(s) \diff \Big(\frac{1}{h}\lambda{[t,t + h]}\Big) \:-\: \int\! \gamma(s) \diff \Big(\frac{1}{h}\lambda{[0,h]}\Big)\Big]\\
        &= \int \gamma(s) \diff \Big(\lim_{h\to 0}\frac{1}{h}\lambda{[t,t+h]}\Big) -  \int \gamma(s) \diff \Big(\lim_{h\to 0}\frac{1}{h}\lambda{[0,h]}\Big) \\
        &= \int \gamma(s) \diff \delta_{t}(s) - \int \gamma(s) \diff \delta_{0}(s) = \gamma(t) - \gamma(0).
    \end{align*} 
    Continuity in $t$ then yields the claim for all $t\in [0,1]$. \par 
    For the second part, we find that for all $t\in (0,1)$,
    \begin{align*}
        \lim_{h\to 0}\frac{1}{h}\Big(\int_0^{t+h} \gamma(s) \diff s - \int_0^{t} \gamma(s) \diff s\Big) &=  \lim_{h\to 0} \int^V \gamma(s) \diff \Big(\frac{1}{h}\lambda [t, t+h]\Big) \\
        &= \int \gamma(s) \diff \Big(\lim_{h\to 0}\frac{1}{h}\lambda [t, t+h]\Big) \\
        &= \int \gamma(s) \diff \delta_{t} = \gamma(t).
    \end{align*}
    As before, the claim then also follows for all $t\in [0,1]$.
\end{proof}

\begin{remark}
    This proof is substantially different from the ``usual'' one concerning the setting of Fréchet spaces given in \cite[p. 71, Theorem 2.1.1; p. 73f, Theorem 2.2.2, Theorem 2.2.3]{hamilton1982inverse}.
\end{remark}

%% file: Chapters/5_More_On_Paired_Linear_k_Spaces.tex
\chapter{Relation to Locally Convex Topological Vector Spaces}

In this chapter, we will draw a detailed comparison between the setting of linear $k$-spaces and the classical theory of locally convex Hausdorff topological vector spaces (abbreviated LCTVS in the following). To this end, we will establish the following main points.
\begin{enumerate}[1.]
    \item \emph{Replete linear $hk$-spaces as $k$-ifications of LCTVS}: Replete linear $hk$-spaces are always the $k$-ification of some LCTVS (\cref{replete_implies_kloc_conv}). Moreover, as noted already by Frölicher and Jarchow \cite{frolicher1972dualitatstheorie}, the category of those linear $k$-spaces which are the $k$-ification of some LCTVS is equivalent to a full subcategory of the category of LCTVS (\cref{kloc_conv_equiv_ksat_lctvs}).
    \item \emph{Completeness of linear $k$-spaces:} The $k$-ification of an LCTVS is not always replete, and the missing condition to make this true is a completeness condition which we refer to as \emph{$k$-completeness}. We will show that a linear $hk$-space is replete if, and only if, it is both the $k$-ification of some LCTVS and $k$-complete (\cref{kloc_conv_kcompl_implies_replete}). This is our contribution, complementing the findings of \cite{frolicher1972dualitatstheorie} by providing an appropriate completeness condition for this context.
    \item \emph{Tensor products:} Fréchet spaces can be viewed as replete linear $hk$-spaces, as paired linear $hk$-spaces, or as complete LCTVS. Each of these interpretations comes with its own notion of tensor product. In the complete locally convex setting, the tensor product that satisfies the expected universal property is the  \emph{completed projective tensor product}. As turns out, for Fréchet spaces, the completed projective tensor product coincides with the tensor product as taken in the closed monoidal category of replete linear $hk$-spaces (\cref{prop_tens_prod_coincide}). As a consequence, it also coincides with the tensor product of paired linear $hk$-spaces.
\end{enumerate}

\section{\texorpdfstring{$k$-Locally Convex Linear $k$-Spaces}{k-Locally Convex Linear k-Spaces}}

Since our goal is to investigate the relationship between linear $k$-spaces and locally convex topological vector spaces, the following definition will be useful.

\begin{defn}[$k$-Locally convex linear $k$-space]\label{defn_kloc}
    \text{}
    \begin{enumerate}[1.]
        \item We call a linear $k$-space $V$ \emph{$k$-locally convex} if there is an LCTVS (i.e. \emph{Hausdorff} locally convex topological vector space) $E$ such that $V=kE$.
        \item Denote the category of LCTVS with continuous linear maps as morphisms by $\lctvs$ and write $\klctvs$ for the full subcategory of the category $\vect$ of linear $k$-spaces spanned by $k$-locally convex linear $k$-spaces.
    \end{enumerate}
\end{defn}

Note that $k$-ification provides a functor,
    $$k: \lctvs \to \klctvs. $$
In the other direction, we can also define a functor moving from linear $k$-spaces to locally convex topological spaces.

\begin{defn}[Associated LCTVS, compactly determined LCTVS]\label{defn_ksat_l}
    \text{}
    \begin{enumerate}[1.]
        \item Given a linear $k$-space $V$, let $lV$ be the LCTVS obtained by re-topologising $V$ with the topological vector space topology generated by the convex $0$-neighbourhoods of $V$. We call $lV$ the LCTVS \emph{associated to} $V$, or simply the \emph{associated LCTVS}. 
        \item Note that given a continuous linear map $f:V\to W$ between linear $k$-spaces, $f$ is also continuous when viewed as a map $lf$ from $lV$ to $lW$ (since the preimage of a convex set under a linear map is convex). 
        \item We denote the image of the resulting functor, restricted to $\klctvs$,
            $$l: \klctvs \to \lctvs, $$
        by $\lklctvs$. In other words, $\lklctvs$ is the full subcategory of $\lctvs$ spanned by spaces of the form $lV$ for some $k$-locally convex linear $k$-space $V$.
        \item An LCTVS $E$ will be called \emph{compactly determined} if it can be written as $E=lkF$ for some LCTVS $F$ (i.e.~if it is an object of $\lklctvs$). In other words, compactly determined LCTVS are exactly those LCTVS associated to some $k$-locally convex linear $k$-space.
    \end{enumerate}
\end{defn}

\begin{remark}
    The term ``compactly determined'' is due to Porta \cite{porta1972compactly}. A detailed study of this class of LCTVS can be found there.
\end{remark}

The precise relationship between $k$-locally convex linear $k$-spaces and compactly determined LCTVS was given in \cite{frolicher1972dualitatstheorie} (and the same statement is given in \cite[Theorem 2.2]{seip1979convenient}), which we record as:

\begin{thm}\label{kloc_conv_equiv_ksat_lctvs}
    \text{}
    \begin{enumerate}[1.]
        \item The functor
            $$ l: \klctvs \to \lctvs$$
        is a left adjoint and a right inverse to 
            $$ k: \lctvs \to \klctvs. $$
        \item As a consequence, we have an equivalence of categories, 
        \[\begin{tikzcd}
	       \lklctvs & \klctvs,
	       \arrow["k"', shift right=1, from=1-1, to=1-2]
	       \arrow["{l}"', shift right=1, from=1-2, to=1-1]
        \end{tikzcd}\]
        between the categories of $k$-locally convex linear $k$-spaces and compactly determined LCTVS.
        \item Moreover, the composite functor $l k:= l\circ k$,
            $$ l k: \lctvs \to \lklctvs, $$
        is right adjoint to the inclusion functor 
            $$ \lklctvs \hookrightarrow \lctvs, $$
        exhibiting $\lklctvs$ as a coreflective subcategory of $\lctvs$.
    \end{enumerate}
\end{thm}

\subsection{\texorpdfstring{Replete linear $k$-spaces are $k$-locally convex}{Replete linear k-spaces are k-locally convex}} The following proposition ensures that most of the linear $k$-spaces we have encountered, in particular all replete linear $hk$-spaces, are $k$-locally convex.

\begin{prop}\label{ksubsp_locally_convex}
    Let $V$ be a $k$-locally convex linear $k$-space and let 
        $$\iota: W\hookrightarrow V$$ 
    be a linear $k$-subspace. Then $W$ is a $k$-locally convex linear $k$-space, as well. 
\end{prop}
\begin{proof}
    Let $E$ be some LCTVS with $kE = V$ (which there is by $k$-local convexity). Endow $W$ with the subspace topology induced from $E$ and call the resulting space $W_{LC}$. As a linear subspace of a LCTVS, $W_{LC}$ is an LCTVS, as well, and we have an embedding
        $$ W_{LC} \hookrightarrow E, $$
    whose $k$-ification is a $k$-embedding
        $$ kW_{LC} \hookrightarrow kE=V. $$
    Hence, by our assumption that $\iota$ is a $k$-embedding, $W = kW_{LC}$ and therefore, $W$ is $k$-locally convex, as claimed.
\end{proof}

\begin{cor}\label{replete_implies_kloc_conv}
    Let $V$ be a replete linear $k$-space. Then $V$ is $k$-locally convex.
\end{cor}
\begin{proof}
    By definition, every replete linear $k$-space $V$ admits a closed embedding into $V^{\wedge\wedge}\subseteq C(V^\wedge)=k\cco(V^\wedge)$, so the claim follows directly from \cref{ksubsp_locally_convex}. 
\end{proof}

\subsection{\texorpdfstring{$k$-Local convexity and the double dual}{k-Local convexity and the double dual}} The following observation gives some evidence that what is missing from a $k$-locally convexity to imply repleteness is a kind of completeness condition. While repleteness means that the canonical map to the double dual is a \emph{closed} embedding, $k$-local convexity amounts to this map being merely a $k$-embedding \cite[Theorem 4.3]{frolicher1972dualitatstheorie}:

\begin{thm}\label{froelicher_jarachow}
    Let $V$ be a linear $k$-space. Then $V$ is $k$-locally convex if, and only, if
        $$ \eta^V : V \to V^{\wedge\wedge}$$
    is a $k$-embedding. 
\end{thm}

\section{\texorpdfstring{Completeness of Linear $k$-Spaces}{Completeness of Linear k-Spaces}}\label{sec_compl_of_lin_k_sp}

\subsection{\texorpdfstring{$k$-Completeness}{k-Completeness}}

The notion of completeness that seems most appropriate in the context of linear $k$-spaces is the following.

\begin{defn}[$k$-Completeness]\label{defn_kcompl_seqcompl}
    Let $V$ be either a linear $k$-space, or a topological vector space (or any vector space carrying some topology). 
    \begin{enumerate}
        \item Recall that a \emph{Cauchy net} is a net $(x_i)_{i\in (I,\leq)}$ in $V$ such that for every neighbourhood $U$ of the origin, there is some $i_0\in I$ such that for all $i,j\geq i_0$, we have $x_i-x_j\in U$. 
        \item A subset $S\subseteq V$ is \emph{totally bounded} if for every neighbourhood $U\subseteq V$ of the origin, there exists a finite subset $F\subseteq V$ such that $F+U\supseteq S$.
        \item A \emph{totally bounded Cauchy net} is a Cauchy net $(x_i)$ in $V$ such that $\{x_i\mid i\in I\}\subseteq V$ is totally bounded.
        \item $V$ is \emph{$k$-complete} if every totally bounded Cauchy net in $V$ converges. 
        \item $V$ is \emph{sequentially complete} if every Cauchy \emph{sequence} in $V$ converges.
    \end{enumerate}
\end{defn}

We will show that a linear $k$-space is replete if, and only if, it is $k$-locally convex and $k$-complete (\cref{kloc_conv_kcompl_implies_replete}).

\begin{remark}
    In \cite{akbarov2022stereotype}, $k$-complete LCTVS are referred to as \emph{pseudo-complete}, but since this terms seems to have various other meanings, we have adopted an alternative, perhaps more informative terminology.
\end{remark}

Observe that $k$-completeness implies sequential completeness:

\begin{prop}\label{prop_k_compl_implies_seq_compl}
    Let $V$ be either a linear $k$-space, or a topological vector space. Then every Cauchy sequence in $V$ is totally bounded. Hence, if $V$ is $k$-complete, then it is sequentially complete. 
\end{prop}
\begin{proof}
    Let $(x_n)$ be a Cauchy sequence in $V$ and let $U\subseteq V$ be a neighbourhood of the origin. Put 
        $$ S := \{x_n\mid n\in \mathbb{N}\}. $$
    We want to show that there is a finite subset $F\subseteq V$ such that $F+U\supseteq S$. Since $(x_n)$ is a Cauchy sequence, there is some $n_0\in\mathbb{N}$ such that for all $n,m\geq n_0$, we have $x_n-x_m \in U$. In particular, $x_n \in x_{n_0} + U$ for all $n\geq n_0$. Hence, letting  
        $$ F := \{x_n \mid n \in \{0, \dots, n_0\}\}, $$
    we see that 
        $$S \subseteq \bigcup_{n\in \{0, \dots, n_0\}} x_n + U = F + U, $$
    which is what we wanted to show. 
\end{proof}

\subsection{\texorpdfstring{Replete linear $k$-spaces are $k$-complete}{Replete linear k-spaces are k-complete}}

We will now show that every replete linear $k$-space is $k$-complete (\cref{replete_implies_k_complete}). This will require a sequence of lemmas.

\begin{lem}\label{tot_bdd_kE_implies_tot_bdd_E}
    Let $E$ be an LCTVS and let $S\subseteq kE$ be totally bounded in $kE$. Then $S$ is totally bounded in $E$, as well.
\end{lem}
\begin{proof}
    Let $U$ be a neighbourhood of the origin in $E$. Then $U$ is also a neighbourhood of the origin in $kE$ and hence there is a finite subset $F\subseteq E$ such that $F+U\supseteq S$, proving the claim.
\end{proof}

\begin{lem}\label{Cauchy_in_kE_implies_Cauchy_inE}
    Let $E$ be an LCTVS and let $(x_i)$ be a Cauchy net in $kE$. Then $(x_i)$ is also a Cauchy net in $E$.
\end{lem}
\begin{proof}
    $U$ be a neighbourhood of the origin in $E$. Then $U$ is also a neighbourhood of the origin in $kE$ and hence there is some $i_0\in I$ such that for all $i,j\geq i_0$, $x_i-x_j\in U$, proving the claim.  
\end{proof}

\begin{lem}\label{closed_tot_bdd_impls_cpct}
    Let $E$ be a $k$-complete LCTVS and let $K\subseteq E$ be closed and totally bounded. Then $K$ is compact.
\end{lem}
\begin{proof}
    In any LCTVS, a subset $S$ is compact if, and only if, it is totally bounded and complete (meaning that every Cauchy net with members in $S$ converges in $S$, see \cite[p.262, Theorem 39.9]{willard2012general}). Since $K$ is totally bounded, every Cauchy net in $K$ is totally bounded and thus convergent by $k$-completeness of $E$. Hence, $K$ is totally bounded and complete and therefore compact, as claimed.
\end{proof}

\begin{lem}\label{E_kcompl_implies_V_kcompl}
    Let $E$ be a $k$-complete LCTVS. Then $V:=kE$ is also $k$-complete.
\end{lem}
\begin{proof}
    Let $(x_i)$ be a totally bounded Cauchy net in $V$ and let 
        $$ K := \overline{\{x_i \mid i\in I\}}^{E}$$
    be the closure of the set of its members in $E$. By  \cref{tot_bdd_kE_implies_tot_bdd_E} and the fact that the closure of a totally bounded set is totally bounded \cite[p.~25, 5.1]{schaefer1971topological}, $K$ is totally bounded in $E$, and hence, by $k$-completeness of $E$ and \cref{closed_tot_bdd_impls_cpct}, it is compact. By Lemmas \ref{tot_bdd_kE_implies_tot_bdd_E} and \ref{Cauchy_in_kE_implies_Cauchy_inE}, $(x_i)$ is also a Cauchy net $E$. Using the $k$-completeness of $E$ once more, we see that $(x_i)$ converges to some $x\in K$. Since the topologies of $E$ and $V=kE$ coincide on the compact subset $K$, and therefore, $(x_i)$ converges also in $V$, which is what we wanted to show.
\end{proof}

\begin{cor}\label{CX_k_complete}
    For any $k$-space $X$, $C(X)$ is $k$-complete.
\end{cor}
\begin{proof}
    $C(X)=k\cco(X)$ and $\cco(X)$ is complete (since $X$ is a $k$-space, see \cite[p. 231, Theorem 12]{kelley2017general}). The claim then follows from  \cref{E_kcompl_implies_V_kcompl}.
\end{proof}

\begin{prop}[Closed subspaces inherit $k$-completeness]\label{closed_subsp_k_complete}
    Let $V$ be a $k$-complete linear $k$-space and $W\subseteq V$ be a closed linear subspace. Then $W$ is $k$-complete, as well.
\end{prop}
\begin{proof}
    Let $(x_i)$ be some totally bounded Cauchy net in $W$. Then, by $k$-completeness of $W$, $(x_i)$ converges to some $x\in V$ and since $W\subseteq V$ is closed, $x\in W$, completing the proof.
\end{proof}

\begin{cor}\label{replete_implies_k_complete}
    Every replete linear $k$-space is $k$-complete.
\end{cor}
\begin{proof}
    Any replete linear $k$-space $V$ can be identified with a closed subspace of $C(V^\wedge)$, so the claim follows from combining \cref{CX_k_complete,closed_subsp_k_complete}.
\end{proof}

 \subsection{\texorpdfstring{$k$-Locally convex and $k$-complete linear $k$-spaces are replete}{k-Locally convex and k-complete linear k-spaces are replete}}

 Complementing \cref{replete_implies_k_complete}, our next goal is to show that for $k$-locally convex linear $k$-spaces, $k$-completeness is also a sufficient condition for repleteness (\cref{kloc_conv_kcompl_implies_replete}). Again, this will follow from a number of lemmas.

\begin{lem}\label{tot_bdd_cauchy_nets_V_vs_lV}
    Let $V$ be a linear $k$-space and let $E:=l V$. Then: 
    \begin{enumerate}[1.]
        \item If $S\subseteq V$ is totally bounded in $V$, then it is also totally bounded in $E$.
        \item If $(x_i)$ is a Cauchy net in $V$, then it is a Cauchy net in $E$, as well.
    \end{enumerate}
\end{lem}
\begin{proof}
    \text{}
    \begin{enumerate}[1.]
        \item Let $U\subseteq E$ be a neighbourhood of the origin in $E$. By the definition of $E=l V$, $U$ contains a convex neighbourhood of the origin $C\subseteq V$ in $V$. By total boundedness of $S$ in $V$, there is a finite set $F\subseteq V$ such that $F+C\supseteq S$. Hence, $F+U\supseteq S$ and we see $S$ is also totally bounded in $E$.
        \item Again, let $U\subseteq E$ be a neighbourhood of the origin in $E$, which then $U$ contains a convex neighbourhood of the origin $C\subseteq V$ in $V$. $(x_i)$ being a Cauchy net in $V$, there is some $i_0$ such that for all $i,j\geq i_0$, $x_i-x_j\in C\subseteq U$. Therefore, $(x_i)$ is a Cauchy net in $E$, as well.
    \end{enumerate}
\end{proof}

\begin{cor}\label{tot_bdd_and_cauchy_coincident_E_kE}
    Let $E$ be a compactly determined LCTVS. Then:
    \begin{enumerate}[1.]
        \item The totally bounded subsets and Cauchy nets of $E$ and $kE$ coincide.
        \item $E$ is $k$-complete if, and only if, $kE$ is $k$-complete.
    \end{enumerate}
\end{cor}
\begin{proof}
    This follows from combining the statements of Lemmas \ref{tot_bdd_cauchy_nets_V_vs_lV}, \ref{Cauchy_in_kE_implies_Cauchy_inE}, \ref{tot_bdd_kE_implies_tot_bdd_E} and \ref{E_kcompl_implies_V_kcompl}.
\end{proof}

 \begin{lem}\label{kcomplete_universally_closed}
     Let $V=kE$ be a $k$-locally convex linear $k$-space (with $E$ some compactly determined LCTVS) and let $W\subseteq V$ be a linear $k$-subspace. Suppose that $W$ is $k$-complete and that $F := l W \subseteq E$ is an embedding, i.e.~that $F$ carries the subspace topology induced from $E$. Then $W$ is closed in $V$. 
 \end{lem}
 \begin{proof}
     Let $K\subseteq V$ be compact. Since $V$ is a $k$-space, it suffices to show that $W\cap K$ is closed in $K$. Let $(x_i)_{i\in I}$ be a net in $W\cap K$, converging in $K$. Now, $(x_i)$ is also convergent in $E$ (since the topologies of $E$ and $V=kE$ coincide on compact subsets), and therefore a Cauchy net in $E$ (since \emph{in a topological vector space}, any convergent net is a Cauchy net). Moreover, being contained in a compact set, it is totally bounded. Since $F=l W$ carries the subspace topology induced from $E$, $(x_i)$ is also a totally bounded Cauchy net in $F$, and by \cref{tot_bdd_and_cauchy_coincident_E_kE}, $(x_i)$ is then a totally bounded Cauchy net in $W$, as well. Hence, by $k$-completeness of $W$, $(x_i)$ converges to some $x\in W\cap K$, showing that $W\cap K\subseteq K$ is closed, which is what we wanted to show.
 \end{proof}

 \begin{lem}\label{polar_of_nbhd_cpct}
     Let $V$ be a linear $k$-space and let $U\subseteq V$ be a neighbourhood of the origin. Then 
        $$ K := \{\phi \in V^\wedge \mid \,\forall x \in U: |\phi(x)|\leq 1\} \subseteq V^\wedge \subseteq C(V) $$
     is compact. 
 \end{lem}
 \begin{proof}
     We apply the Arzelà-Ascoli theorem (\cref{thm_arzela_ascoli}), for which we have to verify the following conditions:
     \begin{enumerate}[1.]
         \item $K\subseteq C(V)$ is closed. (This is clear from the continuity of the evaluation map $\eta^V$, from which we see that $K$ is an intersection of pre-images of a closed set.)
         \item The closure of $K(x):=\{\phi(x) \mid\, \phi \in K\}\subseteq \mathbb{K}$ is compact, for each $x\in V$: let $x\in V$. Since $U$ is a neighbourhood of the origin there exists a scalar $0\not=a\in\mathbb{K}$ such that $ax \in U$. Now, $K(ax)=a K(x)$ is bounded in $\mathbb{K}$, by definition of $K$, showing that the closure of $K(x)$ is indeed compact. 
         \item $K$ is equicontinuous: if $\epsilon > 0$ and $x_0\in V$, then for all $x\in x_0 + (\epsilon/2)U$ and $\phi\in K$, 
            $$|\phi(x)-\phi(x_0)| < \epsilon.$$ 
     \end{enumerate}
 \end{proof}

 \begin{lem}\label{leta_ambedding}
     Let $V$ be a $k$-locally convex linear $k$-space. Then $l(\eta^V): lV \to l(V^{\wedge\wedge})$ is an embedding.
 \end{lem}
 \begin{proof}
     We know that $f:=l(\eta^V)$ is continuous and injective, so it suffices to show that it is open onto its image. Since every LCTVS has a basis of neighbourhoods of the origin consisting of closed absolutely convex sets, it is even sufficient to show that $f$ maps such neighbourhoods to neighbourhoods in $f(lV)$. So let $C$ be a closed absolutely convex neighbourhood of the origin in $lV$. The topology of $lV$ being generated by the absolutely convex open subsets of $V$, we may suppose that $C$ is a closed absolutely convex $0$-neighbourhood in $V$. \par 
     Consider
        $$ K := \{\phi\in V^\wedge \mid \,\forall x \in C: |\phi(x)|\leq 1\}.$$
     Then, by \cref{polar_of_nbhd_cpct}, $K$ is compact, and therefore, 
        $$ U := \{\eta^V(x)\in \eta^V(V)\mid \, \forall \phi \in K: |\eta(x)(\phi)|\leq 1\}$$
     is a $0$-neighbourhood in $f(lV)$ (since the topology of $f(lV)$, i.e.~the subspace topology induced from $l(V^{\wedge\wedge})$, is finer than the compact-open topology, in which $U$ is a basic open neighbourhood). Letting 
        $$ {}^\circ K := \{x \in V \mid \, \forall \phi \in K: |\phi(x)|\leq 1 \}$$
     (the ``pre-polar''), we have that $U=f({}^\circ K)$. By the \emph{bi-polar theorem} (see \cite[p.~126]{schaefer1971topological}), ${}^\circ K = C$ and hence, $U=f(C)$, so the image of $C$ is a $0$-neighbourhood in $f(lV)$, which is what we wanted to show.
 \end{proof}

 \begin{cor}\label{kloc_conv_kcompl_implies_replete}
     Let $V$ be a linear $k$-space. Then $V$ is replete if, and only if, $V$ is $k$-locally convex and $k$-complete.
 \end{cor}
 \begin{proof}
     From Corollaries \ref{replete_implies_kloc_conv} and  \ref{replete_implies_k_complete}, we have that every replete linear $k$-space is $k$-locally convex and $k$-complete. On the other hand, if $V$ is $k$-locally convex, then by  \cref{froelicher_jarachow}, the canonical evaluation map $\eta^V: V \to V^{\wedge\wedge}$ is a $k$-embedding. Now, if $V$ is furthermore $k$-complete, by Lemmas \ref{kcomplete_universally_closed} and \ref{leta_ambedding}, the image of $V$ in $V^{\wedge\wedge}$ under $\eta^V$ is closed, showing that $V$ is replete.
 \end{proof}

 \begin{cor}\label{k_ification_of_k_compl_LCTVS_replete}
     The $k$-ification of a $k$-complete locally convex topological vector space $E$ is a replete linear $k$-space.
 \end{cor}
 \begin{proof}
     The $k$-ification $kE$ is $k$-complete by \cref{tot_bdd_and_cauchy_coincident_E_kE} and $k$-locally convex by definition. Hence, the claim follows from  \cref{kloc_conv_kcompl_implies_replete}.
 \end{proof}

 \begin{cor}\label{cor_replete_equiv_lctvs}
     The category of replete linear $k$-spaces is equivalent (even isomorphic) to that of $k$-complete, compactly determined LCTVS, via the mutually inverse functors $k$ and $l$.
 \end{cor}
 \begin{proof}
     This follows by combining \cref{kloc_conv_kcompl_implies_replete} with the equivalence of the categories of $k$-locally convex linear $k$-spaces and compactly determined LCTVS (\cref{kloc_conv_equiv_ksat_lctvs}), keeping in mind that $k$-completeness is preserved by the functors $k$ and $l$ that constitute the equivalence (see  \cref{tot_bdd_and_cauchy_coincident_E_kE}).
 \end{proof}

\subsection{Comparison to Pettis integral}\label{sec_comp_to_pettis} We now compare the vector-valued integral from Chapter 4, \cref{sec_monadic_vec_int}, with the classical \emph{Pettis integral} \cite{pettis1938integration}. 

\begin{prop}
    Let $\mu$ be a Radon measure on a compact Hausdorff space $X$, let $E$ be a $k$-complete LCTVS and let $f:X\to E$ be a continuous map. By \cref{k_ification_of_k_compl_LCTVS_replete}, $kE$ is a replete linear $hk$-space (and hence a paired linear $hk$-space, by letting $(kE)^*:=(kE)^\wedge$). Now, the vector valued integral (as defined in \cref{defn_Mc_module_str_on_plin_ksp}), 
        $$ I := \int_X f(x) \diff \mu(x), $$
    is the unique element $I\in E$ satisfying that for all continuous linear functionals $\phi\in E$,
    \begin{equation}\label{eq_pettis_comparison}
        \phi(I) = \int_X \phi(f(x)) \diff \mu(x),
    \end{equation}
    where the left hand side denotes an ordinary Lebesgue integral. 
\end{prop}
\begin{proof}
    That $I$ satisfies \cref{eq_pettis_comparison} follows from \cref{prop_exchange_of_int_and_operators}. That it is the unique element of $E$ with this property follows from the fact that the dual of $E$ separates points (since $E$ is locally convex).
\end{proof}

\begin{remark}
    The unique element $I\in E$ is know as the \emph{Pettis integral} or \emph{weak integral} of $f$ against $\mu$ (see \cite[p.~77, Definition 3.26]{rudin1991functional}). While its uniqueness follows directly from local convexity, as we have seen, its existence is non-trivial (see, for example, \cite[p.~78, Theorem 3.27]{rudin1991functional}). Here, we have given an alternative way to show existence by passing to the setting of replete (or paired) linear $hk$-spaces and using the ``monadic'' construction from Chapter 4. As a consequence, the vector-valued integral from \cref{sec_monadic_vec_int} coincides with the classical Pettis integral (whenever the two are comparable).
\end{remark}

\subsection{The sequential completion of locally convex Hausdorff topological vector spaces} Just as one can define the completion of an LCTVS, one can also construct its \emph{sequential completion} satisfying an analogous universal property. We will need this construction in \cref{chap_5_sec_tensor_prods} and recall it here for the convenience of the reader. Let us first give an explicit definition.

\begin{defn}[Sequential completion]\label{defn_scompl}
    Let $E$ be an LCTVS and let $\overline{E}$ be the completion of $E$ (which is an LCTVS, see \cite[p. 49, 4.1]{schaefer1971topological}). Define the \emph{sequential completion} $\widehat{E}$ to be the sequential closure (i.e.~the smallest sequentially closed superset) of $E$ in $\overline{E}$.
\end{defn}

The sequential completion deserves this name for the reasons summarised in the following two propositions.

\begin{prop}
    The sequential completion $\widehat{E}$ of an LCTVS $E$ is sequentially complete. 
\end{prop}
\begin{proof}
    Let $(x_n)$ be a Cauchy sequence in $\widehat{E}$. Then $(x_n)$ is a  Cauchy sequence also in $\overline{E}$ and, by completeness of $\overline{E}$, it converges to some $x\in \overline{E}$. Since $\widehat{E}\subseteq \overline{E}$ is sequentially closed, $x\in \widehat{E}$, so $(x_n)$ converges in $\widehat{E}$ and $\widehat{E}$ is sequentially complete.
\end{proof}

\begin{prop}[Universal property of the sequential completion]
    Let $f:E\to F$ be a continuous linear map, with $E$ an LCTVS and $F$ a sequentially complete LCTVS. Then there exists a unique continuous linear map, 
        $$ \tilde{f}: \widehat{E} \to F, $$
    making the diagram 
    \[\begin{tikzcd}
	{\widehat{E}} & F \\
	E
	\arrow[from=2-1, to=1-1]
	\arrow["{\tilde{f}}", dashed, from=1-1, to=1-2]
	\arrow["f"', from=2-1, to=1-2]
    \end{tikzcd}\]
    commute, where $E\hookrightarrow \widehat{E}\subseteq \overline{E}$ is the inclusion map. 
\end{prop}
\begin{proof}
    By the universal property of the completion, there exists a unique continuous linear map $\overline{f}$ making the diagram 
    \[\begin{tikzcd}
	{\overline{E}} & F \\
	E
	\arrow[from=2-1, to=1-1]
	\arrow["{\overline{f}}", dashed, from=1-1, to=1-2]
	\arrow["f"', from=2-1, to=1-2]
    \end{tikzcd}\]
    commute. Since $F$ is sequentially complete and $\overline{f}$ is continuous, the image of $\widehat{E}$ (which is defined as a sequential closure) under $\overline{f}$ is contained in $F$ (which is sequentially closed). Hence, restricting $\overline{f}$ to $\widehat{E}$ yields the desired map $\tilde{f}: \widehat{E} \to F$.
\end{proof}

\section{Tensor Products}\label{chap_5_sec_tensor_prods}

The purpose of this section is to compare the notions of tensor products of replete linear $hk$-spaces or paired linear $hk$-spaces with the classical \emph{projective tensor product} of locally convex topological vector spaces. In particular, we will see that in the case of Fréchet spaces as well as in the dual case of Brauner spaces all of these notions coincide.

\subsection{\texorpdfstring{The universal property of the tensor product of replete linear $hk$-spaces}{The universal property of the tensor product of replete linear hk-spaces}} In order to achieve our goal, we would like to formulate and compare universal properties of the different tensor products. We begin with the case of replete linear $hk$-spaces. \par 
According to \cref{cor_d_sep_closed_symm_mon}, the category $\sclin$ of replete linear $hk$-spaces has a closed symmetric monoidal structure, with the tensor product given by the repletion of the tensor product of linear $hk$-spaces. Denote this tensor product by $\otimes_r$. Since the tensor product of linear $hk$-spaces comes equipped with a canonical map $V\times W \to V\otimes W$, composing with the canonical map $V\otimes W\to V\otimes_r W$ to the repletion, we likewise have a canonical map $V\times V \to V\otimes_r W$. With this in mind, we can formulate the following universal property.

\begin{prop}[Universal property of repleted tensor product]\label{univ_prop_repl_tens_prod}
    Let $V, W, Z$ be replete linear $hk$-spaces and let $b: V\times W \to Z$ be a continuous bilinear map. Then there exists a unique continuous linear map $\tilde{b}: V\otimes_r W \to Z$ such that the following diagram commutes:
    \[\begin{tikzcd}
	{V\otimes_r W} & Z \\
	{V\times W}
	\arrow[from=2-1, to=1-1]
	\arrow["{\tilde{b}}", dashed, from=1-1, to=1-2]
	\arrow["b"', from=2-1, to=1-2]
    \end{tikzcd}\]
\end{prop}
\begin{proof}
    By the universal property of the tensor product $\otimes$ of linear $hk$-spaces (which is a special case of the one for modules over a commutative monad, see \cref{tens_prod_t_modules_univ_prop}), there exists a unique continuous linear map $\tilde{b}_0$ making the diagram 
    \[\begin{tikzcd}
	{V\otimes W} & Z \\
	{V\times W}
	\arrow[from=2-1, to=1-1]
	\arrow["{\tilde{b}_0}", dashed, from=1-1, to=1-2]
	\arrow["b"', from=2-1, to=1-2]
    \end{tikzcd}\]
    commute. Applying the universal property of the repletion then yields the desired unique continuous linear map $\tilde{b}: V\otimes_r W \to Z$ that makes the following diagram commute:
    \[\begin{tikzcd}
	{V\otimes_r W} \\
	{V\otimes W} & Z \\
	{V\times W}
	\arrow[from=3-1, to=2-1]
	\arrow["{\tilde{b}_0}", dashed, from=2-1, to=2-2]
	\arrow["b"', from=3-1, to=2-2]
	\arrow[from=2-1, to=1-1]
	\arrow["{\tilde{b}}", curve={height=-6pt}, dashed, from=1-1, to=2-2]
    \end{tikzcd}\]
\end{proof}

\subsection{The sequentially completed projective tensor product of locally convex spaces}

Recall that the \emph{projective tensor product} $E\otimes_\pi F$ of two LCTVS $E$ and $F$, which comes with a canonical continuous bilinear map $E\times_{\topsp} F \to E\otimes_\pi F$, is characterised (up to unique isomorphism) by the following universal property (see \cite[p.~33f]{kriegl2016frechet}, also for an explicit construction of $E\otimes_\pi F$). For any further LCTVS $G$ and bilinear map $b: E\times_{\topsp} F  \to G$, there exists a unique continuous linear map $\tilde{b}: E\otimes_\pi F \to G$ making the diagram 
\[\begin{tikzcd}
	{E\otimes_\pi F} & G \\
	{E\times F}
	\arrow[from=2-1, to=1-1]
	\arrow["{\tilde{b}}", dashed, from=1-1, to=1-2]
	\arrow["b"', from=2-1, to=1-2]
\end{tikzcd}\]
commute. 

\begin{defn}
    Let $E,F$ be $k$-complete LCTVS. The \emph{sequentially completed projective tensor product} of $E$ and $F$, 
        $$ E \;\widehat{\otimes}_\pi F := (E\otimes_\pi F)\widehat{\;}, $$
    is the sequential completion (see \cref{defn_scompl}) of the projective tensor product .
\end{defn}

\begin{warning}
    The notation $\widehat{\otimes}_\pi$ is commonly used for the \emph{completion} of the projective tensor product, which in general may be strictly larger than the sequential completion. However, within our context, this will not result in ambiguity, since we are interested in the case of Fréchet spaces, where the two notions agree by metrisability.
\end{warning}

The sequentially completed projective tensor product satisfies the expected universal property.

\begin{prop}\label{univ_prop_scompl_proj_tens_prod}
    Let $E, F, G$ be sequentially complete LCTVS and let $b: E\times F \to G$ be a continuous bilinear map. Then there exists a unique continuous linear map $\tilde{b}: E\;\widehat{\otimes}_\pi F \to G$ such that the following diagram commutes:
    \[\begin{tikzcd}
	{E\:\widehat{\otimes}_\pi F} & G \\
	{E\times F}
	\arrow[from=2-1, to=1-1]
	\arrow["{\tilde{b}}", dashed, from=1-1, to=1-2]
	\arrow["b"', from=2-1, to=1-2]
    \end{tikzcd}\]
\end{prop}
\begin{proof}
    Combine the universal property of the projective tensor product of LCTVS with that of the sequential completion.
\end{proof}

\subsection{\texorpdfstring{The tensor product of replete linear $hk$-Spaces and the projective tensor product}{The tensor product of replete linear hk-Spaces and the projective tensor product}}

\begin{prop}\label{prop_tens_prod_coincide}
    Let $V,W$ be Fréchet spaces. Then the completed (equivalently, sequentially completed) projective tensor product $V \;\,\widehat{\otimes}_\pi W$ agrees with the repleted tensor product $V \otimes_r W$ of $V$ and $W$. 
\end{prop}
\begin{proof}
    It suffices to verfify that $V \;\,\widehat{\otimes}_\pi W$ satisfies the universal property of the tensor product of replete linear $hk$-spaces (see \cref{univ_prop_repl_tens_prod}). Let $Z$ be some further replete linear $hk$-space and let $b: V\times W \to Z$ be a continuous bilinear map. Since $Z$ is replete, there exists a $k$-complete, compactly determined LCTVS $E$ such that $Z=kE$ (see \cref{cor_replete_equiv_lctvs}). In particular, $E$ is sequentially complete (by \cref{prop_k_compl_implies_seq_compl}). Since $V$ and $W$ are metrisable, $V\times_{\topsp} W$ is a $k$-space and therefore, $b$ is also continuous as a map $b: V\times_{\topsp} W \to E$. By the universal property of the sequentially completed projective tensor product (see \cref{univ_prop_scompl_proj_tens_prod}), there is a unique continuous linear map $\tilde{b}: V \;\,\widehat{\otimes}_\pi W \to E$ making the diagram 
    \[\begin{tikzcd}
	{V\;\widehat{\otimes}_\pi W} & E \\
	{V\times W}
	\arrow[from=2-1, to=1-1]
	\arrow["{\tilde{b}}", dashed, from=1-1, to=1-2]
	\arrow["b"', from=2-1, to=1-2]
    \end{tikzcd}\]
    commute. But $V \;\widehat{\otimes}_\pi W$ is a Fréchet space and hence a $k$-space, so a continuous map $V \widehat{\otimes}_\pi W \to E$ is equivalently a continuous map $V \widehat{\otimes}_\pi W \to Z$. Therefore, there exists a unique continuous linear map $\tilde{b}: V \widehat{\otimes}_\pi W \to Z$ (the same $\tilde{b}$ as before) making the diagram 
    \[\begin{tikzcd}
	{V\;\widehat{\otimes}_\pi W} & Z \\
	{V\times W}
	\arrow[from=2-1, to=1-1]
	\arrow["{\tilde{b}}", dashed, from=1-1, to=1-2]
	\arrow["b"', from=2-1, to=1-2]
    \end{tikzcd}\]
    commute, showing that $V \widehat{\otimes}_\pi W$ does indeed satisfy the universal property of the tensor product of replete linear $hk$-spaces.
\end{proof}

Since the underlying linear $hk$-space of the tensor product $\widehat{\otimes}$ of \emph{paired} linear $hk$-spaces is exactly the tensor product of replete linear $hk$-spaces, we obtain the following immediate corollary.

\begin{cor}\label{cor_tens_prod_fre_vs_plin}
    Let $V, W$ be Fréchet spaces. Then their completed projective tensor product $V \;\widehat{\otimes}_\pi W$ agrees with the paired-linear-$hk$-space tensor product of $V$ an $W$. More precisely, let 
        $$ \iota: \mathsf{Fre} \hookrightarrow \plin $$
    be the inclusion functor from the category of paired linear $hk$-spaces to the category of Fréchet spaces. Then 
        $$ \iota(V \;\,\widehat{\otimes}_\pi W) \cong \iota(V) \:\widehat{\otimes}\: \iota(W), $$
    where the tensor product $\widehat{\otimes}$ on the right hand side denotes the tensor product of paired linear $hk$-spaces.
\end{cor}

%% file: Chapters/6_Appendix.tex
\appendix
\chapter{Relation to Condensed, Stereotype and Bornological Vector Spaces}\label{appendix_relation_to_condesed}

The following diagram summarises a number of relationships between the categories appearing in the main body of this thesis and those associated to other (``convenient'') functional-analytic categories. We will explain the notations appearing therein below. References to proofs of these relationships as well as to the relevant definitions will be given in \cref{categories_table,functors_table,inclusions_and_equivalences_table,adjunctions_table}.

\[\begin{tikzcd}
	{\mathsf{L}} & {hk\mathsf{Top}} & {\mathsf{Cond}(\mathsf{Set})_{q.s.}} & {\mathsf{Cond}(\mathsf{Set})} \\
	{r\mathsf{Vect}(hk\mathsf{Top})} & {\mathsf{Vect}(hk\mathsf{Top})} & {\mathsf{Cond}(\mathsf{Vect})_{q.s.}} & {\mathsf{Cond}(\mathsf{Vect})} \\
	{lk\mathsf{LCTVS}_{k\mathrm{c}}} & {lk\mathsf{LCTVS}} & {st\mathsf{CBS}} & {\mathsf{Conv}} \\
	{\mathsf{LCTVS}_{k\mathrm{c}}} & {\mathsf{LCTVS}} & {b\mathsf{LCTVS}} \\
	{\mathsf{Ste}}
	\arrow[""{name=0, anchor=center, inner sep=0}, shift right=3, hook, from=2-1, to=2-2]
	\arrow[""{name=1, anchor=center, inner sep=0}, "r"', shift right=3, from=2-2, to=2-1]
	\arrow[""{name=2, anchor=center, inner sep=0}, shift left=2, hook', from=2-1, to=1-1]
	\arrow[""{name=3, anchor=center, inner sep=0}, shift left=2, tail, from=1-1, to=2-1]
	\arrow["{l \;\;\wr}"', shift right, from=2-1, to=3-1]
	\arrow["{\wr \;\;k}"', shift right=2, from=3-1, to=2-1]
	\arrow[""{name=4, anchor=center, inner sep=0}, shift right, hook, from=3-1, to=4-1]
	\arrow[""{name=5, anchor=center, inner sep=0}, "{l\circ k}"', shift right=3, from=4-1, to=3-1]
	\arrow[""{name=6, anchor=center, inner sep=0}, shift right=3, hook, from=2-2, to=2-3]
	\arrow[""{name=7, anchor=center, inner sep=0}, shift right=3, from=2-3, to=2-2]
	\arrow[shift left, tail, from=2-3, to=3-3]
	\arrow[hook, from=3-1, to=3-2]
	\arrow[""{name=8, anchor=center, inner sep=0}, shift right=2, hook, from=4-1, to=4-2]
	\arrow[""{name=9, anchor=center, inner sep=0}, shift right=2, from=4-2, to=4-1]
	\arrow[""{name=10, anchor=center, inner sep=0}, shift right=2, hook, from=3-2, to=4-2]
	\arrow[""{name=11, anchor=center, inner sep=0}, "{l\circ k}"', shift right=2, from=4-2, to=3-2]
	\arrow["\wr"', shift left, from=3-3, to=4-3]
	\arrow[""{name=12, anchor=center, inner sep=0}, "b", shift left=2, from=4-2, to=4-3]
	\arrow[""{name=13, anchor=center, inner sep=0}, shift left=2, hook', from=4-3, to=4-2]
	\arrow[curve={height=6pt}, hook, from=4-3, to=3-2]
	\arrow["\wr"', shift right=3, from=4-3, to=3-3]
	\arrow[""{name=14, anchor=center, inner sep=0}, "l"', shift right=2, from=2-2, to=3-2]
	\arrow[""{name=15, anchor=center, inner sep=0}, "k"', shift right=2, hook, from=3-2, to=2-2]
	\arrow[""{name=16, anchor=center, inner sep=0}, shift right=2, tail, from=2-3, to=1-3]
	\arrow[""{name=17, anchor=center, inner sep=0}, shift right=2, from=1-3, to=2-3]
	\arrow[""{name=18, anchor=center, inner sep=0}, shift right=3, from=2-4, to=2-3]
	\arrow[""{name=19, anchor=center, inner sep=0}, shift right=3, hook, from=2-3, to=2-4]
	\arrow[""{name=20, anchor=center, inner sep=0}, shift right=2, from=1-4, to=2-4]
	\arrow[""{name=21, anchor=center, inner sep=0}, shift right=2, tail, from=2-4, to=1-4]
	\arrow[""{name=22, anchor=center, inner sep=0}, shift right=2, tail, from=2-2, to=1-2]
	\arrow[""{name=23, anchor=center, inner sep=0}, shift right=2, from=1-2, to=2-2]
	\arrow[""{name=24, anchor=center, inner sep=0}, shift right=2, hook, from=1-2, to=1-3]
	\arrow[""{name=25, anchor=center, inner sep=0}, shift right=2, from=1-3, to=1-2]
	\arrow[""{name=26, anchor=center, inner sep=0}, shift right=2, hook, from=1-3, to=1-4]
	\arrow[""{name=27, anchor=center, inner sep=0}, shift right=2, from=1-4, to=1-3]
	\arrow[""{name=28, anchor=center, inner sep=0}, shift right=2, shorten <=4pt, shorten >=4pt, tail, from=1-1, to=1-2]
	\arrow[""{name=29, anchor=center, inner sep=0}, "{\mathcal{M}_c}"', shift right=2, shorten <=4pt, shorten >=4pt, from=1-2, to=1-1]
	\arrow[""{name=30, anchor=center, inner sep=0}, shift left, hook', from=5-1, to=4-1]
	\arrow[""{name=31, anchor=center, inner sep=0}, shift left=3, from=4-1, to=5-1]
	\arrow[""{name=32, anchor=center, inner sep=0}, shift left=2, hook', from=3-4, to=3-3]
	\arrow[""{name=33, anchor=center, inner sep=0}, shift left=2, from=3-3, to=3-4]
	\arrow["\otimes", "\dashv"{anchor=center, rotate=-90}, shift right, draw=none, from=1, to=0]
	\arrow["\otimes", "\dashv"{anchor=center}, shift right, draw=none, from=23, to=22]
	\arrow["\otimes", "\dashv"{anchor=center}, shift right, draw=none, from=17, to=16]
	\arrow["\otimes", "\dashv"{anchor=center}, shift right, draw=none, from=20, to=21]
	\arrow["\dashv"{anchor=center}, draw=none, from=14, to=15]
	\arrow["\otimes", "\dashv"{anchor=center}, shift right, draw=none, from=2, to=3]
	\arrow["\dashv"{anchor=center}, draw=none, from=4, to=5]
	\arrow["\dashv"{anchor=center, rotate=-90}, draw=none, from=12, to=13]
	\arrow["\dashv"{anchor=center, rotate=-90}, draw=none, from=9, to=8]
	\arrow["\otimes", "\dashv"{anchor=center, rotate=-90}, shift right, draw=none, from=18, to=19]
	\arrow["\otimes", "\dashv"{anchor=center, rotate=-90}, shift right, draw=none, from=29, to=28]
	\arrow["\otimes", "\dashv"{anchor=center, rotate=-90}, shift right, draw=none, from=25, to=24]
	\arrow["\otimes", "\dashv"{anchor=center, rotate=-90}, shift right, draw=none, from=27, to=26]
	\arrow["\otimes", "\dashv"{anchor=center, rotate=-90}, shift right, draw=none, from=7, to=6]
	\arrow["\dashv"{anchor=center}, draw=none, from=30, to=31]
	\arrow["\dashv"{anchor=center}, draw=none, from=10, to=11]
	\arrow["\dashv"{anchor=center, rotate=-90}, draw=none, from=33, to=32]
\end{tikzcd}\]

\begin{notation}
    In the above diagram, the following notation is used.
\begin{enumerate}[1.]
    \item The ``hooked'' arrows $\hookrightarrow$ denote fully faithful functors and the arrows 
    \[\begin{tikzcd}
	{\text{``}} & {\text{''}}
	\arrow[tail, from=1-1, to=1-2]
    \end{tikzcd}\]
    denote faithful functors.
    \item As customary, the ``turnstile'' $\uparrow\, \dashv\, \downarrow$ means that the functor $\uparrow$ is left adjoint to the functor $\downarrow$. Moreover, adjunctions decorated with a tensor product (``$
    \dashv\otimes$'') indicate symmetric monoidal adjunctions (see \cref{symmmon_adj}). In particular, when the right adjoint is fully faithful, i.e.~we have a reflective subcategory inclusion,
    \[\begin{tikzcd}
	{\mathsf{A}} & {\mathsf{B},}
	\arrow[""{name=0, anchor=center, inner sep=0}, shift right=2, hook, from=1-1, to=1-2]
	\arrow[""{name=1, anchor=center, inner sep=0}, shift right=2, from=1-2, to=1-1]
	\arrow["\otimes", "\dashv"{anchor=center, rotate=-90}, shift right, draw=none, from=1, to=0]
    \end{tikzcd}\]
    of a symmetric monoidal closed category $\mathsf{A}$ into a closed symmetric monoidal category, then this inclusion preserves internal homs. In this case it even holds that if $A\in \mathsf{A}$ and $B\in\mathsf{B}$, then the internal hom $[A,B]$ is in $\mathsf{B}$ (i.e.~$\mathsf{A}$ is an \emph{exponential ideal} in $\mathsf{B}$; this is Day's reflection theorem, see Chapter 3, \cref{day_refl_thm}). In each case when there is such a symmetric monoidal adjunction, this is either proved in the main body of this thesis, or given below in \cref{lin_hk_cond_vect_exp_ideal}.   
\end{enumerate}
\end{notation}

\begin{lem}\label{lem_dense_refl_ccc_subcat_exp_ideal}
    Let $\mathsf{C}$ be a cartesian closed category and let 
    \[\begin{tikzcd}
	{\mathsf{D}} & {\mathsf{C}}
	\arrow[""{name=0, anchor=center, inner sep=0}, "\iota"', shift right=2, hook, from=1-1, to=1-2]
	\arrow[""{name=1, anchor=center, inner sep=0}, "L"', shift right=2, from=1-2, to=1-1]
	\arrow["\dashv"{anchor=center, rotate=-90}, draw=none, from=1, to=0]
    \end{tikzcd}\]
    be a reflective subcategory that is also cartesian closed. Assume furthermore that $\iota$ is \emph{dense}, meaning that the restricted Yoneda embedding
        $$ \mathsf{C} \to \mathsf{Set}^{\mathsf{D}}, \;\; X \mapsto \Hom_{\mathsf{C}}(\iota(-), X), $$
    is fully faithful. \par 
    Then $\mathsf{D}$ is an exponential ideal of $\mathsf{C}$, i.e.~whenever $Y\in \mathsf{D}$ and $X\in \mathsf{C}$, then $Y^X$ is in $\mathsf{D}$ (or, to be more precise, $\iota(X)^Y$ is in the essential image of $\iota$). 
\end{lem}
\begin{proof}
    We first show that $\iota$ preserves exponentials. Let $X,Y,Z\in \mathsf{D}$. Then, using that $\iota$ preserves products (since it is a right adjoint), 
    \begin{align*}
        \Hom_{\mathsf{C}}(\iota(X), \iota(Z)^{\iota(Y)}) &\cong \Hom_{\mathsf{C}}(\iota(X)\times \iota(Y), \iota(Z)) \\
        &\cong \Hom_{\mathsf{C}}(\iota(X\times Y), \iota(Z)) \\
        &\cong \Hom_{\mathsf{D}}(X\times Y, Z) \\
        &\cong \Hom_{\mathsf{D}}(X, Z^Y) \\
        &\cong \Hom_{\mathsf{C}}(\iota(X), \iota(Z^Y)) 
    \end{align*}
    Therefore, by our assumption that $\iota$ is dense, 
        $$ \iota(Z)^{\iota(Y)}\cong \iota(Z^Y). $$
    Now, let $X,Z\in \mathsf{D}$, $Y\in\mathsf{C}$. We calculate,
    \begin{align*}
        \Hom_{\mathsf{C}}(\iota(X), \iota(Z^{LY})) &\cong \Hom_{\mathsf{D}}(X, Z^{LY}) \\
        &\cong \Hom_{\mathsf{D}}(LY, Z^{X}) \\
        &\cong \Hom_{\mathsf{C}}(Y, \iota(Z^X)) \\
        &\cong \Hom_{\mathsf{C}}(Y, \iota(Z)^{\iota(X)})) \\
        &\cong \Hom_{\mathsf{C}}(\iota(X), \iota(Z)^{Y})). 
    \end{align*}
    Hence, using our assumption once more, 
        $$ \iota(Z)^{Y} \cong \iota(Z^{LY}). $$
    In particular, $\iota(Z)^{Y}$ is in the image of $\iota$, which is what we wanted to show.
\end{proof}

\begin{cor}\label{lin_hk_cond_vect_exp_ideal}
    The category $\spaces$ of $hk$-spaces (linear $hk$-spaces, resp.) is an exponential ideal in the category $\mathsf{Cond(Set)}_{q.s.}$ of quasi-separated condensed sets (quasi-separated condensed vector spaces, resp.), and similarly, the category of quasi-separated condensed sets (quasi-separated condensed vector spaces, resp.) is an exponential ideal in the category $\mathsf{Cond(Set)}$ of condensed sets (condensed vector spaces, resp.). In particular, the respective subcategory inclusions preserve all limits and internal homs.
\end{cor}
\begin{proof}
    For the embedding of $hk$-spaces into (quasi-separated) condensed sets as a reflective subcategory, see \cite[Proposition 1.7]{scholze2019condensed}. For the fact that quasi-separated condensed sets form a reflective subcategory of condensed sets, see \cite[Lemma 4.14]{sclausen2020lectures}. In both cases, the inclusion is dense (since the respective subcategories contain all compact Hausdorff spaces), and \cref{lem_dense_refl_ccc_subcat_exp_ideal} applies. \par 
    The reflective subcategory inclusion of linear $hk$-spaces into (quasi-separated) condensed vector space is also implied by \cite[Proposition 1.7]{scholze2019condensed} (see also \cite[p.~15]{clausen2022condensed}; and similarly for the inclusion of quasi-separated condensed vector spaces into condensed vector spaces). Now, the claim follows from the fact that $hk$-spaces form an exponential ideal in condensed sets, and that the internal hom in the category of condensed vector spaces is formed from equalisers and internal homs in $\mathsf{Cond(Set)}$ (cf.~the proof of \cref{L_VW_QCB} in Chapter 2).
\end{proof}

\begin{remark}
    \cref{lin_hk_cond_vect_exp_ideal} implies that our observations from Chapter 2 concerning the topology of the space of continuous linear maps $L(V,W)$ between linear $hk$-spaces $V,W$ can also be phrased as results concerning the internal hom of condensed vector spaces. For example, when $V$ and $W$ are separable Fréchet spaces, \cref{lem_LVW_seq} implies that the internal hom $[\underline{V}, \underline{W}]$ of the associated condensed vector spaces $\underline{V}, \underline{W}$ is (the condensed vector space associated to) $L(V,W)$ with the strong sequential topology. Similarly, questions concerning duality of linear $hk$-spaces can be phrased as questions about the duality of condensed vector spaces (see, for example, \cref{question_reflexivity}). In particular, for any condensed vector space $V$, the internal hom $[V,\mathbb{K}]$ (i.e.~the ``internal dual'') is (the condensed vector space associated to) a replete linear $hk$-space. 
\end{remark}

\begin{table}[]
\vspace*{2 cm}
\caption{\label{categories_table}Overview of categories}
\begin{tabular}{@{}lll@{}}
\toprule
\textit{Category}                     & \textit{Objects}                                                                                                        & \textit{References}                                                                                                                                                                                                  \\ \midrule
$hk\mathsf{Top}$                      & $hk$-spaces/CGWH spaces                                                                                                 & \cref{defn_k_sp_hk_sp}                                                                                                                                                                                               \\ \midrule
$\mathsf{Vect}(hk\mathsf{Top})$       & linear $hk$-spaces                                                                                                      & \cref{def_lin_hk_space}                                                                                                                                                                                              \\ \midrule
$r\mathsf{Vect}(hk\mathsf{Top})$      & replete linear $hk$-spaces                                                                                              & \cref{ex_repl_lin_sk_sp}                                                                                                                                                                                             \\ \midrule
$\plin$                               & paired linear $hk$-spaces                                                                                               & \cref{defn_paired_lin_hk_sp}                                                                                                                                                                                         \\ \midrule
$\mathsf{Cond(Set)}$                  & condensed sets                                                                                                          & \cite{scholze2019condensed}                                                                                                                                                                                          \\ \midrule
$\mathsf{Cond}_{q.s.}(\mathsf{Set})$  & \begin{tabular}[c]{@{}l@{}}quasi-separated (q.s.) condensed sets/\\ compactological spaces\end{tabular}                 & \begin{tabular}[c]{@{}l@{}}\cite[p.~15]{clausen2022condensed},\\ see \cite[p. 15]{clausen2022condensed} \\ for the equivalence to the ``compactological \\ spaces'' of \cite{lucien1971compactological}\end{tabular} \\ \midrule
$\mathsf{Cond(Vect)}$                 & condensed vector spaces                                                                                                 & \cite{scholze2019condensed}                                                                                                                                                                                          \\ \midrule
$\mathsf{Cond}_{q.s.}(\mathsf{Vect})$ & \begin{tabular}[c]{@{}l@{}}quasi-separated (q.s. )condensed/\\ compactological vector spaces\end{tabular}               & \begin{tabular}[c]{@{}l@{}}\cite[p.~15]{clausen2022condensed}, \\ \cite{lucien1971compactological}\end{tabular}                                                                                                      \\ \midrule
$\mathsf{LCTVS}$                      & \begin{tabular}[c]{@{}l@{}}Locally convex Hausdorff \\ topological vector spaces\end{tabular}                           & ---                                                                                                                                                                                                                  \\ \midrule
$lk\mathsf{LCTVS}$                    & compactly determined LCTVS                                                                                              & \cref{defn_ksat_l}, \cite{porta1972compactly}                                                                                                                                                                        \\ \midrule
$\mathsf{LCTVS}_{k\mathrm{c}}$        & $k$-complete/pseudo-complete LCTVS                                                                                      & \cref{defn_kcompl_seqcompl}, \cite[p. 370]{akbarov2022stereotype}                                                                                                                                                    \\ \midrule
$lk\mathsf{LCTVS}_{k\mathrm{c}}$      & \begin{tabular}[c]{@{}l@{}}compactly determined \\ $k$-complete LCTVS\end{tabular}                                      & (see above)                                                                                                                                                                                                          \\ \midrule
$b\mathsf{LCTVS}$                     & bornological LCTVS                                                                                                      & \cite[p.~71]{schaefer1971topological}, \cite{blute2010convenient}                                                                                                                                                    \\ \midrule
$st\mathsf{CBS}$                      & \begin{tabular}[c]{@{}l@{}}separated, topological convex \\ bornological vector spaces (CBS)\end{tabular}               & \cite[p.~10]{blute2010convenient}                                                                                                                                                                                    \\ \midrule
$\mathsf{Conv}$                       & \begin{tabular}[c]{@{}l@{}}``convenient vector spaces'' \\ (Mackey-complete, separated,\\ topological CBS)\end{tabular} & \cite{kriegl1997convenient} , \cite[Definition 3.15]{blute2010convenient}                                                                                                                                            \\ \midrule
$\mathsf{Ste}$                        & stereotype spaces                                                                                                       & \cite[p.~477]{akbarov2022stereotype}                                                                                                                                                                                 \\ \bottomrule
\end{tabular}
\end{table}

\begin{table}[]
\vspace*{0.2 cm}
\caption{\label{functors_table}Overview of named functors}
\begin{tabular}{@{}lll@{}}
\toprule
\textit{Functor} & \textit{Name}                                                                                          & \textit{Reference}                        \\ \midrule
$k$              & $k$-ification                                                                                          & \cref{defn_k_ification}, \cref{defn_kloc} \\ \midrule
$l$              & associated LCTVS                                                                                       & \cref{defn_ksat_l}                        \\ \midrule
$r$              & repletion                                                                                              & \cref{defn_repletion}                     \\ \midrule
$b$              & \begin{tabular}[c]{@{}l@{}}bornologification/\\ associated bornological space\end{tabular}             & \cite[p.~63]{schaefer1971topological},    \\ \midrule
$\mathcal{M}_c$  & \begin{tabular}[c]{@{}l@{}}free paired linear $hk$-space/\\ space of comp. supp. measures\end{tabular} & \cref{defn_free_plin_ksp}                 \\ \bottomrule
\end{tabular}
\end{table}

\begin{table}[]
\vspace*{0.75 cm}
\caption{\label{inclusions_and_equivalences_table}Overview of inclusions and equivalences of categories}
\begin{tabular}{@{}lll@{}}
\toprule
\textit{Inclusion of...}         & \textit{...into...}                  & \textit{Reference}                           \\ \midrule
$b\mathsf{LCTVS}$                & $lk\mathsf{LCTVS}$                   & \cite[1.3]{porta1972compactly}               \\ \midrule
$hk\mathsf{Top}$                 & $\mathsf{Cond}(\mathsf{Set})_{q.s.}$ & \cite[Proposition 1.7]{scholze2019condensed} \\ \midrule
$\mathsf{Ste}$                   & $\mathsf{LCTVS}_{k\mathrm{c}}$       & \cite[Theorem 4.1.1]{akbarov2022stereotype}  \\ \midrule
\textit{Equivalence of...}       & \textit{...to...}                    &                                              \\ \midrule
$b\mathsf{LCTVS}$                & $st\mathsf{CBS}$               & \cite[Corollary 3.8]{blute2010convenient}    \\ \midrule
$r\mathsf{Vect}(hk\mathsf{Top})$ & $lk\mathsf{LCTVS}_{k\mathrm{c}}$     & \cref{cor_replete_equiv_lctvs}               \\ \bottomrule
\end{tabular}
\end{table}

\begin{table}[]
\vspace*{0.75 cm}
\caption{\label{adjunctions_table}Overview of adjunctions}
\begin{tabular}{@{}lll@{}}
\toprule
\textit{Adjunction between...}        & \textit{...and...}                    & \textit{Reference}                                                              \\ \midrule
$\mathsf{L}$                          & $r\mathsf{Vect}(hk\mathsf{Top})$      & \cref{sepobjs_corefl_subcat_of_little_chu}                                      \\ \midrule
$\mathsf{Vect}(hk\mathsf{Top})$       & $\mathsf{Cond}(\mathsf{Vect})_{q.s.}$ & \cite[Proposition 1.7]{scholze2019condensed}                                    \\ \midrule
$hk\mathsf{Top}$                      & $\mathsf{Cond}(\mathsf{Set})_{q.s.}$  & \cite[Proposition 1.7]{scholze2019condensed}                                    \\ \midrule
$\mathsf{Cond}(\mathsf{Set})_{q.s.}$  & $\mathsf{Cond}(\mathsf{Set})$         & \cite[Lemma 4.14]{sclausen2020lectures}                                         \\ \midrule
$\mathsf{Cond}(\mathsf{Vect})_{q.s.}$ & $\mathsf{Cond}(\mathsf{Vect})$        & \cite[Lemma 4.14]{sclausen2020lectures}                                         \\ \midrule
$\mathsf{Vect}(hk\mathsf{Top})$       & $lk\mathsf{LCTVS}$                    & \cref{kloc_conv_equiv_ksat_lctvs}                                               \\ \midrule
$\mathsf{LCTVS}_{k\mathrm{c}}$        & $\mathsf{LCTVS}$                      & \cite[Theorem 3.3.1]{akbarov2022stereotype}                                     \\ \midrule
$\mathsf{Ste}$                        & $\mathsf{LCTVS}_{k\mathrm{c}}$        & \cite[Theorem 3.3.15, Theorem 4.1.1]{akbarov2022stereotype}                     \\ \midrule
$st\mathsf{CBS}$                      & $\mathsf{Conv}$                       & \cite[Lemma 2.6.5]{frolicher1988linear}, \cite[Lemma 3.18]{blute2010convenient} \\ \bottomrule
\end{tabular}
\end{table}